\documentclass[leqno,11pt]{amsart}
\usepackage{amssymb, amsmath}
\numberwithin{equation}{section}

\def\inte#1{
\displaystyle\mathop{#1\kern0pt}^\circ }

\let\e=\varepsilon

\def\cC{{\mathcal C}}

\def\cG{{\mathcal G}}

\def\virgp{\raise 2pt\hbox{,}}
\def\cdotpv{\raise 2pt\hbox{;}}

\def\C{{\mathop{\bf C\kern 0pt}\nolimits}}
\def\DD{{\mathop{\bf D\kern 0pt}\nolimits}}
\def\K{{\mathop{\bf K\kern 0pt}\nolimits}}
\def\N{{\mathop{\mathbb N\kern 0pt}\nolimits}}
\def\PP{{\mathop{\mathbb P\kern 0pt}\nolimits}}
\def\Q{{\mathop{\mathbb Q\kern 0pt}\nolimits}}
\def\R{{\mathop{\mathbb R\kern 0pt}\nolimits}}
\def\SS{{\mathop{\mathbb S\kern 0pt}\nolimits}}
\def\ZZ{{\mathop{\mathbb Z\kern 0pt}\nolimits}}
\def\TT{{\mathop{\mathbb T\kern 0pt}\nolimits}}
\def\BB{{\mathop{\mathbb B\kern 0pt}\nolimits}}
\def\PP{{\mathop{\mathbb P\kern 0pt}\nolimits}}

\newcommand{\ds}{\displaystyle}

\newcommand{\beq}{\begin{equation}}
\newcommand{\eeq}{\end{equation}}
\newcommand{\ben}{\begin{eqnarray}}
\newcommand{\een}{\end{eqnarray}}
\newcommand{\beno}{\begin{eqnarray*}}
\newcommand{\eeno}{\end{eqnarray*}}

\newtheorem{defi}{Definition}[section]
\newtheorem{thm}{Theorem}
\newtheorem{assumption}{Assumption}
\newtheorem{lem}[defi]{Lemma}
\newtheorem{rmk}[defi]{Remark}
\newtheorem{cor}{Corollary}
\newtheorem{prop}[defi]{Proposition}

\begin{document}

  \title[Stability in weak topology of global solutions to the Navier-Stokes equations]{On the stability in   weak topology  of the set of global solutions to the Navier-Stokes equations}

\author[H. Bahouri]{Hajer Bahouri}
\address[H. Bahouri]%
{Laboratoire d'Analyse et de Math{\'e}matiques Appliqu{\'e}es UMR 8050 \\
Universit\'e Paris 12 \\
61, avenue du G{\'e}n{\'e}ral de Gaulle\\
94010 Cr{\'e}teil Cedex, France}
\email{hbahouri@math.cnrs.fr}
\author[I. Gallagher]{Isabelle Gallagher}
\address[I. Gallagher]%
{Institut de Math{\'e}matiques de Jussieu - Paris Rive Gauche
 UMR CNRS 7586 \\
      Universit{\'e} Paris-Diderot (Paris 7) \\
B\^atiment Sophie Germain \\
Case 7012 , 75205 Paris Cedex 13     \/France}
\email{gallagher@math.univ-paris-diderot.fr}
 \thanks{The second  author was partially supported by
the A.N.R grant ANR-08-BLAN-0301-01 "Mathoc\'ean", and  the Institut Universitaire de France.}
 \begin{abstract}
 Let~$X$ be a suitable function space and let~$\cG \subset X$ be the set of divergence free vector fields generating a global, smooth solution to the incompressible, homogeneous three dimensional Navier-Stokes equations. We prove that   a sequence of divergence free vector fields converging in the sense of distributions to an element of~$\cG$ belongs  to~$\cG$ if~$n$ is large enough, provided the convergence holds     ``anisotropically"   in frequency space. Typically that excludes self-similar type convergence.
 Anisotropy   appears as an important  qualitative feature in the analysis of the Navier-Stokes equations; it is also shown that   initial data  which does not belong to~$\cG$ (hence which produces a solution blowing up in finite time) cannot have a strong anisotropy in its   frequency support.
  \end{abstract}
\subjclass[2010]{}

 \maketitle

 %%%%%%%%%%%%%%%%%%%%%%%%%%%%%%%%%%%%%%%%%%%%%

 \section{Introduction and statement of   results}
\subsection{Setting of the problem}
We are interested in the three dimensional, incompressible Navier-Stokes equations
$$
\begin{cases}
\partial_t u+u\cdot\nabla u- \Delta u=-\nabla p \quad\text{in}\quad \R^+\times \R^3 \\
\mbox{div} \:  u=0\\
u_{|t=0}=u_{0} \, ,
\end{cases}
$$
where~$u(t,x)$ and~$p(t,x)$ are respectively the velocity and the pressure of the fluid at time~$t \geq 0$ and position~$x \in \R^3$. We recall that the pressure may be eliminated by projection onto divergence free vector fields, hence we shall consider   the following version of the equations:
$$
\rm{(NS)} \ \begin{cases}
\partial_t u+{\mathbb P} (u\cdot\nabla u)- \Delta u= 0 \quad\text{in}\quad \R^+\times \R^3 \\
\mbox{div} \:  u=0\\
u_{|t=0}=u_{0} \, ,\end{cases}
$$
where~${\mathbb P} :=  \mbox{Id} - \nabla \Delta^{-1} \mbox{div}$. 

\medskip
\noindent Note also that the Navier-Stokes system may be written as 
$$
\begin{cases}
\partial_t u+ \mbox{div} (u\otimes u) -\Delta u=-\nabla p \quad\text{in}\quad \R^+\times \R^3 \\
\mbox{div} \:  u=0\\
u_{|t=0}=u_{0} \, ,
\end{cases}
$$
where $\ds \mbox{div} (u\otimes u)^j= \sum_{k=1}^d\partial_{k}(u^ju^k) = \mbox{div}
(u^ju). $ The advantage of this weak formulation is that it makes sense for  singular vector fields  and allows to consider weak solutions. 
\medskip
\noindent The question of the existence of  global, smooth (and unique) solutions is a long-standing open problem, and we shall only recall here a few of the many results on this question. We refer for instance  to~\cite{BCD} or~\cite{lemarie} and the references therein, for  a precise definition of weak solutions and recent surveys on the subject. An important point in the study of (NS) is its scale invariance: if~$u$ is a solution of~(NS) on~$\R^+\times\R^d$ associated with the data~$u_0$, then for any~$\lambda>0$,~$u_\lambda(t,x):=\lambda u(\lambda^2 t,\lambda x)$ is a solution on~$\R^+\times\R^d$, associated with the data
\begin{equation} \label{scaling}
u_{0,\lambda} (x) := \lambda u_0(\lambda  x) \, .
\end{equation}
In two space dimensions, $L^2(\R^2)$ is scale invariant, while in three space dimensions that is the case for~$L^3(\R^3)$, the (smaller) Sobolev space~$\dot H^\frac12(\R^3)$, or the  Besov spaces~$\dot B^{-1+\frac3p}_{p,q}(\R^3)$, with~$1 \leq p \leq \infty$ and~$0<q \leq \infty$. We refer to Appendix~\ref{appendixlp} for all necessary definitions and properties of those spaces. Note that    anisotropic spaces  such as~$L^2(\R^2 ; \dot H^\frac12(\R))$ can also be scale invariant under~(\ref{scaling}), but also more generally under the anisotropic scaling
\begin{equation} \label{scalinganisotropic}
f_{ \lambda,\mu} (x)  :=\lambda f(\lambda x_1,\lambda x_2, \mu x_3 ) \, , \quad \forall \lambda,\mu >0 \, .
\end{equation}
Of course (NS) is not invariant through that transformation if~$\lambda \neq \mu$.

\noindent It is well-known that (NS) is globally wellposed if the initial data is small in~$\dot B^{-1+\frac3p}_{p,\infty}$ as long as~$p< \infty$ (see the successive results by~\cite{leray}, \cite{fk},     
\cite{kato}, \cite{Cannone} and~\cite{Planchon}). Let us emphasize  that in all those results, the global solution  lies  in~$\cC(\R^+;X)$ when the Cauchy data belongs to the Banach space $X$. We note that the proof of uniqueness may require the use of more refined spaces. In~\cite{kochtataru}, H. Koch and D. Tataru obtained a unique (in a space we shall not detail here) global in time
solution for data small enough in   the larger space~$ \text{BMO}^{-1}$, consisting of vector fields whose components are derivatives of $\text{BMO}$ 
functions.

\smallskip
\noindent
The smallness assumption is not   necessary in order to obtain global solutions to~(NS), as pointed out for instance in~\cite{cg2}-\cite{cgz}. We also recall that in two space dimensions,~(NS) is globally wellposed as soon as the initial data belongs to~$L^2(\R^2)$, with no restriction on its size (see~\cite{leray2D}); this is due to the fact that the~$L^2(\R^d)$ norm is controlled a priori globally in time. This estimate also allowed J. Leray in~\cite{leray} to prove the existence of global in time weak solutions in  two and three dimensions. J. Leray's result extends   to any dimension, as   shown in~\cite{cdgg} for instance.  

\bigskip
\noindent
In this article we are interested in the structure of the set~$\cG$ of initial data giving rise to a unique, global solution to the Navier-Stokes equations.  More precisely our interest will be in the {\bf global} nature of the solution, as  the uniqueness of the solution will  not be an issue. The solutions will be obtained via a fixed point procedure in an adequate function space.
  It is known that  the set~$\cG$ contains small balls  in~$ \text{BMO}^{-1}$ centered at the origin. But it is known to include many more classes of functions.  
We recall that it was  proved in~\cite{adt} (see~\cite{gip} for the setting of Besov spaces) that~$\cG$  is {\bf open} for the strong topology of~$\text{BMO}^{-1}$,  provided one restricts the setting to the closure of Schwartz-class functions for the~$\text{BMO}^{-1}$ norm. In this paper we address the  same question for   {\bf weak topology}. More precisely we wish to understand under what   conditions a sequence of divergence free vector fields  converging in the sense of distributions to an initial data in~$\cG$, will itself be in~$\cG$ (up to a finite number of terms in the sequence).

\bigskip
\noindent Before going into more details let us discuss some examples. If a sequence converges not only weakly but strongly  in~$\dot B^{-1+\frac3p}_{p,\infty}$, say,  to an element of~$\cG$ then the result is known, see~\cite{gip}. To give another example,    consider  a sequence of divergence free vector fields~$u_{0,n}$, bounded   in~$L^3(\R^3)$, converging  in the sense of distributions to some vector field~$u_0 $ in~$  L^3(\R^3) \cap \cG$. If~$(1+|\cdot|)^{1+\varepsilon} u_{0,n}$ is bounded in~$L^\infty$ for some~$\varepsilon>0$, then it is easy to see that~$u_{0,n}$ generates a global unique solution to~(NS) for~$n$ large enough. This can be seen  using the ``stability of singular points'' of~\cite{jiasverak,rs}, or   more directly using the fact that such a sequence is actually compact\footnote{This  fact can readily be seen by applying a profile decomposition technique and eliminating all profiles except for the weak limit, thanks to the additional bounds satisfied by the sequence.} in~$\dot B^{-1+\frac3p}_{p,\infty}$ for~$p>3$  and applying the strong stability result~\cite{gip}. This example shows that in some cases, the weak convergence assumption
implies  the strong convergence in spaces where stability results are available.
Here we consider a situation where such a reduction does not occur. One way to achieve this is   considering sequences bounded in a scale-invariant space only, with no 
additional bound in a non-scale-invariant space. However in that case clearly some restrictions have to be imposed to hope to prove such a weak openness result; indeed   consider  for instance the sequence
 \begin{equation}\label{defphinisotropic}
\phi_n (x) := 2^{n} \phi (2^{n} x)\,  ,\quad  n \in \N \, ,
\end{equation}
where~$\phi$ is any smooth, divergence free vector field. This sequence converges  to zero in the sense of distributions as~$n$ goes to infinity, and zero belongs to~$\cG$. If one could infer, by weak stability, that~$\phi_n$ gives rise to a global unique solution for large enough~$n $, then so would~$\phi$ by scale invariance and that  would solve the problem of global regularity for the Navier-Stokes equations. Note that the
same can be said of the sequence
 \begin{equation}\label{tildephin}
\widetilde \phi_n (x) := \phi (x-x_n) \, , \quad |x_n| \to \infty \, .
\end{equation}
Since the global regularity problem seems out of reach, we choose here to add   assumptions on the {\bf spectral structure} of the sequences converging weakly to an element of~$\cG$, which  in particular forbid   sequences such as~$\phi_n$ or~$\widetilde \phi_n$ which in a way are ``too isotropic".

\medskip
\noindent
Actually  one has the following  interesting and rather easy result, which highlights  the role   anisotropy can play in the study of the Navier-Stokes equations. This result shows that  initial data generating a solution blowing up in finite time  cannot be too anisotropic in frequency space, meaning that the set of its horizontal and vertical frequency sizes cannot  be too separated; the threshold depends only on the
  norm of the initial data. The result is proved in Appendix~\ref{appendixlp};  its proof relies on   elementary inequalities on the Littlewood-Paley decomposition, which are all recalled in that appendix. The notation~$ \Delta_k^h\Delta_j^v$ appearing in the statement stands for   horizontal and vertical Littlewood-Paley truncations at scale~$2^k$ and~$2^j$ respectively, and is also introduced in Appendix~\ref{appendixlp}. The space~$\dot B^\frac12_{2,1}(\R^3)$ is a scale invariant space, slightly smaller than~$\dot H^\frac12(\R^3)$.
 \begin{thm}\label{anisominimal}
 Let~$\rho >0$ be given. There  is a constant~$N_0 \in \N$  such that
  any    divergence free vector field~$u_0$  of norm~$\rho$  in~$\dot B^\frac12_{2,1}(\R^3)$ satisfying~$
  \displaystyle u_0 = \sum_{|j-k| \geq N_0} \Delta_k^h\Delta_j^v u_0$ gives rise to a global,  unique solution to~{\rm(NS)} in~${\mathcal C}(\R^+;L^3(\R^3))$.
      \end{thm}

\medskip
\noindent
Let us  now define the function spaces we shall be working with.
As explained above we want to work in anisotropic spaces, invariant through the scaling~(\ref{scalinganisotropic}). For technical reasons we shall assume quite a lot of smoothness on the sequence of initial data: we choose the sequence bounded in essentially the smallest anisotropic Besov space~$\dot B^{s,s'}_{p,q}$ invariant through~(\ref{scalinganisotropic}). It is likely that this smoothness could be relaxed somewhat, but perhaps not with the method we follow. We shall point out as we go along where those restrictions appear, see in particular Remark~\ref{p=q=1} page~\pageref{p=q=1}.

\begin{defi} We define, for~$0<q \leq \infty$, the space~$ {\mathcal B}^1_q$ by the (quasi)-norm
\begin{equation} \label{firstnorm}
\|f\|_{  {\mathcal B}^1_q }:=\Big( \sum_{j,k \in \ZZ} 2^  {(j+k)q} \|\Delta_k^h \Delta_j^v f\|^q_{L^1(\R^3)} \Big)^\frac1q  \, ,
\end{equation}
 where~$\Delta_k^h$ and~$\Delta_j^v$
are horizontal and vertical frequency localization operators (see Appendix~\ref{appendixlp}).
\end{defi}
\noindent This   corresponds to the space~$\dot B^{1,1}_{1,q}$ defined in Appendix~{\rm \ref{appendixlp}},  where the reader will also find its properties  used in this text.  More generally we define in Appendix~\ref{appendixlp}
$$ \|f\|_{  \dot B^{s,r}_{p,q} }:=\Big( \sum_{j,k \in \ZZ} 2^  {(rj+sk)q} \|\Delta_k^h \Delta_j^v f\|^q_{L^p(\R^3)} \Big)^\frac1q \, .
$$
The norm \eqref{firstnorm} is equivalent to the norm \eqref{defbesovanisoheat}  which is clearly invariant by the scaling~(\ref{scalinganisotropic}), and is slightly larger (if~$q \leq 1$) than the more classical~$\dot B^{0,\frac12}_{2,1}$ norm (for the role of~$\dot B^{0,\frac12}_{2,1}$ in the study of the Navier-Stokes equations see for instance~\cite{cheminzhang},\cite{paicu}).
  Moreover the space~$ {\mathcal B}^1_q $   is {\bf anisotropic} by essence, which   as pointed out above,  will be an important feature of our analysis. 
  
 \medskip \noindent It is proved in Appendix~\ref{globalsmallaniso} that any initial data small enough in~$ {\mathcal B}^1_1$ generates a unique, global solution to (NS) in the space~$ {\mathcal S}_{1,1}:=   \widetilde{L^\infty}(\R^+; {\mathcal B}^1_1)  \cap   {L^1}
         (\R^+;\dot B^{3 ,1}_{1,1}
          \cap \dot B^{ 1  ,3 }_{1,1} ) $, and  if the data is not small then there is a unique   solution in   the local space 
          $$
   {\mathcal S}_{1,1}(T):=  \widetilde{ L^\infty_{loc}}([0,T); {\mathcal B}^1_1)  \cap   L^1_{loc}
         ((0,T);\dot B^{3 ,1}_{1,1} \cap \dot B^{ 1  ,3 }_{1,1} ) 
         $$
         for some~$T>0$.
         
       \medskip \noindent     We provide also in  Appendix~\ref{globalsmallaniso} a {\bf strong} stability result in~$  {\mathcal B}^1_1$, whose proof follows a classical procedure, and the main
 goal of this text is to prove a stability result in the {\bf weak} topology for data in~$  {\mathcal B}^1_q$ for~$0 <q<1$.
  
  \medskip
\noindent Now let us define  our notion of an anisotropically oscillating sequence. We shall  need another more technical assumption later, which is stated in Section~\ref{decompositiondata} (see Assumption~\ref{sousdecompo} page~\pageref{sousdecompo}).
        \begin{assumption}\label{defadf}
        Let~$0<q \leq \infty$ be given. We say that a sequence~$(f_n)_{n \in \N}$, bounded in~$ {\mathcal B}^1_q$, is
        \textit{\textbf{anisotropically oscillating}‚Äâ}  if the following property holds. There exists~$p\geq 2$ such that
 for all sequences~$ (k_n,j_n) $ in~$   \ZZ^\N \times  \ZZ^\N$,
 \begin{equation}\label{frequencyextraction}
   \liminf_{n \to \infty}  \, 2^{k_n(-1+\frac2p)+\frac{j_n}p} \|\Delta_{k_n}^{h} \Delta_{j_n}^v f_n \|_{L^p(\R^3)}   = C >0 \quad
   \Longrightarrow  \:  \lim_{n \to \infty}  \:|j_n - k_n| = \infty \, .
\end{equation}

  \end{assumption}

  \begin{rmk}{\rm \label{rmkdifferentscales}
It is   easy to see (see Appendix~\ref{appendixlp}) that any function~$f$ in~$ {\mathcal B}^1_q $ belongs also to~$\dot B^{-1+\frac2p,\frac1p}_{p,\infty} $ for any~$p \geq 1$ hence
  $$
  f \in {\mathcal B}^1_q   \Longrightarrow  \sup_{(k,j) \in \ZZ^2}2^{k(-1+\frac2p)+\frac jp} \|\Delta_{k}^{h}   \Delta_{j}^v f \|_{L^p}   < \infty  \, .
  $$
The left-hand side of~(\ref{frequencyextraction}) indicates which ranges of frequencies are predominant in the sequence~$(f_n)$: if~$\displaystyle\limsup_{n \to \infty} 2^{k_n(-1+\frac2p)+\frac{j_n}p} \|\Delta_{k_n}^{h} \Delta_{j_n}^v f_n \|_{L^p} $  is zero for a couple of frequencies~$(2^{k_n}, 2^{j_n})$, then
 the sequence~$(f_n)_{n \in \N}$ is ``unrelated" to those frequencies, with the vocabulary of~\cite{G} (see also Lemma~\ref{orthoaniso} in this paper).
 The right--hand side of~(\ref{frequencyextraction})  is then an   {\textbf{anisotropy}} property. Indeed
     one sees easily that a sequence such as~$(\phi_n)_{n \in \N}$ defined in~{\rm (\ref{defphinisotropic})} is precisely not anisotropically oscillating: for the left-hand side of~{\rm (\ref{frequencyextraction})} to hold for~$\phi_n$ one  would need~$j_n \sim k_n \sim n $, which is precisely not the condition required on the right-hand side of~(\ref{frequencyextraction}).
  A typical   sequence satisfying Assumption~\ref{defadf} is  rather (for~$a \in \R^3$)
      $$
   f_n(x) := 2^{\alpha n} f\big(2^{\alpha n} (x_1-a_1),2^{\alpha n}  (x_2-a_2),2^{\beta n}  (x_3-a_3) \big), \quad
   (\alpha, \beta) \in \R^2, \quad\alpha \neq \beta          
   $$
  with~$ f$ smooth. One of the   results of  this paper states that   any sequence   satisfying  Assumption~{\rm \ref{defadf}}    may be written as  the superposition of   such   sequences, up to a small remainder term (see Proposition~\ref{prop:decompositiondata} page~\pageref{prop:decompositiondata}).}
     \end{rmk}

    \subsection{Main results}
  \noindent We prove in this article that~$\cG$ is open for   weak topology, provided the weakly converging sequence is of the type described in Assumption~{\rm \ref{defadf}}.
 \begin{thm}\label{mainresult}
         Let~$q \in ]0,1[$ be given and let~$(u_{0,n})_{n \in \N}$ be a   sequence of divergence free vector fields bounded in~${\mathcal B}^1_q$, converging   towards~$u_0 \in {\mathcal B}^1_q$ in the sense of distributions, and  assume that~$u_0$ generates a unique solution in~$  {\mathcal S}_{1,1}(\infty)$.
  If~$u_0-(u_{0,n} )_{n \in \N}$ is anisotropically oscillating and satisfies Assumption~{\rm \ref{sousdecompo}} page~{\rm \pageref{sousdecompo}}, then for~$n$ large enough,~$u_{0,n}$ generates a unique, global solution to {\rm(NS)} in~${\mathcal S}_{1,1}(\infty)$.
 \end{thm}
 
 \begin{rmk}
{\rm Theorem~{\rm \ref{mainresult}} may be generalized  by adding two more sequences to~$u_{0,n}$,   where in each additional sequence the ``privileged"   direction is not~$x_3$ but~$x_1$ or~$x_2$. It is clear from the proof that the same result holds, but we choose not to present the proof of that more general result due to its
technical cost. Actually a more interesting generalization would consist in considering more geometrical assumptions, but that requires  more work and ideas, and will not be addressed here.}
\end{rmk}
 
 \begin{rmk}{\rm
  Assumption~\ref{sousdecompo} is stated page~\pageref{sousdecompo}, along with some comments (see in particular Remarks~\ref{rmkass2},~\ref{ifalsoell2} and~\ref{ifalsoell22}).
  Its statement requires the introduction of the profile decomposition of the sequence of initial data and it requires    that some of the profiles vanish at zero. 
  }\end{rmk}

\begin{rmk}{\rm
Theorem~\ref{mainresult} generalizes the result of~\cite{cgz}, where it is shown that   the  initial data 
$$
u_0(x)+ \sum_{j= 1}^J (v^{1 (j)}_0 + \e_j w_0^{1 (j)} , v^{2 (j)}_0 + \e_j w_0^{2 (j)} ,w_0^{3
(j)})(x_1,x_2,\e_j x_3)
$$
generates a global solution if~$u_0 $ belongs to~$\dot H^\frac12(\R^3) \cap  \cG$, if  the profiles~$(v^{1 (j)}_0,v^{2 (j)}_0,0)$ and~$w^{(j)}_0$  are divergence free and in~$L^2(\R_{x_3}; \dot H^{-1}(\R^2))$, as well as all their  derivatives,    if~$\e_1,\dots \varepsilon_J >0$ are small enough, and finally under the assumption that~$v^{1 (j)}_0 (x_1,x_2,0) \equiv v^{2 (j)}_0 (x_1,x_2,0) \equiv 0$ and~$w_0^{3
(j)}(x_1,x_2,0) \equiv 0$. Those last requirements  are analogous to Assumption~{\rm \ref{sousdecompo}}. Note that even in the case when~$u_0 \equiv 0$, such initial data cannot be dealt with simply using Theorem~\ref{anisominimal} since it is not bounded in~$\dot B^\frac12_{2,1}$.
Note also that as in~\cite{cgz}, the special structure of (NS) is used in the proof of Theorem~\ref{mainresult}.}
\end{rmk}

\begin{rmk}\label{globalintime}{\rm
 Notice that it is not assumed that the  global solution associated with~$u_0$ satisfies   uniform, global in time integral bounds. Similarly to~\cite{adt} and~\cite{gip} such  bounds may be derived a posteriori from the fact that the solution is global: see Appendix~\ref{globalsmallaniso}, Corollary~\ref{strongstability}.
     }\end{rmk}

\begin{rmk}{\rm
One can see from the proof of Theorem~\ref{mainresult} that the solution~$u_n(t)$ associated with~$u_{0,n}$ converges for all times, in the sense of distributions to the solution associated with~$u_0$.  In this sense the Navier-Stokes equations are stable by weak convergence.}
\end{rmk}

\noindent
The proof of  Theorem~{\rm \ref{mainresult}}  enables us to infer   easily the following results. The first corollary generalizes the statement of Theorem~{\rm \ref{mainresult}}  to the case when~$u_0 \notin \cG$.
\begin{cor}\label{corthm1}
  Let~$(u_{0,n})_{n \in \N}$ be a   sequence of divergence free vector fields bounded in the space~${\mathcal B}^1_q$ for some~$0<q<1$, converging   towards some~$u_0\in {\mathcal B}^1_q$ in the sense of distributions, with~$u_0-(u_{0,n} )_{n \in \N}$   anisotropically oscillating and satisfying Assumption~{\rm\ref{sousdecompo}}. Let~$u$ be the solution to the Navier-Stokes equations associated with~$u_0$ and assume that the life span of~$u$ is~$T^*< \infty$. Then for all~$T < T^*$, there is a subsequence such that  the life span
 of the solution associated with~$u_{0,n}$ is at least~$T$.
   \end{cor}

\smallskip
\noindent    The second corollary deals with the case when the {\bf sequence} belongs to~$\cG$, with an a priori boundedness assumption on the solution (which could actually be generalized but we choose not to complicate things too much at this stage; see Appendix~\ref{appendixlp}
 for  definitions), and infers that the weak limit also belongs to~$\cG$.
\begin{cor}\label{corthm}
Assume~$(u_n^0)_{n \in \N}$ is a sequence of initial data, such that the associate solution~$u_n$ is uniformly bounded in~$\widetilde{L^2}\big (\R^+; \dot B^{\frac23, \frac13}_{3,1}  \big) $.  If~$u_n^0$ converges in the sense of distributions to some~$u_0$, with~$u_0-(u_{0,n} )_{n \in \N}$   anisotropically oscillating and satisfying Assumption~{\rm\ref{sousdecompo}}, then~$u_0$ gives rise to a unique, global solution in~${\mathcal S}_{3,1}(\infty)$.
\end{cor}

 \subsection{Notation}
  For all points~$x = (x_1,x_2,x_3) $ in~$\R^3$ and all   vector fields~$v = (v^1,v^2,v^3)$, we shall denote by
  $$
  x_h:= (x_1,x_2) \quad \mbox{and} \quad v^h := (v^1,v^2)
  $$
   their horizontal parts. We shall also define   horizontal differentiation operators~$\nabla^h := (\partial_1, \partial_2)$ and~$\mbox{div}_h := \nabla^h \cdot  $, as well as~$\Delta_h:=\partial_1^2 + \partial_2^2$.

\smallskip \noindent We shall also use the shorthand notation for function spaces~$X$ (defined on~$\R^2$) and~$Y$ (defined on~$\R$): $X_h Y_v:=X(\R^2;Y(\R))$.

\smallskip \noindent Finally we shall denote by~$C$ a constant which does not depend on the various parameters appearing in this paper, and which may change from line to line. We shall also denote sometimes~$x \leq Cy$ by~$x \lesssim y$.

 %%%
   \subsection{General scheme of the proof and organization of the paper}
   The main arguments leading to Theorem~\ref{mainresult} are the following: by a profile decomposition argument,  the sequence of initial data   is decomposed into the sum of the weak limit~$u_0$ and a sequence of ``orthogonal" profiles, up to a small remainder term. 
 Under Assumptions~\ref{defadf} and~\ref {sousdecompo} and using scaling arguments it is   proved that each individual profile belongs to~${\mathcal G}$; this step relies crucially on the results of~\cite{cg3} and~\cite{cgz}. The mutual orthogonality of each term in the decomposition of the initial data implies finally that the sum of the solutions associated to each profile is itself an approximate solution to (NS), globally in time, which concludes the proof.   
  \medskip
  
  \noindent  The paper is organized
  in the following way:
  \begin{itemize}
  \item In Section~{\rm\ref{decompositiondata}} we provide an ``anisotropic profile decomposition" of the sequence of initial data, based on a general result, Theorem~\ref{anisocompthm} stated and proved in Section~\ref{profile} page~\pageref{anisocompthm}. This enables us to replace the sequence of initial data, up to   an arbitrarily small remainder term, by a finite (but large) sum of   profiles.

\smallskip

\item Section~\ref{globalprofile} is then devoted to the
construction of an approximate solution by propagating globally in time  each individual profile of the decomposition. The propagation is through either the Navier-Stokes flow or transport-diffusion equations.

 \smallskip

\item In Section~{\rm \ref{superpositionglobal}}
we prove
 that the construction of the previous step does provide an approximate solution to the Navier-Stokes equations, thus completing the proof of Theorem~{\rm \ref{mainresult}}, while Corollaries~{\rm \ref{corthm1}} and~{\rm \ref{corthm}} are proved at the end of Section~{\rm \ref{superpositionglobal}}. That section is the most technical part of the proof, as one must check
 that the nonlinear interactions of all the functions constructed in the previous step are negligible. It also relies on results proved in Appendix~{\rm \ref{globalsmallaniso}}, on the  global regularity for the Navier-Stokes equation (and perturbed versions of that equation) for  small data and forces in anisotropic Besov spaces.

\smallskip

\item Finally in Appendix~{\rm \ref{appendixlp}} we collect   useful results   on isotropic and anisotropic spaces which are used in this text, and we prove Theorem~\ref{anisominimal}.

 \end{itemize}
%%%%%%%%%%%%%%%%%%%%%%%%%%%%%%%%%%%%%%%%%%%%%

 \section{Profile decomposition of the initial data}\label{decompositiondata}
 In this section we
consider a sequence of initial data
as given in Theorem~{\rm \ref{mainresult}}, and write down an anisotropic profile decomposition for that sequence.
We shall   constantly be using the following  scaling operators.
  \begin{defi}\label{notationprofiles} For any two sequences~$\boldsymbol{\varepsilon}=(\varepsilon_n)_{n \in \N}$ and~$\boldsymbol{\gamma}=(\gamma_n)_{n \in \N}$  of positive real numbers and any sequence~$\boldsymbol{x}=(x_n)_{n \in \N}$  in~$\R^3$ we define the scaling operator
  $$
  \Lambda^n_{\boldsymbol{\varepsilon} ,\boldsymbol{\gamma},\boldsymbol{x} } \phi (x) :=
  \frac1 {\varepsilon_n}  \phi\left( \frac{x_h-x_{n,h} }{\varepsilon_n}, \frac{x_3-x_{n,3}} {\gamma_n} \right) \, .
  $$
       \end{defi}
       \begin{rmk}\label{isometry}{\rm
The operator~$  \Lambda^n_{\boldsymbol{\varepsilon} ,\boldsymbol{\gamma},\boldsymbol{x} } $   is an isometry in the space~$\dot B^{-1+\frac2p,\frac1p}_{p,q}$ for any~$1 \leq p \leq \infty$ and~$0<q \leq \infty$.}
\end{rmk}

\noindent Then we define the notion of orthogonal cores/scales as follows (see also Section~\ref{profile}).
\begin{defi}\label{orthoseq}
We say that two triplets of sequences~$(\boldsymbol{\varepsilon^\ell} ,\boldsymbol{\gamma^\ell},\boldsymbol{x^\ell}  )$ for~$\ell \in \{1,2\}$, where~$(\boldsymbol{\varepsilon^\ell} ,\boldsymbol{\gamma^\ell})$ are two sequences  of positive real numbers and $\boldsymbol{x^\ell} $  are  sequences  in~$ \R^3$,   are   \textit{\textbf{orthogonal}} if
$$
  \begin{aligned}
 \mbox{either} \quad \frac{\varepsilon_n^{1}}{\varepsilon_n^{2}} +  \frac{\varepsilon_n^{2}}{\varepsilon_n^{1}}
  +  \frac{\gamma_n^{1}}{\gamma_n^{2}} +  \frac{\gamma_n^{2}}{\gamma_n^{1}}
  \to \infty\,
, \quad n \to \infty\,
  \\
  \mbox{or}\quad  (\varepsilon_n^{1} , \gamma_n^1)= (\varepsilon_n^{2}, \gamma_n^2)\quad \mbox{and} \quad
  | (x_n^1)^{\varepsilon^1,\gamma^1} - (x_n^2)^{\varepsilon^1,\gamma^1} |   \to \infty\,
, \quad n \to \infty \, ,
\end{aligned}
  $$
  where we have denoted
    $\displaystyle
 ( \boldsymbol{x^\ell} )^{\boldsymbol{\varepsilon^k} ,\boldsymbol{\gamma^k}} := \big(
 \frac{ \boldsymbol{x^{\ell}_h}}{\boldsymbol{\varepsilon^k} } ,
  \frac{ \boldsymbol{x^{\ell}_3}}{\boldsymbol{\gamma^k} }
 \big).
  $
  \end{defi}
\noindent   Note that up to extracting a subsequence, any sequence of positive real numbers can be assumed to converge either to~0, to $\infty$, or to a constant. In the rest of this paper, up to rescaling the profiles by a fixed constant, we shall assume that if the limit of any one of the sequences~$\boldsymbol{\varepsilon^\ell} ,\boldsymbol{\gamma^\ell}, \boldsymbol{\eta^\ell} ,\boldsymbol{\delta^\ell} $ is a constant, then it is one.

 \medskip
\noindent The main result of this section is the following.
\begin{prop}\label{prop:decompositiondata}
Under the assumptions of Theorem~{\rm\ref{mainresult}}, the following holds.  Let~$2 \leq  p \leq \infty$ be given  as in Assumption~\ref{defadf}.
For all integers~$\ell $ there are two sets of   orthogonal sequences in the sense of Definition~{\rm \ref{orthoseq}}, $(\boldsymbol{\varepsilon^\ell} ,\boldsymbol{\gamma^\ell},\boldsymbol{x^\ell})$ and~$(\boldsymbol{\eta^\ell} ,\boldsymbol{\delta^\ell} ,\boldsymbol{\tilde x^\ell})$ and for all~$\alpha \in (0,1)$
 there are  arbitrarily smooth divergence free vector fields~$( \tilde  \phi_\alpha^{h,\ell},0) $ and~$(- \nabla_h \Delta_h^{-1} \partial_3  \phi_\alpha^{\ell}  ,\phi_\alpha^{\ell})  $  such that up to extracting a subsequence,   one can write
$$
\begin{aligned}
u_{0,n}   =  u_0  & +
 \sum_{\ell = 1}^{\tilde L}
  \Lambda^n_{\boldsymbol{\eta^{\ell}} ,\boldsymbol{\delta^{\ell}},\boldsymbol{\tilde x^{\ell}} }
\left(\tilde  \phi_\alpha^{h,\ell} +\tilde  r_\alpha^{h,\ell} ,0\right) +  \sum_{\ell  = 1}^L  \Lambda^n_{\boldsymbol{\varepsilon^{\ell}} ,\boldsymbol{\gamma^{\ell}},\boldsymbol{x^{\ell}} } \left( -\frac{ \varepsilon_n^{\ell} }{\gamma_n^{\ell} }
 \nabla_h \Delta_h^{-1} \partial_3 ( \phi_\alpha^{\ell} + r_\alpha^{\ell}   )  ,  \phi_\alpha^{\ell}+ r_\alpha^{\ell}  \right)\\
& +( \tilde \psi_n^{h,\tilde L}  -\nabla_h \Delta_h^{-1} \partial_3  \psi_n^L  ,\psi_n^L ) \, , \quad
\mbox{\rm div} \: \tilde  r_\alpha^{h,\ell} = 0\, , \quad \| \tilde  r_\alpha^{h,\ell}\|_{ {\mathcal B}^{1}_q} +\|    r_\alpha^{\ell}\|_{ {\mathcal B}^{1}_q} \leq \alpha \, ,
  \end{aligned}
 $$
 with~$\tilde \psi_n^{h,\tilde L}$ and~$\psi_n^L$ independent of~$\alpha$ and uniformly bounded (in~$n$ and~$L$) in~${\mathcal B}^{1}_q$, and
 \begin{equation}\label{smallremainderpsi}
\limsup_{n \to \infty} \Big(   \|\tilde \psi_n^{h,  L}\|_{\dot B^{-1+\frac2p, \frac1p}_{p,1}  } + \| \psi_n^{L}\|_{\dot B^{-1+\frac2p, \frac1p}_{p,1}  } \Big)\to 0\, , \quad L \to \infty \, .
\end{equation}
Moreover the following properties hold:
\begin{equation}\label{anisoscales}
\forall \ell \in \N, \quad   \displaystyle  \displaystyle \lim_{n \to \infty}  \:(\delta_n^\ell)^{-1} \eta_n^{\ell}  \in \{0 , \infty\}\,  , \quad\lim_{n \to \infty}  \: (\gamma_n^\ell)^{-1} \varepsilon_n^{\ell}  = 0 \, ,
  \end{equation}
\medskip
\noindent
as well as the following stability result:
    \begin{equation}\label{orthonorms}
   \sum_{\ell \in \N} \big(\|\tilde  \phi_\alpha^{h,\ell} \|_{{\mathcal B}^{1}_q} +   \|   \phi_\alpha^{\ell}  \|_{{\mathcal B}^{1}_q}  \big) \lesssim \sup_n \|u_{0,n}\|_{{\mathcal B}^{1}_q} + \|u_0\|_{{\mathcal B}^{1}_q} \, .
  \end{equation}
\end{prop}
\begin{proof}[Proof of Proposition~{\rm \ref{prop:decompositiondata}}]
The proof is divided into two steps. First we decompose the third component~$ u_{0,n}^3$ according to Theorem~\ref{anisocompthm} in Section~{\rm \ref{profile}}, and then we decompose the horizontal component~$ u_{0,n}^h$  using both the first step and Theorem~\ref{anisocompthm}  again (for the divergence free part of~$u_{0,n}^h$).

\medskip
\noindent
{\bf Step 1. Decomposition of~$ u_{0,n}^3$. }
Let us apply Theorem~{\rm \ref{anisocompthm}} of Section~{\rm \ref{profile}} (see page~\pageref{anisocompthm}) to the sequence~$u_{0,n}^3 - u_0^3$. With the notation of Theorem~{\rm \ref{anisocompthm}}, we define
$$
\begin{aligned}
 \varepsilon_n^\ell & := 2^{-j_1 (\lambda_\ell (n))} \\
 \gamma_n^\ell & := 2^{-j_2 (\lambda_\ell (n))} \\
 x_{n,h}^{\ell}& := 2^{-j_1 (\lambda_\ell (n))} k_1(\lambda_\ell(n)) \\
  x_{n,3}^{\ell}& := 2^{-j_2 (\lambda_\ell (n))} k_2(\lambda_\ell(n)) \, .
\end{aligned}
$$
The orthogonality of the sequences $(\boldsymbol{\varepsilon^\ell} ,\boldsymbol{\gamma^\ell},\boldsymbol{x^\ell})$, as given in Definition~{\rm \ref{orthoseq}}, is
a consequence of the orthogonality property stated in Theorem~{\rm \ref{anisocompthm}} along with Remark~\ref{equalscales}. According to that theorem we can write
\begin{equation}\label{decompositionthirdcomponent}
u_{0,n}^3  - u_0^3  =
   \sum_{\ell  = 1}^L  \Lambda^n_{\boldsymbol{\varepsilon^{\ell}} ,\boldsymbol{\gamma^{\ell}},\boldsymbol{x^{\ell}} }    \varphi^{\ell}
 +    \psi_n^L   \, ,
\end{equation}
where due to~(\ref{ortogonal}) in Theorem~{\rm \ref{anisocompthm}},
$$
\sum_{\ell \in \N}   \|  \varphi^{\ell}   \|_{{\mathcal B}^{1}_q}  \lesssim \sup_n \|u_{0,n}^3 - u_0^3\|_{{\mathcal B}^{1}_q}   < \infty \, .
$$
In particular~$\psi_n^{L}$ is uniformly bounded (in~$n$ and~$L$) in~${\mathcal B}^{1}_q \subset \dot B^{-1+\frac2p, \frac1p}_{p,q}$, and   Theorem~{\rm \ref{anisocompthm}} gives
$$
 \limsup_{n \to \infty}  \|\psi_n^{L}\|_{\dot B^{-1+\frac2p, \frac1p}_{p,p}  } \to 0 \, , \quad L \to \infty \, .
$$
The result~(\ref{smallremainderpsi}) then follows by H\"older's inequality for sequences.
Note that we have used here the fact that~$q<1$.

\smallskip
\noindent Using horizontal and vertical frequency truncations, given any~$\alpha>0$ we may further decompose~$  \varphi^{\ell} $ into
\begin{equation}\label{rlalpha}
  \varphi^{\ell} =   \phi_\alpha^{\ell}  + r_\alpha^{\ell} \, , \quad   \mbox{ with } \phi_\alpha^{\ell} \, \,  \mbox{arbitrarily  smooth  and}\, \, \|r_\alpha^{\ell}\|_{{\mathcal B}^{1}_q} \leq \alpha \, ,
\end{equation}
and we have, by this choice of regularization,
$$
\|   \phi_\alpha^{\ell} \|_{{\mathcal B}^{1}_q} +\|   r_\alpha^{\ell} \|_{{\mathcal B}^{1}_q}  \leq 2 \| \varphi^{\ell} \|_{{\mathcal B}^{1}_q} \, .
$$
This implies~(\ref{orthonorms}) for~$ \phi_\alpha^{\ell} $.

\medskip
\noindent
Now let us prove that
$$
\forall \ell \in \N, \quad   \displaystyle \lim_{n \to \infty}  \: (\gamma_n^\ell)^{-1} \varepsilon_n^{\ell}  = 0 \, .
 $$
Assumption~\ref{defadf}
along with Lemma {\rm \ref{orthoaniso}} page~\pageref{orthoaniso} imply that the limit of~$\displaystyle
  (\gamma_n^\ell)^{-1} \varepsilon_n^{\ell}$ belongs to~$ \{0,\infty\}$. Moreover by the divergence free condition on~$u_{0,n}$ we have~$\mbox{div}_h \:  u_{0,n}^h = - \partial_3  u_{0,n}^3$
and since~$u_{0,n}^h$ is bounded in~${\mathcal B}^{1}_q$ we infer that~$ \partial_3  u_{0,n}^3$ is bounded in~$
\dot B^{0,1}_{1,q}$ and~$ \partial_3  u_{0}^3$ also belongs to~$
\dot B^{0,1}_{1,q}$. This in turn, due to Lemma~{\rm \ref{orthoanisodiv1}},  implies that
$$
\lim_{n \to \infty}  \: (\gamma_n^\ell)^{-1} \varepsilon_n^{\ell}  = 0 \, .
$$

 \noindent
{\bf Step 2. Decomposition of~$ u_{0,n}^h$. }
 The  divergence free assumption on the
initial data enables us to recover from the previous step a profile decomposition for~$ u_{0,n}^h$.
Indeed
 there is a two-dimensional, divergence free vector field~$\nabla^\perp_h C_{0,n}$
such that
$$
u_{0,n}^h   = \nabla^\perp_h C_{0,n}  - \nabla_h \Delta_h^{-1} \partial_3  u_{0,n}^3 \, ,
$$
where~$\nabla^\perp_h = (-\partial_1, \partial_2)$. Similarly there is some function~$\varphi$ such that
$$
u_{0}^h   = \nabla^\perp_h \varphi- \nabla_h \Delta_h^{-1} \partial_3  u_{0}^3 \, .
$$
Furthermore as recalled in the previous step~$ \partial_3  u_{0,n}^3$ is bounded in~$
\dot B^{0,1}_{1,q}$. This implies that the sequence~$\nabla^\perp_h C_{0,n} $ is bounded in~${\mathcal B}^{1}_q$ and arguing similarly for~$ \nabla^\perp_h \varphi$,  the profile decomposition of  Section~{\rm \ref{profile}}  may also be applied to~$\nabla^\perp_h C_{0,n}(x) -  \nabla^\perp_h \varphi$: we get
$$
\nabla^\perp_h C_{0,n}   -  \nabla^\perp_h \varphi  =
 \sum_{\ell  = 1}^L   \Lambda^n_{\boldsymbol{\eta^{\ell}} ,\boldsymbol{\delta^{\ell}},\boldsymbol{\tilde x^{\ell}} }
\tilde  \phi^{h,\ell}   + \tilde \psi_n^{h,L}
$$
with~$\displaystyle \limsup_{n \to \infty} \|\tilde \psi_n^L\|_{\dot B^{-1+\frac2p,\frac1p}_{p,p} } \to 0 $ as $  L \to \infty$  and~$\mbox{div}_h \:  \tilde \phi^{h,\ell} =0$ thanks to Lemma {\rm \ref{orthoanisodiv2}}. Finally~$\eta_n^\ell /\delta_n^\ell \to 0 $ or~$\infty$ due to the anisotropy assumption as in the previous step. The rest of the construction is   identical to Step~1. The proposition is proved.
\end{proof}

\noindent Before evolving in time the decomposition provided in Proposition~{\rm \ref{prop:decompositiondata}} we notice that it may happen that the cores and scales of concentration~$(\boldsymbol{\eta^\ell} ,\boldsymbol{\delta^\ell} ,\boldsymbol{\tilde x^\ell})$ 
 appearing in the decomposition of~$\nabla_h^\perp C_{0,n}$ coincide with (or more generally are non orthogonal to)  those appearing in the decomposition of~$u_{0,n}^3$, namely~$(\boldsymbol{\varepsilon^\ell} ,\boldsymbol{\gamma^\ell},\boldsymbol{x^\ell})$. In that case the corresponding profiles should be evolved together in time. 
 This leads naturally to the next definition.
 \begin{defi}\label{defkappaell}
 For each~$\ell \in \N$, we define~$\kappa(\ell)$    by  the condition (with the notation of Definition~{\rm \ref{orthoseq}})
\begin{equation}\label{existencekappa}
\lim_{n \to \infty}\Big (\frac{\boldsymbol{\varepsilon^{\kappa (\ell)}}}{\boldsymbol{\eta^{\ell}} } ,
\frac{\boldsymbol{\gamma^{\kappa (\ell)}}}{\boldsymbol{\delta^{\ell}}}
,\boldsymbol{ ( x^{\kappa (\ell)}}   - \boldsymbol{ {\tilde x^{\ell}})^{\varepsilon^{\kappa (\ell)}, \gamma^{\kappa (\ell)} }} \Big) = (\lambda_1,\lambda_2,a ) \, ,  \quad \lambda_1,\lambda_2>0 \, ,\, a \in \R^3 \, .
\end{equation}
We also define for each~$L \in \N$ the set
 $$
\begin{aligned}
 {\mathcal K}(  L) := \Big\{\ell \in \N \, / \,   \ell = \kappa(\tilde\ell)\,, \, \, \tilde\ell \in \{1, \dots ,   L\} \Big\}\, .
\end{aligned}
 $$
  \end{defi}
\begin{rmk}{\rm 
  Note that for each~$\ell$ there is at most one such~$\kappa (\ell) $  by orthogonality. Moreover up to rescaling-translating the profiles we can assume   that~$\lambda_1=\lambda_2=1$ and~$a = 0$.
  } 
\end{rmk}
 \noindent 
   The decomposition of Proposition~{\rm \ref{prop:decompositiondata}} can now be written, for any~$ L\in \N$ in the following way. The interest of the next formulation is that as  we shall see,   each
   profile   is either small, or orthogonal to all the others.
  In the next formula we decide, to simplify notation that the   profile~$ \phi_\alpha^{\kappa (\ell)} $ is   equal to zero if~(\ref{existencekappa}) does not hold. We also have changed slightly the
    remainder terms~$r_\alpha^\ell$ and~$\psi_n^L$,  without altering their smallness properties (and keeping their notation for simplicity), due to the fact that in Definition~\ref{defkappaell} the ratios converge to a fixed limit but are in fact  not strictly equal to the limit.
So we write
    \begin{eqnarray}\label{newformulationdecompositiondata}
u_{0,n}  & =  &u_0   +
 \sum_{\ell  = 1}^{   L}    \Lambda^n_{\boldsymbol{\eta^{\ell}} ,\boldsymbol{\delta^{\ell}},\boldsymbol{\tilde x^{\ell}} } \left(\tilde  \phi_\alpha^{h,\ell} +\tilde  r_\alpha^{h,\ell}
  -\frac{ \eta_n^{\ell} }{\delta_n^{\ell} }
 \nabla_h \Delta_h^{-1} \partial_3 ( \phi_\alpha^{\kappa(\ell)} + r_\alpha^{\kappa(\ell)})
  ,  \phi_\alpha^{\kappa(\ell)} + r_\alpha^{\kappa(\ell)}\right) \nonumber
\\
& \quad  &{}+      \sum_{
\kappa( \ell ) =1 \atop
 \kappa( \ell ) \in  {\mathcal K}(  \infty) \setminus     {\mathcal K}(  L)  }^L   \Lambda^n_{\boldsymbol{\eta^{\ell}} ,\boldsymbol{\delta^{\ell}},\boldsymbol{\tilde x^{\ell}} }  \left(\tilde  \phi_\alpha^{h,\ell} +\tilde  r_\alpha^{h,\ell}-\frac{ \eta_n^{\ell} }{\delta_n^{\ell} }
 \nabla_h \Delta_h^{-1} \partial_3 ( \phi_\alpha^{\kappa(\ell)} + r_\alpha^{\kappa(\ell)})     , \phi_\alpha^{\kappa(\ell)} + r_\alpha^{\kappa(\ell)}\right)
\nonumber \\
&  \quad &{}+  \sum_{
   \ell = 1\atop
 \ell \notin  {\mathcal K}(  \infty)  }^{L } \Lambda^n_{\boldsymbol{\varepsilon^{\ell}} ,\boldsymbol{\gamma^{\ell}},\boldsymbol{x^{\ell}} } \left( -\frac{ \varepsilon_n^{\ell} }{\gamma_n^{\ell} }
 \nabla_h \Delta_h^{-1} \partial_3 ( \phi_\alpha^{\ell} + r_\alpha^{\ell}   )  ,  \phi_\alpha^{\ell}+ r_\alpha^{\ell}  \right)  \\
&  \quad & {}-    \!  \!  \! \sum_{
 \ell >L \atop
 \ell \in  {\mathcal K}(  L)  }  \!  \!  \!  \Lambda^n_{\boldsymbol{\varepsilon^{\ell}} ,\boldsymbol{\gamma^{\ell}},\boldsymbol{x^{\ell}} } \left( -\frac{ \varepsilon_n^{\ell} }{\gamma_n^{\ell} }
 \nabla_h \Delta_h^{-1} \partial_3  \phi^{\ell}     ,  \phi^{\ell}  \right) - \!  \!   \!  \!  \!   \sum_{
 \ell >L \atop
1 \leq \kappa( \ell ) \leq L  }  \!  \!   \!  \!  \!  \Lambda^n_{\boldsymbol{\eta^{\ell}} ,\boldsymbol{\delta^{\ell}},\boldsymbol{\tilde x^{\ell}} }  \big( \tilde  \phi_\alpha^{h,\ell} +\tilde  r_\alpha^{h,\ell}   ,  0 \big) \nonumber\\
&\quad &{}+\big( \tilde \psi_n^{h,  L} -\nabla_h \Delta_h^{-1} \partial_3  \psi_n^{L} ,\psi_n^{L} \big)   \, .\nonumber
   \end{eqnarray}
%\begin{rmk}
%{\rm
%We have decided to incorporate somewhat artificially in the  first two sets of profiles (the first two lines in the decomposition of~$u_{0,n}$ above) the indexes~$\kappa (\ell)$ (resp.~$\ell$) which are larger than~$L$. These are then removed in the fourth line of~(\ref{newformulationdecompositiondata}).
%One could have chosen to remove those large indexes  at once from the first two lines, but that would create  notational complications and is in fact unnecessary since
%  the large-index contributions will be shown to be arbitrarily small. }
%\end{rmk}

\medskip

 \noindent Before moving on to the time evolution of~(\ref{newformulationdecompositiondata}), we are now in position to
 state the second assumption entering in the statement of Theorem~\ref{mainresult}.
  \begin{assumption}\label{sousdecompo}
 With the notation of  Proposition~{\rm \ref{prop:decompositiondata}}, there is~$L_0$ such that for every~$L \geq L_0$, the following holds.

\medskip

   $\bullet $ $ $     Suppose there are   two indexes~$\ell_1 \neq \ell_2$ in~$\{1,\dots,L\}$
  such that the following properties are satisfied:
  \begin{equation}\label{assumptions}
 \begin{aligned}
  \eta_n^{\ell_1} =  \eta_n^{\ell_2}  \,, \quad   \delta_n^{\ell_1} \to \infty \,, \quad  \delta_n^{\ell_2} \to  1\:  \: \mbox{or} \:  \:  \infty
  \quad \mbox{with}\quad \delta_n^{\ell_1} /  \delta_n^{\ell_2} \to \infty \, ,\\
\mbox{and} \:  \: (\tilde x_{n}^{\ell_1} - \tilde x_{n}^{\ell_2} )^{\eta_n^{\ell_2}, \delta_n^{\ell_2}} \to a^{\ell_1,\ell_2} \in \R^3 \, ,
%\quad \frac{\tilde x_{n,3}^{\ell_1}}{ \delta_n^{\ell_1} } \to  0 \, ,
\quad   \frac{\tilde x_{n,3}^{\ell_2}}{ \delta_n^{\ell_2} } \to  a_3^k \in \R\, .
   \end{aligned}
\end{equation}
   Then one has~$  \tilde \phi^{h,\ell_1} (\cdot , 0) := ( \tilde \phi^{h,\ell_1}_\alpha+ \tilde r^{h,\ell_1}_\alpha )(\cdot , 0)  \equiv 0 .$

\medskip

  $\bullet $ $ $ If~$u_0 \not\equiv 0$ and if there are~$\ell_1 \neq \ell_2 \in \{1,\dots,L\}$ such that for~$i \in \{1,2\}$,~$\eta_n^{\ell_i} = 1$ with~$\delta_n^{\ell_i} \to \infty$ while~$\tilde x_{n,h}^{\ell_i}$  is bounded
   and~$\tilde x_{n,3}^{\ell_i} / \delta_n^{\ell_i} \to \tilde a_{3}^{\ell_i}\in \R$,
  then~$\tilde \phi^{h,\ell _i }(\cdot , -\tilde a_{3}^{\ell_i}  ) \equiv 0$ for each~$i \in \{1,2\}$.

\medskip

  $\bullet $ $ $ A similar result holds for the profiles~$\phi^\ell:=\tilde \phi^{h,\ell}_\alpha+   r^{h,\ell}_\alpha$, with the corresponding assumptions on the scales and cores.
  \end{assumption}
   
 \begin{prop}\label{zeroatzeroagain}
  With the notation of  Proposition~{\rm \ref{prop:decompositiondata}} assume the following:

     $\bullet $ $ $  If~$\ell_1 \neq \ell_2$ in~$\{1,\dots,L\}$ are two indexes satisfying~{\rm (\ref{assumptions})}, then a weak limit of the sequence~$\eta_n^{\ell_2} (u_{0,n}^h - u^h_0 -   \tilde \psi_n^{h,  L} +\nabla_h \Delta_h^{-1} \partial_3  \psi_n^L ) (\eta_n^{\ell_2} y_h +\tilde x_{n,h}^{\ell_2} ,\delta_n^{\ell_2} y_3+ \tilde x_{n,3}^{\ell_2}  )$  is~$\tilde \phi^{h,\ell_2}(y)$.

  $\bullet $ $ $  A similar result holds for~$\varepsilon_n^{\ell_2}(u_{0,n}^3 - u^3_0 -   \tilde \psi_n^{3,\tilde L} - \psi_n^L )  (\varepsilon_n^{\ell_2} y_h + x_{n,h}^{\ell_2} ,\gamma_n^{\ell_2} y_3+ x_{n,3}^{\ell_2}  )$, with the corresponding assumptions on the scales and cores.

\medskip
\noindent
  Then Assumption~{\rm \ref{sousdecompo}} holds.
      \end{prop}
 \begin{proof}[Proof of Proposition~{\rm \ref{zeroatzeroagain}}]

  $\bullet $ $ $ We shall start by proving the result for a couple~$\ell_1\neq\ell_2$  chosen in~$\{1,\dots , L\}$  so that~$\delta_n^{\ell_1}$
 is the largest vertical scale among the vertical scales associated with all couples satisfying~{\rm(\ref{assumptions})}.

 \noindent We begin by noticing that the limit   (after extraction)  of~$  \displaystyle\frac{\tilde x_{n,3}^{\ell_1}}{ \delta_n^{\ell_1} }$ is  necessarily  zero since
 \begin{equation}\label{limxelldeltaell}
   \frac{\tilde   x_{n,3}^{\ell_1}}{ \delta_n^{\ell_1} } =  \left( \frac{\tilde x_{n,3}^{\ell_1} -\tilde x_{n,3}^{\ell_2} }{ \delta_n^{\ell_2} }
   + \frac{\tilde x_{n,3}^{\ell_2} }{ \delta_n^{\ell_2} }
    \right)\frac{ \delta_n^{\ell_2} }{ \delta_n^{\ell_1} } \to 0 \, .
 \end{equation}
  Without loss of generality we may also assume that for the index~$\ell_1$ we have chosen, $\delta_n^{\ell_2}$ is the largest vertical scale satisfying~(\ref{assumptions}).
By the   hypothesis of Proposition~\ref{zeroatzeroagain} we know that the weak limit of~$\eta_n^{\ell_2} \big(u_{0,n}^h - u_0^h-   \tilde \psi_n^{h,\tilde L} +\nabla_h \Delta_h^{-1} \partial_3  \psi_n^L \big) (\eta_n^{\ell_2} y_h + \tilde  x_{n,h}^{\ell_2} ,\delta_n^{\ell_2} y_3+  \tilde x_{n,3}^{\ell_2}  )$ is~$\tilde \phi^{h,\ell_2}(y)$. This weak limit may be explicitly computed: noticing that for any integer~$k$,
$$
\eta_n^{\ell_2}   \big( \Lambda^n_{\boldsymbol{\eta^{k}} ,\boldsymbol{\delta^{k}},\boldsymbol{\tilde x^{k}} }
\tilde \phi^{h,k}  \big)(\eta_n^{\ell_2} y_h + \tilde  x_{n,h}^{\ell_2} ,\delta_n^{\ell_2} y_3+  \tilde x_{n,3}^{\ell_2}  ) =
  \Lambda^n_{\frac{\boldsymbol{\eta^{k}}}{ \boldsymbol{\eta^{\ell_2}} } ,\frac{\boldsymbol{\delta^{k}}}{ \boldsymbol{\delta^{\ell_2}} } ,{\boldsymbol {\tilde x^{k,\ell_2}}} } \tilde \phi^{h,k} (y), $$
  with~$  {\tilde x_{n}^{k,\ell_2}} :=
   (\tilde x_{n}^{k} - \tilde x_{n}^{\ell_2} )^{\eta_n^{\ell_2}, \delta_n^{\ell_2}} \, ,  $
 we find that the weak  limit of such a term is zero except in three situations : if~$k = \ell_2$, if~$k = \ell_1$, or if
  \begin{equation}\label{lastpossibility}
  \eta_n^{k} =  \eta_n^{\ell_2} \, , \quad     \delta_n^{k} /  \delta_n^{\ell_2} \to \infty \, ,\quad
  (\tilde x_{n}^{k} - \tilde x_{n}^{\ell_2} )^{\eta_n^{k}, \delta_n^{k}}  \to a^{k,\ell_2} \in \R ^3 \,  .
\end{equation}
  If~$k = \ell_2$, then the function is simply equal to~$\tilde\phi^{h,\ell_2} (y)$, and if~$k = \ell_1$  then by~(\ref{assumptions}) the weak  limit is equal to~$\tilde  \phi^{h,\ell_1} (y_h + a_{h}^{\ell_1,\ell_2} , 0 ). $ Finally if~(\ref{lastpossibility}) were to hold then
  in particular~$k$ would satisfy the same properties as~$\ell_2$ in the statement of the proposition, while~$  \delta_n^{k} /  \delta_n^{\ell_2} \to \infty$, and that is impossible by choice of~$\ell_2$ as corresponding to the largest vertical scale satisfying~(\ref{assumptions}).

\noindent So finally the weak limit of~$\eta_n^{\ell_2} \big(u_{0,n}^h - u_0^h-   \tilde \psi_n^{h,\tilde L} +\nabla_h \Delta_h^{-1} \partial_3  \psi_n^L \big) (\eta_n^{\ell_2} y_h + \tilde  x_{n,h}^{\ell_2} ,\delta_n^{\ell_2} y_3+  \tilde x_{n,3}^{\ell_2}  )$  is~$\tilde\phi^{h,\ell_2} (y) +
\tilde  \phi^{\ell_1} (y_h + a_{h}^{\ell_1,\ell_2} , 0 ) ,
 $ hence necessarily by the assumptions of Proposition~\ref{zeroatzeroagain}, we have that~$\tilde  \phi^{h,\ell_1} (y_h + a_{h}^{\ell_1,\ell_2} , 0 ) \equiv 0$ so the result is proved in the case of the largest possible vertical scale.

\smallskip
\noindent Now  we can argue by induction for the other possible~$\ell^1$'s: suppose that~$\ell^1$  corresponds to the second largest for instance, then calling~$\delta_n^{\ell_0}$ the largest one, the same argument implies that the weak limit of the sequence~$\eta_n^{\ell_2} \big(u_{0,n}^h - u_0^h-   \tilde \psi_n^{h,\tilde L} +\nabla_h \Delta_h^{-1} \partial_3  \psi_n^L \big) (\eta_n^{\ell_2} y_h + \tilde  x_{n,h}^{\ell_2} ,\delta_n^{\ell_2} y_3+  \tilde x_{n,3}^{\ell_2}  )$  is the function~$\tilde\phi^{\ell_2} (y) +
\tilde  \phi^{h,\ell_1} (y_h + a_{h}^{\ell_1,\ell_2} , 0 )
 + \tilde  \phi^{h,\ell_0} (y_h + a_{h}^{\ell_0,\ell_2} , 0 )  = \tilde\phi^{\ell_2} (y) +
\tilde  \phi^{h,\ell_1} (y_h + a_{h}^{\ell_1,\ell_2} , 0 ) $ hence~$\tilde  \phi^{\ell_1} (y_h + a_{h}^{\ell_1,\ell_2} , 0 )  \equiv 0$
and by induction, the result is proved.

    \medskip
 \noindent   $\bullet $ $ $ The proof of the second point is  very similar: we first consider~$\ell_1$ 
  corresponding to the largest vertical scale  among the indexes satsfying~$\eta_n^\ell = 1$,~$\delta_n^\ell \to \infty$,~$\tilde x_{n,h}^\ell \to \tilde a_{h}^\ell$    bounded
   and~$\tilde x_{n,3}^\ell / \delta_n^\ell \to \tilde a_{3}^\ell\in \R$. If there is no other
  index satisfying those requirements then we notice that the weak limit of~$u_{0,n}^h - u_0^h-   \tilde \psi_n^{h,\tilde L} +\nabla_h \Delta_h^{-1} \partial_3  \psi_n^L$ is~$ \tilde\phi^{\ell} (y_h -  \tilde a_{h}^\ell , -  \tilde a_{3}^\ell) $, while we also know that it is zero, so the result follows. If there is a  second index satisfying those requirements, then we consider~$\delta_n^{\ell_2}$ the next largest vertical scale (by orthogonality it cannot be equal to~$\delta_n^{\ell_1}$) 	and we use the assumption of Proposition~\ref{zeroatzeroagain}, which implies that the weak limit of the sequence~$ (u_{0,n}^h - u_0^h-   \tilde \psi_n^{h,\tilde L} +\nabla_h \Delta_h^{-1} \partial_3  \psi_n^L ( y_h + \tilde  x_{n,h}^{\ell_2} ,\delta_n^{\ell_2} y_3+  \tilde x_{n,3}^{\ell_2}  )$  is the function~$\tilde\phi^{h,\ell_2} (y) $ while a direct computation gives the limit~$ \tilde\phi^{h,\ell_2} (y) +
\tilde  \phi^{h,\ell_1} (y_h - \tilde a_{h}^{\ell_1} + \tilde a_{h}^{\ell_2},    -  \tilde a_{3}^\ell )  $  and again we get the result.

\noindent  The rest of the argument is as above, by induction on the size of the vertical scales.

    \medskip
 \noindent $\bullet $ $ $  The proof is identical for the profiles~$\phi^\ell$.

    \medskip
  \noindent Proposition \ref{zeroatzeroagain} is proved.
 \end{proof}
\begin{rmk}\label{rmkass2}
 {\rm Assuming the hypotheses of Proposition~\ref{zeroatzeroagain} is actually quite natural. Indeed for any choice of sequences of cores~$(x_{n,h}^\ell)_{n \in \N}$ and of scales~$(\eta_n^\ell)_{n \in \N}$, one has that the sequence~$\eta_n^\ell (u_{0,n}^h - u^h_0 -   \tilde \psi_n^{h,\tilde L} +\nabla_h \Delta_h^{-1} \partial_3  \psi_n^L ) (\eta_n ^\ell y_h + x_{n,h}^\ell ,\delta_n ^\ell y_3+ x_{n,3}  )$  converges in~${\mathcal S}'$, and it is assumed here that the weak limit is precisely the profile~$\tilde \phi^{h,\ell}$. Note that
 for a profile decomposition in the space~$\dot B^{s,s'}_{p,q}$ that is obvious as soon as~$s<2/p$ and~$s'<1/p$. Here we have~$s' = 1/p$
 so this is a true assumption (in the same way as the sequence~$f(x_h,\varepsilon x_3)$ does not necessarily converge weakly to zero with~$\varepsilon$).

\smallskip
\noindent    For example the sequence provided in Remark~{\rm \ref{rmkdifferentscales}} satisfies Assumption~\ref{sousdecompo} since there is only one profile involved.

\smallskip
\noindent   More generally consider the sequence (assuming that~$0 \neq  \alpha  \neq \gamma,$ and~$ \beta_1,\beta_2 \neq \alpha, \beta_1 \neq \gamma$)
     $$
      2^{\alpha  n} \Big( f_1\big(2^{\alpha  n} x_1,2^{\alpha  n}  x_2 ,2^{\beta_1 n}  x_3  -a_3 \big) +    f_2\big(2^{\alpha n} x_1 ,2^{\alpha n} x_2,2^{\beta_2 n}   x_3\big)\Big) + 2^{\gamma n} f_3\big(2^{\gamma n} x_1,2^{\gamma n}  x_2 ,2^{\beta_1 n}  x_3 \big)  \, .
            $$
It clearly    satisfies   Assumption~\ref{defadf}. If~$\beta_2 = \beta_1$ then Assumption~\ref{sousdecompo} is also satisfied. If ~$\beta_1 < \beta_2<0$   then one must have~$f_1 (\cdot, -a_3)  \equiv 0$ to ensure Assumption~\ref{sousdecompo}: if there are two profiles with the same horizontal scale (here~$2^{-\alpha n}$) and different vertical scales going both to infinity (since~$\beta_2 \neq \beta_1$ and both are negative), then the profile with the largest vertical scale (here~$f_1\big(2^{\alpha  n} x_1,2^{\alpha  n}  x_2 ,2^{\beta_1 n}  x_3  -a_3 \big)$ since~$\beta_1 < \beta_2$), must vanish at~$x_3 = 0$. 
}
 \end{rmk}

\begin{rmk}\label{ifalsoell2}
 {\rm If it is assumed that the initial data is bounded also in~$L^2(\R^3)$, then    the same arguments as those leading to Lemma~\ref{orthoanisodiv1}
 allow to infer that the vertical scales~$\gamma_n^\ell$ and~$\delta_n^\ell$ must all go to zero. In particular    Assumption~\ref{sousdecompo} is unnecessary in that case since the hypotheses are never met.
 }
 \end{rmk}

\begin{rmk} \label{ifalsoell22}
 {\rm   Assumption~\ref{sousdecompo} is used in the following to show that profiles do not interact one with another (see Paragraph~\ref{studyforcingterm}).
 }
 \end{rmk}

 %%%%%%%%%%%%%%%%%%%%%%%%%%%%%%%%%%%%%%%%%%%%%

 \section{Time evolution  of each profile, construction of an approximate solution}\label{globalprofile}

 In this section we
 shall construct an approximate solution to the Navier-Stokes equations by evolving in time each individual profile provided in Proposition~{\rm \ref{prop:decompositiondata}} -- or rather the version written in~(\ref{newformulationdecompositiondata}) --  either by the Navier-Stokes flow or by a linear transport-diffusion equation, depending on  the profiles. First we shall be needing a time-dependent version of the scaling operator~$ \Lambda^n_{\boldsymbol{\varepsilon} ,\boldsymbol{\gamma},\boldsymbol{x} }$ given in Definition~\ref{notationprofiles}.
    \begin{defi}\label{notationprofilestime} For any two sequences~$\boldsymbol{\varepsilon}=(\varepsilon_n)_{n \in \N}$ and~$\boldsymbol{\gamma}=(\gamma_n)_{n \in \N}$  of positive real numbers and any sequence~$\boldsymbol{x}=(x_n)_{n \in \N}$  in~$\R^3$ we define the scaling operator
 $$
 \widetilde \Lambda^n_{\boldsymbol{\varepsilon} ,\boldsymbol{\gamma},\boldsymbol{x} } \phi (t,x) :=
  \frac1 {\varepsilon_n}  \phi \left( \frac t {\varepsilon_n^2}, \frac{x_h-x_{n,h} }{\varepsilon_n}, \frac{x_3-x_{n,3}} {\gamma_n} \right) \, .
  $$
       \end{defi}

\noindent
Next  let us introduce some notation for function spaces naturally associated with the  resolution of the Navier-Stokes equations.  We refer to Appendix~{\rm \ref{appendixlp}} for   definitions.
   \begin{defi}\label{notationspaces}
   We define the following function spaces, for~$1 \leq p \leq \infty$ and~$0<q\leq \infty$:
    $$
     \begin{aligned}
       {\mathcal I}_{p,q} &:=   \bigcap_{r=1}^\infty \widetilde{L^r}\big(\R^+; \dot B^{-1+\frac 3p+\frac2r}_{p,q}(\R^3)\big) \, , \\
     {\mathcal A}_{p,q} &:=
      \bigcap_{r=1}^\infty \,      \widetilde{L^r}(\R^+;\dot B^{-1+\frac2p+\frac2r ,\frac1p  }_{p,q}) \, ,\\
          {\mathcal S}_{p,q}&:=  \widetilde{L^\infty}(\R^+;\dot B^{-1+\frac2p  ,\frac1p  }_{p,q} )  \cap \widetilde {L^1}
         (\R^+;\dot B^{1+\frac2p  ,\frac1p  }_{p,q} \cap \dot B^{ -1+\frac2p  ,2+\frac1p  }_{p,q} ) \, .
     \end{aligned}
     $$
            \end{defi}
        \begin{rmk}\label{onthefunctionspaces}{\rm
The spaces defined above are
 natural spaces for the resolution of the Navier-Stokes equations: for instance
 $   {\mathcal I}_{p,\infty}$ is associated with small data in~$\dot B^{-1+\frac 3p}_{p,\infty}(\R^3)$ (see~\cite{Cannone},\cite{Planchon}, as well as~\cite{BCD}) and~$ {\mathcal S}_{p,1}$ with small  data in~${\dot B}^{-1+\frac2p,\frac1p}_{p,1}$ (see Appendix~{\rm \ref{globalsmallaniso}}).  Note that~$ {\mathcal A}_{p,q}$  contains strictly~${\mathcal S}_{p,q}$ and~${\mathcal A}_{p_1,q}$ is embedded in~${\mathcal A}_{p_2,q}$ as soon as~$p_1 \leq p_2$, and similarly for~${\mathcal S}_{p_1,q}$ and~${\mathcal S}_{p_2,q}$.}
\end{rmk}
\begin{rmk}\label{tildelambdainvariant} {\rm
The operator~$  \widetilde \Lambda^n_{\boldsymbol{\varepsilon} ,\boldsymbol{\gamma},\boldsymbol{x} } $   is an isometry in~$   {\mathcal A}_{p,q}$ for all~$1 \leq p \leq \infty$ and~$0<q\leq \infty$. That is however not the case for the space~${\mathcal S}_{p,q}$.}
\end{rmk}

\noindent  Now let us consider the decomposition~(\ref{newformulationdecompositiondata}), and evolve   each term in time so as to construct by superposition an approximate solution to the Navier-Stokes equations with data~$u_{0,n}$. We leave to Section~\ref{superpositionglobal} the proof that the superposition is indeed an approximate solution to~(NS).

 \medskip
 \noindent
 $\bullet$ $ $ The first term of the decomposition~(\ref{newformulationdecompositiondata}) is the weak limit~$u_0 \in {\mathcal B}^1_q  $, which gives rise to a unique, global solution by assumption: we  define~$u \in  {\mathcal S}_{1,1}(\infty)$ the associate global solution.
 Due to Corollary~\ref{strongstability} stated page~\pageref{strongstability}, we know that actually~$u$ belongs to~$  {\mathcal S}_{1,1}$.
 
   \bigskip
 \noindent
 $\bullet$ $ $ Let us
 turn to the profiles   in the decomposition~(\ref{newformulationdecompositiondata}), namely first the terms
  $$
\tilde  \varphi_{0,n}^{\ell} :=\displaystyle \Lambda^n_{\boldsymbol{\eta^{\ell}} ,\boldsymbol{\delta^{\ell}},\boldsymbol{\tilde x^{\ell}} } \left(
\tilde  \phi_\alpha^{h,\ell}   -\frac{ \eta_n^{\ell} }{\delta_n^{\ell} }
 (\nabla_h \Delta_h^{-1} \partial_3  \phi_\alpha^{\kappa(\ell)}  )
,  \phi_\alpha^{\kappa(\ell)}
\right) $$
for  any~$\ell \in \N$.
We use the notation of Appendix~\ref{globalsmallaniso}, and in particular that of Theorem~\ref{globalaniso}.

\begin{lem}\label{globalfirstprofile}
Let~$\ell \in \N$. There is~$ \tilde L_0$, independent of~$n$ and~$\alpha$, such that the following properties hold.

\medskip
\noindent $\bullet$ $ $  If~$\ell \geq \tilde L_0$ and~$\kappa(\ell) \geq \tilde L_0$, then
for all~$\alpha \in (0,1)$ and~$n$ large enough, $\tilde  \varphi_{0,n}^{\ell} $ belongs to~$\cG$ and the associate solution~$ \tilde u^\ell_n$ to~{\rm(NS)} satisfies
\begin{equation}\label{smallXp}
 \forall \ell \geq \tilde L_0\, \, s.t.\,  \, \kappa(\ell) \geq \tilde L_0  \,, \quad \|\tilde u^\ell_n  \|_{ {\mathcal S}_{3,1}} \leq 2 \big( \|\tilde  \phi_\alpha^{h,\ell} \|_{\dot B^{-\frac13,\frac13}_{3,1}} + \|   \phi_\alpha^{\kappa(\ell)} \|_{\dot B^{-\frac13,\frac13}_{3,1}} \big)\leq 2 c_0  \, .
\end{equation}
 \medskip
\noindent $\bullet$ $ $  For every~$\ell \in \N$, if~$\eta_n^\ell /\delta_n^\ell $ converges to~$\infty$ when~$n$ goes to infinity, then for all~$\alpha \in (0,1)$ and for~$n$ large enough~$\tilde  \varphi_{0,n}^{\ell} $ belongs to~$\cG$: the associate solution~$ \tilde u^\ell_n$  to~{\rm(NS)} is bounded in~${\mathcal S}_{3,1}$ and satisfies for all~$1 \leq r \leq \infty$ and all~$  \displaystyle   \frac13  \leq  \sigma \leq  \frac13 +\frac2r  $
\begin{equation}\label{smallkato}
 \tilde u^\ell_n  \to 0  \quad \mbox{in} \quad
\widetilde{L^{r}}(\R^+; \dot B^{-\frac13 + \sigma ,\frac2r - \sigma+\frac13}_{3,1})  \, , \quad  n \to \infty \, .
 \end{equation}

\medskip
\noindent $\bullet$ $ $  For every~$\ell \in \N$, if~$\eta_n^\ell /\delta_n^\ell $ converges to~$0$ when~$n$ goes to infinity,  then for all~$\alpha \in (0,1)$ and for~$n$ large enough~$\tilde  \varphi_{0,n}^{\ell}$  belongs to~$\cG$:  the associate solution~$ \tilde u^\ell_n$  to~{\rm(NS)}  is uniformly bounded in the space~${\mathcal S}_{1,1}$ and satisfies for all~$\alpha \in (0,1)$
\begin{equation}\label{explicitformjlarge}
\begin{aligned}
  \tilde u^\ell_n   =\widetilde \Lambda^n_{\boldsymbol{\eta^{\ell}} ,\boldsymbol{\delta^{\ell}},\boldsymbol{\tilde x^{\ell}} }
  \left(\tilde  U^{h,\ell}  +\frac{ \eta_n^{\ell} }{\delta_n^{\ell} }
   U_n^{\kappa(\ell),h}    ,U_n^{\kappa(\ell),3}
 \right)
  + \tilde R_n^{\ell} \,  ,  \quad \tilde R_n^{\ell} \,  \mbox{bounded in} \,\, {\mathcal S}_{3,1}  \\
  \mbox{with}  \quad  \tilde R_n^{\ell}  \to 0 \: \: \mbox {in} \: \:\widetilde{L^{2}}(\R^+; \dot B^{\frac23 ,\frac13}_{3,1})\cap L^1(\R^+ ; \dot B^{\frac53, \frac13}_{3,1} \cap \dot B^{\frac23, \frac43}_{3,1}) \, , \quad  n \to \infty \,   ,
\end{aligned}
\end{equation}
 while~$\tilde  U^{h,\ell}$,~$U_n^{\kappa(\ell),3}$ and~$\displaystyle   \frac{ \eta_n^{\ell} }{\delta_n^{\ell} }
   U_n^{\kappa(\ell),h}   $ are      smooth and bounded in~${\mathcal S}_{1,1}$.

  \noindent Finally if~$\tilde  \phi^{h,\ell}(\cdot,z_3) \equiv   \phi^{{\kappa(\ell)} }(\cdot,z_3) \equiv 0$ for some~$z_3 \in \R$, then for all~$s \geq 0$,
\begin{equation}\label{smallerthanalphaatz3}
  \limsup_{n \to \infty} \big\|   \tilde  U^{h,\ell} (\cdot, z_3)  +  U_n^{\kappa(\ell),3}   (\cdot, z_3)\big\|_{L^\infty(\R^+;H^s(\R^2)) \cap L^2(\R^+;\dot H^{s+1}(\R^2))}\lesssim \alpha   \, .
\end{equation}
 \end{lem}

\begin{proof} [Proof of Lemma~{\rm \ref{globalfirstprofile}}]
 $\bullet$ $ $ By the stability property~{\rm (\ref{orthonorms})}, for all~$\beta >0$ there is~$\tilde L (\beta)$ such that
if~$\ell \geq \tilde L (\beta)$ and~$\kappa(\ell)\geq \tilde L (\beta)$, then
$$
 \|\tilde  \phi_\alpha^{h,\ell} \|_{{\mathcal B}^{1}_q} + \|   \phi_\alpha^{\kappa(\ell)} \|_{{\mathcal B}^{1}_q}  \leq \beta \, .
$$
Then if~$\beta$ is small enough,  in particular~$\tilde  \phi_\alpha^{h,\ell}$ is smaller than, say~$c_0/2$ in~$\dot B^{-\frac13,\frac13}_{3,1}$ (by Sobolev embeddings).

\noindent
Now let~$\alpha>0$ be given and let us consider the initial data~$(\displaystyle -\frac{ \eta_n^{\ell} }{\delta_n^{\ell} }
\nabla_h \Delta_h^{-1} \partial_3  \phi_\alpha^{\kappa(\ell)}
,   \phi_\alpha^{\kappa(\ell)} )$. 
Notice that  the only possible limit for the ratio of scales associated with~$\phi_\alpha^{\kappa(\ell)} $ is zero by Proposition~\ref{prop:decompositiondata}, so we can restrict our attention here to the case when~${ \eta_n^{\ell} } / {\delta_n^{\ell} } \to 0$. 
By construction of~$   \phi_\alpha^{ \kappa(\ell)}$ in~(\ref{rlalpha}), 
 the vector field~$
\nabla_h \Delta_h^{-1} \partial_3  \phi_\alpha^{\kappa(\ell)}
$ belongs to~${\mathcal B}^{1}_q$ for each given~$\alpha$, hence since~$\eta_n^\ell /\delta_n^\ell $ converges to~$0$ when~$n$ goes to infinity, then for~$n$ large enough and for~$\kappa(\ell)\geq \tilde L (\beta)$
$$
\Big\| -\frac{ \eta_n^{\ell} }{\delta_n^{\ell} }
 (\nabla_h \Delta_h^{-1} \partial_3  \phi_\alpha^{\kappa(\ell)}  )
,  \phi_\alpha^{\kappa(\ell)} \Big\|_{{\mathcal B}^{1}_q} \leq 2 \beta \, .
$$
Finally choosing~$\beta \leq c_0/4$, for~$\ell \geq \tilde L (\beta)$,~$\kappa(\ell)\geq \tilde L (\beta)$
and~$n$ large enough (depending on~$\ell$ and~$\alpha$)  Theorem~{\rm \ref{globalaniso}}  applies (using also Remark~{\rm \ref{isometry}}) to yield that~$\tilde  \varphi_{0,n}^{\ell}$  belongs to~$\cG$  and~{\rm (\ref{smallXp})} holds.

\medskip\noindent
  $\bullet$ $ $If~$\eta_n^\ell /\delta_n^\ell $ converges to~$\infty$,  then we observe that~$\phi_\alpha^{\kappa(\ell)} \equiv 0$ (since  as recalled above the only possible limit for the ratio of scales associated with~$\phi_\alpha^{\kappa(\ell)} $ is zero) and we have by a direct computation
   $$
\left\| \Lambda^n_{\boldsymbol{\eta^{\ell}} ,\boldsymbol{\delta^{\ell}},\boldsymbol{\tilde x^{\ell}} } \left(
\tilde  \phi_\alpha^{h,\ell} ,0\right)\right\|_{ \dot B^{0}_{3,1} } \lesssim  \left( \frac{\delta_n^\ell
}{\eta_n^\ell} \right)^\frac13 \, .
$$
In particular for~$n$ large enough the data is small in~$ \dot B^{0}_{3,1} $   so small data theory of~\cite{kato} and~\cite{Planchon} (see also~\cite{BCD}) gives the result:
there is a global solution to (NS) associated with that initial data,  which goes to zero (like~$({\delta_n^\ell
}/{\eta_n^\ell})^\frac13$) in~$\widetilde {L^\infty}(\R^+ ; \dot B^{0}_{3,1}) \cap\widetilde {L^1}(\R^+ ; \dot B^{2}_{3,1} ) $. By
  Proposition~\ref{propanisoiso} and interpolation,  it therefore goes to zero in~$\widetilde{L^{r}}(\R^+; \dot B^{-\frac13 + \sigma ,\frac2r - \sigma+\frac13}_{3,1}) $ for all~$1 \leq r \leq \infty$ and all~$\sigma \in \displaystyle [\frac13 \, , \frac13 +\frac2r  ]$, as expected.

\smallskip
\noindent In particular~$\widetilde u_n^\ell$ is bounded in~$ \widetilde{L^{2}}(\R^+; \dot B^{\frac23,\frac13}_{3,1})$ which controls the Navier-Stokes equation  for data in~$\dot B^{-\frac13,\frac13}_{3,1}$  (see Theorem~\ref{globalaniso}), so we get in particular that~$\widetilde u_n^\ell$ is bounded in~$ {\mathcal S}_{3,1}$.

\medskip
  \noindent
   $\bullet$ $ $Conversely let us suppose that~$\eta_n^\ell /\delta_n^\ell $ converges to~$0$.
  Then by  (isotropic) scale and translation invariance of (NS) we can first rescale by~$\eta_n^\ell $ and translate by~$\widetilde x_n^\ell$, hence   consider the initial data
   $$
  \begin{aligned}
\widetilde \phi_{0,n}^\ell (x) & :=  \displaystyle \Lambda^n_{\boldsymbol{1} ,\boldsymbol{\frac{\delta^\ell}{\eta^{\ell}}},\boldsymbol{0}}
 \left( \tilde  \phi_\alpha^{h,\ell}   -\frac{ \eta_n^{\ell} }{\delta_n^{\ell} }
 (\nabla_h \Delta_h^{-1} \partial_3  \phi_\alpha^{\kappa(\ell)}  ) ,
  \phi_\alpha^{\kappa(\ell)}
\right) (x)\\
& =  \left( \tilde  \phi_\alpha^{h,\ell} -\frac{ \eta_n^{\ell} }{\delta_n^{\ell} }
 (\nabla_h \Delta_h^{-1} \partial_3  \phi_\alpha^{\kappa(\ell)}  ),  \phi_\alpha^{\kappa(\ell)} \right) (  x_h ,  \frac{\eta^{\ell}_n} {\delta^{\ell}_n} x_3 )
 \,.
  \end{aligned}$$
 Since~$\eta_n^\ell /\delta_n^\ell \to 0$ as~$n$ goes to infinity,
we can rely on Theorem 3 in~\cite{cg3} which states that as soon as~$\eta_n^\ell /\delta_n^\ell$ is small enough (depending on norms of the profiles~$\tilde \phi_\alpha^{h,\ell} ,  \phi_\alpha^{\kappa(\ell)} $), then~$\widetilde \phi_{0,n}^\ell$ belongs to~$\cG$ and  according to~\cite{cg3} the solution to~(NS) associated with~$\widetilde \phi_{0,n}^\ell$ is of the form
$$
    \big (\tilde  U^{h,\ell}   +\frac{ \eta_n^{\ell} }{\delta_n^{\ell} }
 U_n^{\kappa(\ell),h}    ,U_n^{\kappa(\ell),3}
\big ) (  t, x_h ,  \frac{\eta^{\ell}_n} {\delta^{\ell}_n} x_3 )
  + \tilde r_n^{\ell} (t,x)
$$
where for each~$z_3 \in \R$, $\tilde  U^{h,\ell}(\cdot,z_3)$ is the global solution to the two-dimensional Navier-Stokes equations with data~$\tilde  \phi_\alpha^{h,\ell} (\cdot,z_3)$, while~$U_n^{\kappa(\ell)}$ is a divergence-free vector field solving the linear transport-diffusion equation~$(T_{\underline v}^{\varepsilon} )$ of~\cite{cg3}
 with~$\underline v = \tilde  U^{h,\ell}$ and~$\varepsilon =  \eta_n^{\ell} /\delta_n^{\ell} $, with data~$  \displaystyle \big( - \nabla_h \Delta_h^{-1} \partial_3  \phi_\alpha^{\kappa(\ell)}    ,
  \phi_\alpha^{\kappa(\ell)} \big)$: we have, for some pressure~$p_n^{\kappa(\ell)} $
 $$
\partial_tU_n^{\kappa(\ell)} +\tilde  U^{h,\ell} \cdot \nabla^hU_n^{\kappa(\ell)} -\Delta_hU_n^{\kappa(\ell)}
-\left(\frac{ \eta_n^{\ell} }{\delta_n^{\ell} } \right)^2\partial_3^2U_n^{\kappa(\ell)}= -\left(\nabla^h,  \Big(\frac{ \eta_n^{\ell} }{\delta_n^{\ell} } \Big)^2 \partial_{3}\right) p_{n} ^{\kappa(\ell)}  \, .
 $$
 Both~$ \tilde  U^{h,\ell}$ and~$U_n^{\kappa(\ell)}$ are as smooth as needed.

 \smallskip
 \noindent In particular relying on~\cite{cg3} Proposition~3.2, and~\cite{ghz} (where estimates in the -- more difficult -- inhomogeneous situation are obtained), we have that~$\tilde  U^{h,\ell}$, $U_n^{\kappa(\ell),3}$ and~$\displaystyle   \frac{ \eta_n^{\ell} }{\delta_n^{\ell} }
   U_n^{\kappa(\ell),h}   $ are    bounded in~${\mathcal S}_{2,1} $. It is not difficult to prove also (for instance using the estimates of Appendix~\ref{globalsmallaniso}) that they are bounded in~${\mathcal S}_{1,1} $.

 \smallskip
 \noindent
 Furthermore~$\widetilde r_n^\ell $ goes to zero in~${\mathcal I}_{2,1}$    by~\cite{cg3}  (actually the result of~\cite{cg3} only states the convergence to zero in~$L^\infty(\R^+;\dot H^\frac12) \cap  L^2(\R^+;\dot H^\frac32)$  but it is clear from the proof that it can be extended all the way to~${\mathcal I}_{2,1} $). It then suffices to unscale to the original data to find the form~{\rm (\ref{explicitformjlarge})},  with~$\widetilde R_n^\ell   $ going to zero in~${\mathcal I}_{2,1}$. We infer in particular  by Proposition~\ref{propanisoiso} and Sobolev embeddings that~$\widetilde R_n^\ell   $ goes to zero in~$\widetilde{L^{2}}(\R^+; \dot B^{\frac23 ,\frac13}_{3,1})\cap L^1(\R^+ ; \dot B^{\frac53, \frac13}_{3,1} \cap \dot B^{\frac23, \frac43}_{3,1})$ as required.
   Finally let us prove that~$\widetilde R_n^\ell $ is bounded in~${\mathcal S}_{3,1}$. We notice that due to the above bounds, the function~$\widetilde u_n^\ell$ solves~(NS) and is bounded in~$\widetilde{L^{2}}(\R^+; \dot B^{\frac23 ,\frac13}_{3,1})$ since that holds for
  the right-hand side of~(\ref{explicitformjlarge})  by direct inspection.
    By Theorem~\ref{globalaniso} this implies that~$\widetilde u_n^\ell$ is bounded in particular in~$\widetilde{L^{\infty}}(\R^+; \dot B^{-\frac13 ,\frac13}_{3,1})$, which proves the result for~$\widetilde R_n^\ell $ again inspecting   the formula~(\ref{explicitformjlarge}) giving~$ \widetilde u_n^\ell - \widetilde R_n^\ell $ and recalling that~$\eta_n^\ell /\delta_n^\ell \to 0$ as~$n$ goes to infinity.

\medskip
\noindent To conclude suppose that~$\tilde  \phi^{h,\ell}(\cdot,z_3) \equiv \phi^{\kappa(\ell)}(\cdot,z_3) \equiv 0$ for some~$z_3 \in \R$. Then by construction of~$\phi_\alpha^\ell$ in~(\ref{rlalpha}) and that of~$\tilde  U^{h,\ell} $ recalled above, the result follows for~$\tilde  U^{h,\ell} (t,\cdot,z_3)$. For~$U_n^{\kappa(\ell)}(t,\cdot,z_3)$ we get the result from Proposition 3.2 of~\cite{cgz}.

\medskip
\noindent
 Lemma~{\rm \ref{globalfirstprofile}} is proved.
  \end{proof}

 \bigskip
 \noindent
 $\bullet$ $ $ Now let us
 consider~$\displaystyle  \Lambda^n_{\boldsymbol{\varepsilon^{\ell}} ,\boldsymbol{\gamma^{\ell}},\boldsymbol{x^{\ell}} } \left( -\frac{ \varepsilon_n^{\ell} }{\gamma_n^{\ell} }
 (\nabla_h \Delta_h^{-1} \partial_3  \phi_\alpha^{\ell}  ) (x)
,
  \phi_\alpha^{\ell} (x)
\right)$, when~$\ell \notin {\mathcal K}(\infty)$.
 \begin{lem}\label{globalsecondprofile}
Assume~$\ell \notin {\mathcal K}(\infty)$.
Then there is~$  L_0$, independent of~$n$ such that the following result holds. For any~$\ell$ and for~$n$ large enough,~$\displaystyle  \Lambda^n_{\boldsymbol{\varepsilon^{\ell}} ,\boldsymbol{\gamma^{\ell}},\boldsymbol{x^{\ell}} } \left( -\frac{ \varepsilon_n^{\ell} }{\gamma_n^{\ell} }
 (\nabla_h \Delta_h^{-1} \partial_3  \phi_\alpha^{\ell}  ) (x)
,
  \phi_\alpha^{\ell} (x)
\right)$  belongs to~$\cG$ and  the associate solution~$   u^\ell_n$ to~{\rm(NS)} enjoys the following properties.

\medskip
\noindent $\bullet$ $ $  For every~$\ell \geq   L_0$, $\alpha \in (0,1)$ and~$n \in \N$ large enough,
\begin{equation}\label{smallXpsecond}
 \|  u^\ell_n  \|_{ {\mathcal S_{3,1}}} \leq 2 \|   \phi_\alpha^{\ell} \|_{\dot B^{-\frac13,\frac13}_{3,1}} \leq 2 c_0 \, .
\end{equation}

\noindent $\bullet$ $ $  For every~$\ell \in \N$,  $\alpha \in (0,1)$ and~$n$ large enough, the sequence~$ u^\ell_n$ is uniformly bounded in the space~$\widetilde{L^{\infty}}(\R^+; \dot B^{-\frac13 ,\frac13}_{3,1})\cap L^1(\R^+ ; \dot B^{\frac53, \frac13}_{3,1} \cap \dot B^{\frac23, \frac43}_{3,1})$ and satisfies
\begin{equation}\label{explicitformjlargesecond}
\begin{aligned}
  u^\ell_n   =  \widetilde \Lambda^n_{\boldsymbol{\varepsilon^{\ell}} ,\boldsymbol{\gamma^{\ell}},\boldsymbol{x^{\ell}} } \left( \frac{ \varepsilon_n^{\ell} }{\gamma_n^{\ell} }
  U_n^{\ell,h}
,
  U_n^{\ell,3}
 \right)
+   R_n^{\ell}  \quad \mbox{where}  \\
   R_n^{\ell}  \to 0 \: \: \mbox {in} \: \:  \widetilde{L^{2}}(\R^+; \dot B^{\frac23 ,\frac13}_{3,1}) \cap L^1(\R^+ ; \dot B^{\frac53, \frac13}_{3,1} \cap \dot B^{\frac23, \frac43}_{3,1}) \, ,\quad  n \to \infty  \, ,
\end{aligned}
\end{equation}
and all the properties stated in Lemma~\ref{globalfirstprofile} hold.
 \end{lem}
%%%
\begin{proof} [Proof of Lemma~{\rm \ref{globalsecondprofile}}]
The proof follows
the lines of the proof of Lemma~{\rm \ref{globalfirstprofile}}, and is in fact easier. One first uses the stability property~{\rm (\ref{orthonorms})} to obtain the existence of~$L_0$ such that
for all~$\ell \geq L_0$, for each~$\alpha \in (0,1)$ and for~$n$ large enough,
$$
 \| (\nabla_h \Delta_h^{-1}\partial_3  \phi_\alpha^{\ell},   \phi_\alpha^{\ell} )\|_{\dot B^{0,\frac12}_{2,1}} \leq c_0
$$
and  Theorem~{\rm \ref{globalaniso}} applies.
  Then we notice again that  by rescaling and translation it is enough to consider the vector field~$\displaystyle  \Lambda^n_{\boldsymbol{1} ,\boldsymbol{
\frac{\gamma^{\ell}}{\varepsilon^{\ell}}},\boldsymbol{0} } \left( -\frac{ \varepsilon_n^{\ell} }{\gamma_n^{\ell} }
 (\nabla_h \Delta_h^{-1} \partial_3  \phi_\alpha^{\ell}  ) (x)
,
  \phi_\alpha^{\ell} (x)
\right)
 $, and again~\cite{cg3} gives the result (recalling that~$\varepsilon_n^{\ell} / \gamma_n^{\ell} $ goes to zero by Proposition~{\rm \ref{prop:decompositiondata})}.
Compared with the proof of Lemma~{\rm \ref{globalfirstprofile}}, in this case the profile~$U^\ell_n$ is simply a solution to the heat equation in~$\R^3$ with viscosity~$({\varepsilon^{\ell}_n}/{\gamma^{\ell}_n})^2$ in the third direction (see~\cite{cg3} system~$(T^\e_{\underline v})$, with~$\underline v \equiv 0$ and~$\e = {\varepsilon^{\ell}_n}/{\gamma^{\ell}_n}$).  The lemma is proved.
 \end{proof}

 \noindent
 In the following we define, with the notation of Lemmas~\ref{globalfirstprofile} and~\ref{globalsecondprofile},
 \begin{equation}\label{defmathcalU}
 \begin{aligned}
 {\mathcal U}_n^L :=  \sum_{1 \leq \ell \leq L} \tilde u_n^\ell + \sum_{1 \leq \kappa(\ell) \leq L \atop \ell >L} \tilde u_n^\ell +   \sum_{\ell = 1}^L u_n^\ell \,  ,
 \quad \mbox{and} \\
  {\mathcal R}_n^L :=  \sum_{1 \leq \ell \leq L} \tilde R_n^\ell  \, +  \sum_{1 \leq \kappa(\ell) \leq L \atop \ell >L}  \tilde R_n^\ell  \, + \sum_{\ell = 1}^{  L}   R_n^\ell \, ,
 \end{aligned}
 \end{equation}
and we recall that
 \begin{equation}\label{RnJto0}
 \forall L\, , \quad \lim_{n \to \infty} \|{\mathcal R}_n^L\|_{\widetilde{L^{2}}(\R^+; \dot B^{\frac23 ,\frac13}_{3,1})\cap L^1(\R^+ ; \dot B^{\frac53, \frac13}_{3,1} \cap \dot B^{\frac23, \frac43}_{3,1})}¬†  = 0   \, .
 \end{equation}

 \bigskip
 \noindent
 $\bullet$ $ $ Finally we propagate all the
remaining terms in~(\ref{newformulationdecompositiondata}) by the heat equation: we define
\begin{equation}\label{defmathcalV}
{\mathcal V}_n^L := \rho_n^L+\Psi_n^L
  \end{equation}
with
$$
\begin{aligned}
 \Psi_n^L (t):=e^{t \Delta}  \Big( (\tilde \psi_n^{h,L}
  -\nabla_h \Delta_h^{-1} \partial_3  \psi_n^L
  , \psi_n^{L}  ) -   \sum_{  \ell >L \atop
 \ell \in  {\mathcal K}(  L)   }  \Lambda^n_{\boldsymbol{\varepsilon^{\ell}} ,\boldsymbol{\gamma^{\ell}},\boldsymbol{x^{\ell}} } \big( -\frac{ \varepsilon_n^{\ell} }{\gamma_n^{\ell} }
 \nabla_h \Delta_h^{-1} \partial_3  \phi^{\ell}      ,  \phi^{\ell}  \big)
 \\
%- \sum_{\ell= 1}^L\frac{ \varepsilon_n^{\ell} }{\gamma_n^{\ell} } \Lambda^n_{\boldsymbol{\varepsilon^{\ell}} ,\boldsymbol{\gamma^{\ell}},\boldsymbol{x^{\ell}} }
 %(\nabla_h \Delta_h^{-1} \partial_3  r^{\ell}_\alpha  ,  0)
 - \sum_{
\ell >L \atop
 1 \leq \kappa(\ell) \leq L   }  \!  \!   \!  \!  \! \Lambda^n_{\boldsymbol{\eta^{\ell}} ,\boldsymbol{\delta^{\ell}},\boldsymbol{\widetilde x^{\ell}} } \big( \tilde  \phi_\alpha^{h, \ell } 
 %+\tilde  r_\alpha^{h,\kappa^{-1}(\ell)} 
   ,  0 \big)    \Big)
\end{aligned}
$$
and
$$
\begin{aligned}
\rho_n^L (t) := e^{t \Delta} \Big(
 \sum_{\ell  = 1}^{   L}    \Lambda^n_{\boldsymbol{\eta^{\ell}} ,\boldsymbol{\delta^{\ell}},\boldsymbol{\tilde x^{\ell}} }  (\tilde  r_\alpha^{h,\ell}
   -\frac{ \eta_n^{\ell} }{\delta_n^{\ell} }
 \nabla_h \Delta_h^{-1} \partial_3 r_\alpha^{\kappa(\ell)} ,  r_\alpha^{\kappa(\ell)}
 ) \\
 +   \sum_{\kappa( \ell ) =1 \atop
\kappa( \ell ) \in  {\mathcal K}(  \infty) \setminus     {\mathcal K}(  L)  }^L     \Lambda^n_{\boldsymbol{\eta^{\ell}} ,\boldsymbol{\delta^{\ell}},\boldsymbol{\tilde x^{\ell}} }( \tilde  r_\alpha^{h,\ell}-\frac{ \eta_n^{\ell} }{\delta_n^{\ell} }  \nabla_h \Delta_h^{-1} \partial_3  r_\alpha^{\kappa(\ell)})     ,  r_\alpha^{\kappa(\ell)} ) \\
+ \sum_{
    \ell = 1\atop
 \ell \notin  {\mathcal K}(  \infty)  }^{L } \Lambda^n_{\boldsymbol{\varepsilon^{\ell}} ,\boldsymbol{\gamma^{\ell}},\boldsymbol{x^{\ell}} }  
    ( -\frac{ \varepsilon_n^{\ell} }{\gamma_n^{\ell} }
\nabla_h \Delta_h^{-1} \partial_3 
 r_\alpha^{\ell} , r_\alpha^{\ell}   ) 
 -  \sum_{ 
  \ell >L \atop
1 \leq \kappa( \ell ) \leq L 
 }  \Lambda^n_{\boldsymbol{\eta^{\ell}} ,\boldsymbol{\delta^{\ell}},\boldsymbol{\tilde x^{\ell}} }  (\tilde  r_\alpha^{h,\ell}   ,  0)
 \Big)
\end{aligned}
$$
 We notice that by~(\ref{rlalpha})
 \begin{equation} \label{ytilde}
\begin{aligned}
& \forall L \in \N \, ,
  \quad \limsup_{n\to \infty} \|\rho_n^L\|_{{\mathcal S}_{3,1}} \leq C(L) \, \alpha \, ,   \\
 &\mbox{and} \quad \limsup_{n\to \infty} \big( \|\Psi_n^{L,h}\|_{{\mathcal S}_{3,1}+\widetilde {\mathcal S}_{3,1} } + \|\Psi_n^{L,3}\|_{{\mathcal S}_{3,1}  } \big) \to 0 \, , \quad L \to \infty \,   \,   \mbox{uniformly in} \, \alpha \,,
\end{aligned}
  \end{equation}
   {where} $\displaystyle \widetilde {\mathcal S}_{3,1}:=      \bigcap_{r=1}^\infty \,
      \bigcap_{\sigma=0}^\frac2r \,      \widetilde{L^r}(\R^+;\dot B^{ \frac23+\sigma ,\frac2r - \sigma-\frac23 }_{3,1}) .$   The presence of that space
      is due to  terms of the type~$ \nabla_h \Delta_h^{-1} \partial_3  \phi^{\ell}   $ and those bounds are due to~(\ref{smallremainderpsi}) as well as the stability property~(\ref{orthonorms}) and the fact that~$\varepsilon^\ell_n / \gamma^\ell_n \to 0$. In particular
     \begin{equation}\label{smallpsiY}
       \limsup_{n\to \infty} \big( \|\Psi_n^{L}\|_{L^1(\R^+; \dot B^{\frac53  ,\frac13 }_{3,1} \cap\dot B^{\frac23  ,\frac43 }_{3,1})\cap \widetilde{L^2}(\R^+; \dot B^{\frac23,\frac13}_{3,1} ) } \big) \to 0 \, , \quad L \to \infty \,   \,   \mbox{uniformly in} \,  \, \alpha \,.
      \end{equation}
         %%%%%%%%%%%%%%%%%%%%%%%%%%%%%%%%%%%%%%%%%%%%%

 \section{Global regularity for the profiles superposition}\label{superpositionglobal}

  Now we need to superpose each of the solutions constructed in the previous section, and check that the superposition is indeed a good approximate solution. This will prove Theorem~{\rm \ref{mainresult}}, and at the end of this section we shall show how the methods developed here give easily Corollaries~{\rm \ref{corthm1}} and~{\rm \ref{corthm}}.

\subsection{Statement of the   superposition result and main steps of its proof}
 The main result is the following, where we use the notation of the previous section.
  \begin{prop}\label{decompositionsolution}
 For~$n$ and~$L$ large enough, $\alpha$ small enough and up to an extraction, we have
 \begin{equation}\label{writeun}
  u_n = u +{\mathcal U}_n^L+ {\mathcal V}_n^L + w_n^{L}¬†\, ,
 \end{equation}
 where~$w_n^L $ belongs to~${\mathcal S}_{3,1}$ with
 $
 \displaystyle \lim_{\alpha \to 0} (\limsup_{n \to \infty} \|w_n^{L}\|_{{\mathcal S}_{3,1}} ) \to 0$ as~$ L \to \infty.
 $
\end{prop}
 \begin{rmk}{\rm
The choice of the function space~${\mathcal S}_{3,1}$ in the statement of Proposition~\ref{decompositionsolution} is for convenience, we have not tried to optimize on the integrability index here and other spaces would certainly do as well.
}
 \end{rmk}
 \begin{rmk}{\rm
This proposition proves Theorem~{\rm \ref{mainresult}}. Indeed the sequence~$(u_n)$ belongs in particular to the space~$\widetilde{L^2}(\R^+;\dot B^{\frac23,\frac13}_{3,1})$, since the results of the previous section show   that this is the case   for all the terms in the right-hand side of~({\rm\ref{writeun}}). But we know from Theorem~\ref{globalaniso} that this norm controls the equation so the result follows.
}
 \end{rmk}
 \begin{proof}[Proof of Proposition~{\rm \ref{decompositionsolution}}]
  Let~$u_n$ be the solution of~(NS) associated with the data~$u_{0,n}$, which a priori has a finite life span~$T_n^*$, and define
   $$
w_n^{L}  := u_n  - G_n^L \quad \mbox{with}  \quad G_n^L := u + F_n^L \quad \mbox{and}  \quad   F_n^L  :=  {\mathcal U}_n^L + {\mathcal V}_n^L \, .
$$
 The vector field~$w_n^{L}$ satisfies 
  $$
 \partial_t w_n ^{L}+ {\mathbb P} (w_n^{L} \cdot \nabla w_n^{L} + G_n^L \cdot \nabla w_n^{L} + w_n ^{L}\cdot \nabla  G_n^L ) - \Delta w_n^{L}
  = -  {\mathbb P}  Z_n^L   \, , \quad \mbox{div} \: w_n^{L} = 0
  $$
with initial data~$w_{n |t = 0}^{L} = 0$, and where, recalling the definitions of~$ {\mathcal U}_n^L $ and~${\mathcal V}_n^L $ in~(\ref{defmathcalU}) and~(\ref{defmathcalV}) respectively,
$$
\begin{aligned}
Z_n^L  := \sum_{\ell \neq k}  u_n^\ell \cdot \nabla u_n^k +
 \sum_{\ell \neq k}
\tilde u_n^\ell ({\mathbf 1}_{1 \leq \ell \leq L} + {\mathbf 1}_{1 \leq \kappa(\ell) \leq L \atop \ell >L} ) \cdot \nabla  \tilde u_n^k ({\mathbf 1}_{1 \leq k \leq L} + {\mathbf 1}_{1 \leq \kappa(k) \leq L \atop \ell >L} )\\
{}  +  \sum_{\ell \neq k}
\big(\tilde u_n^\ell  ({\mathbf 1}_{1 \leq \ell \leq L} + {\mathbf 1}_{1 \leq \kappa(\ell) \leq L \atop \ell >L} )  \cdot \nabla u_n^k + u_n^\ell \cdot \nabla \tilde u_n^k ({\mathbf 1}_{1 \leq k \leq L} + {\mathbf 1}_{1 \leq \kappa(k) \leq L \atop \ell >L} )  \big) \\
{} + u \cdot \nabla  F_n^L  +  F_n^L \cdot \nabla u  +  {\mathcal U}_n^L\cdot  \nabla {\mathcal V}_n^L + {\mathcal V}_n^L \cdot  \nabla {\mathcal U}_n^L +
  {\mathcal V}_n^L \cdot  \nabla {\mathcal V}_n^L  \, .
\end{aligned}
$$
   The proposition follows from the two following lemmas.
      \begin{lem}\label{lemdrift}
      Define~${\mathcal Y}:=    {L^2}(\R^+; \dot B^{\frac23,\frac13}_{3,1} ) \cap L^1(\R^+; \dot B^{\frac53  ,\frac13 }_{3,1}\cap\dot B^{\frac23  ,\frac43 }_{3,1} ) $. With the notation of Lemmas~\ref{globalfirstprofile} and~\ref{globalsecondprofile}, there is a constant~$K$ (depending on~$L_0, \widetilde L_0$ and  bounds on~$u_0$, $(u_n)$ and~$u$) such that
      one can decompose~$G_n^L =    {\mathcal G}_n^{L,1} +  {\mathcal G}_n^{L,2}   $,
      with the following properties: for each~$L¬†\in \N$ and each~$\alpha \in (0,1)$ there is~$N(L,\alpha)$ such that
      $$
      \|{\mathcal G}_n^{L,1} \|_{ \mathcal Y}   \leq K \quad  \mbox{ for} \, \,  \,  n \geq N(L,\alpha)   \,,
      $$
     while  for all~$L \in \N$ there is~$\alpha _0>0$ such that
      $$
      \forall \,  0< \alpha \leq \alpha_0\,  , \quad    \|{\mathcal G}_n^{L,2} \|_{\mathcal Y} \leq K\quad \mbox{ uniformly in } \, n \, .
      $$
          \end{lem}
       \begin{lem}\label{lemforce}
Define
$$
{\mathcal X}:=  L^1(\R^+; \dot B^{-\frac13  ,\frac13 }_{3,1})   +  \widetilde{L^2} (\R^+;\dot B^{-\frac13,-\frac23 }_{3,1}  ) \cap
  L^1(\R^+; \dot B^{\frac23,-\frac23 }_{3,1} )\, .
  $$  We can write~$Z_n^L = {\mathcal Z}_n^{L,1} + {\mathcal Z}_n^{L,2}+{\mathcal Z}_n^{L,3}$ with
\begin{eqnarray}
  \limsup_{L \to \infty}    \| {\mathcal Z}_n^{L,1}\|_{  {\mathcal X}}    = 0  \,  \, \mbox{uniformly in} \, \, n, \alpha \, , \label{calhn1}\\
 \forall L, \quad \limsup_{\alpha \to 0}    \| {\mathcal Z}_n^{L,2}\|_{ {\mathcal X}}    = 0 \,   \, \mbox{uniformly in} \, \, n \, , \label{calhn2}\\
 \mbox{and} \quad     \forall L \, ,  \forall \alpha \, , \quad \limsup_{n \to \infty}    \| {\mathcal Z}_n^{L,3}\|_{ {\mathcal X}}    = 0  \, .\label{calhn3}
\end{eqnarray}
      \end{lem}
        \noindent Assume indeed for the time being that those two lemmas are true. Then we start by choosing~$L$ large enough
        so that uniformly in~$\alpha$ and~$N$ one has
        \begin{equation}\label{1w}
    \| {\mathcal Z}_n^{L,1}\|_{ {\mathcal X}}  \leq \frac{c_0}{12} \exp  \big(-  2K c_0^{-1}\big) \quad   \mbox{uniformly in} \, \, n, \alpha
\end{equation}
with the notation of Theorem~{\rm \ref{globalanisoperturbed}} stated and proved in Appendix~{\rm \ref{globalsmallaniso}}, and Lemma~\ref{lemforce}.
Then now that~$L$ is fixed we choose~$\alpha \in (0,\alpha_0)$ small enough so that
\begin{equation}\label{2w}
 \| {\mathcal Z}_n^{L,2}\|_{  {\mathcal X}}  \leq \frac{c_0}{12} \exp  \big(-  2K c_0^{-1}\big) \quad   \mbox{uniformly in} \, \, n
\end{equation}
and
$$
\| {\mathcal G}_n^{L,2}\|_{  {\mathcal Y}} \leq K \quad   \mbox{uniformly in} \, \, n \, ,
$$
with the notation of Lemma~\ref{lemdrift}.
Finally now that~$L$ and~$\alpha$ are fixed we take~$N_0  \geq N(L, \alpha)$ so that for all~$n\geq N_0$,
\begin{equation}\label{3w}
  \| {\mathcal Z}_n^{L,3}\|_{ {\mathcal X}}    \leq \frac{c_0}{12} \exp \big(- 2K c_0^{-1}\big)
\end{equation}
and
$$
\| {\mathcal G}_n^{L,1}\|_{  {\mathcal Y}} \leq K \, .
$$
It then suffices to apply Theorem~{\rm \ref{globalanisoperturbed}} in Appendix~\ref{globalsmallaniso}
with~$U = {  G}_n^{L}$, $F = {  Z}_n^{L}$ and data~$u_0 \equiv 0$, noticing that~${\mathcal X} = {\mathcal X}_{3,1} $ and~${\mathcal Y} \subset {\mathcal Y}_{3,1} $. The result follows immediately: we get that~$w_n^L$ belongs to~${\mathcal S}_{3,1}$, and the fact that~$
 \displaystyle \lim_{\alpha \to 0} (\limsup_{n \to \infty} \|w_n^{L}\|_{{\mathcal S}_{3,1}} ) \to 0$ as~$ L \to \infty $ is due to the fact that one can choose the bounds in~(\ref{1w})-(\ref{3w}) as small as one want, provided~$L$ and~$n$ are large enough, and~$\alpha$ is small enough.
   \end{proof}

    \medskip
    \noindent
    The two coming paragraphs are devoted
    to the proofs of Lemmas~{\rm \ref{lemdrift}} and~{\rm \ref{lemforce}}, thus achieving the proof of Theorem~{\rm \ref{mainresult}}. The final paragraph of this section contains the proofs of  Corollaries~{\rm \ref{corthm1}} and~{\rm \ref{corthm}}.

    \subsection{Study of the drift term~$G_n^L$}
          \begin{proof}[Proof of Lemma~{\rm \ref{lemdrift}}]
               Recall that  $     G_n^L = u + F_n^L =  u +{\mathcal U}_n^L + {\mathcal V}_n^L$ with the notation of Section~\ref{globalprofile},
so since we know that~$u$ belongs to~${\mathcal S}_{2,1}$, which embeds continuously in~$\mathcal Y $,
and~$u  $ depends neither on~$L$, on~$\alpha$ nor on~$n$, we need to study~$F_n^L$.
According to
  Lemmas~{\rm \ref{globalfirstprofile}} and~{\rm \ref{globalsecondprofile}} and recalling the notation~{\rm (\ref{defmathcalU})},
  we can split~$ F_n^L=  {\mathcal U}_n^L + {\mathcal V}_n^L$ into~$F_n^L :=
   F_n^{L, 1} +  F_n^{L, 2}   +   {\mathcal V}_n^L $, with
\begin{equation}\label{decomposeF}
   F_n^{L,1}  := \sum_{1 \leq \ell \leq L\atop \eta_n^\ell/ \delta_n^\ell \to 0} \tilde u_n^\ell + \sum_{1 \leq \kappa(\ell) \leq L \atop {\ell >L\atop \eta_n^\ell/ \delta_n^\ell \to 0}} \tilde u_n^\ell + \sum_{\ell = 1 }^{    L }  u_n^\ell \quad
   \mbox{and} \quad  F_n^{L,2}  :=\sum_{\ell = 1 \atop \eta_n^\ell / \delta_n ^\ell \to \infty}^{    L } \tilde u_n^\ell       \, .
\end{equation}
The result~(\ref{smallkato})     deals with~$ F_n^{L,2} $, since according to~(\ref{smallkato}), $\tilde u_n^\ell   $ goes to zero in~$\mathcal Y$ for each~$\ell$ as~$n$ goes to infinity. So that term   is incorporated in the term~${\mathcal G}_n^{L,1}$.

\noindent Now let us consider~$  F_n^{L, 1}$. We can decompose the sum again into several pieces, writing with the notation of Lemmas~\ref{globalfirstprofile} and~\ref{globalsecondprofile}, for all~$L > \max (L_0, \tilde L_0)$,
$$
\begin{aligned}
 \sum_{\ell = 1 }^{    L }  u_n^\ell & =  \sum_{\ell = 1 }^{    L_0 }  u_n^\ell  + \sum_{\ell = L_0+1 }^{    L }  u_n^\ell   \, , \\
 \sum_{1 \leq \ell \leq L\atop \eta_n^\ell/ \delta_n^\ell \to 0} \tilde u_n^\ell &=
  \sum_{1 \leq \ell \leq \tilde L_0\atop \eta_n^\ell/ \delta_n^\ell \to 0} \tilde u_n^\ell +
   \sum_{\tilde L_0< \ell \leq L\atop{ 1 \leq \kappa(\ell) \leq \tilde L_0 \atop\eta_n^\ell/ \delta_n^\ell \to 0}} \tilde u_n^\ell +
    \sum_{\tilde L_0< \ell \leq L\atop{\tilde  L_0< \kappa(\ell)  \atop\eta_n^\ell/ \delta_n^\ell \to 0}} \tilde u_n^\ell \, , \\
     \mbox{and}  \quad\sum_{1 \leq \kappa(\ell) \leq L \atop {\ell >L\atop \eta_n^\ell/ \delta_n^\ell \to 0}} \tilde u_n^\ell& = \sum_{1 \leq \kappa(\ell) \leq \tilde  L_0 \atop {\ell >L\atop \eta_n^\ell/ \delta_n^\ell \to 0}} \tilde u_n^\ell + \sum_{\tilde L_0< \kappa(\ell) \leq L \atop {\ell >L\atop \eta_n^\ell/ \delta_n^\ell \to 0}} \tilde u_n^\ell  \, .
\end{aligned}
 $$
 In all three right-hand-sides, the easiest term to deal with is the last one: indeed
 we can write
 $$
 \Big\|   \sum_{\ell = L_0+1 }^{    L } \!u_n^\ell   +\sum_{\tilde L_0< \ell \leq L\atop{ \tilde L_0<\kappa(\ell)  \atop\eta_n^\ell/ \delta_n^\ell \to 0}}  \tilde u_n^\ell +  \sum_{\tilde L_0< \kappa(\ell) \leq L \atop {\ell >L\atop \eta_n^\ell/ \delta_n^\ell \to 0}}  \tilde u_n^\ell   \Big\|_{ \mathcal Y }  \lesssim  \sum_{\ell = L_0+1 }^{    L }  \|u_n^\ell \|_{ \mathcal Y}
+ \sum_{\tilde L_0< \ell  \atop \tilde L_0< \kappa(\ell)  } \| \tilde u_n^\ell \|_{ \mathcal Y} \, .
 $$
  Then  by~(\ref{smallXp}) and~(\ref{smallXpsecond})
 we infer that as soon as~$n$ is large enough (depending on the choice of~$L$ and~$\alpha$)
 $$
   \Big\|   \sum_{\ell = L_0+1 }^{    L } \!u_n^\ell   +\sum_{\tilde L_0< \ell \lesssim L\atop{ \tilde L_0<\kappa(\ell)  \atop\eta_n^\ell/ \delta_n^\ell \to 0}}  \tilde u_n^\ell +  \sum_{\tilde L_0< \kappa(\ell) \lesssim L \atop {\ell >L\atop \eta_n^\ell/ \delta_n^\ell \to 0}}  \tilde u_n^\ell   \Big\|_{ \mathcal Y }  \leq  \sum_{L_0< \ell  }   \|\phi_\alpha^\ell \|_{  \dot B^{-\frac13,\frac13 }_{3,1} }
+    \sum_{\tilde L_0< \ell     }     \|\tilde \phi_\alpha^{h,\ell} \|_{  \dot B^{-\frac13,\frac13 }_{3,1}  }
 $$
 and the conclusion comes from the embedding of~$ {\mathcal B}^1_q$ into~$ \dot B^{-\frac13,\frac13 }_{3,1} $ along with the stability property~(\ref{orthonorms}): for~$n \geq N(L, \alpha)$
 $$
    \Big\|   \sum_{\ell = L_0+1 }^{    L } \!u_n^\ell   +\sum_{\tilde L_0<\ell \leq L\atop{ \tilde L_0<\kappa(\ell)  \atop\eta_n^\ell/ \delta_n^\ell \to 0}}  \tilde u_n^\ell +  \sum_{\tilde L_0<  \kappa(\ell) \leq L \atop {\ell >L\atop \eta_n^\ell/ \delta_n^\ell \to 0}}  \tilde u_n^\ell   \Big\|_{ \mathcal Y }  \lesssim \sum_{L_0<\ell  }   \|\phi_\alpha^\ell \|_{ {\mathcal B}^1_q }
+   \sum_{\tilde L_0< \ell    }   \| \tilde \phi_\alpha^{h,\ell} \|_{ {\mathcal B}^1_q } \leq C \, .
 $$
 So~$\displaystyle   \sum_{\ell = L_0+1 }^{    L } \!u_n^\ell   +\sum_{\tilde L_0<\ell \leq L\atop{ \tilde L_0<\kappa(\ell)  \atop\eta_n^\ell/ \delta_n^\ell \to 0}}  \tilde u_n^\ell +  \sum_{\tilde L_0<  \kappa(\ell) \leq L \atop {\ell >L\atop \eta_n^\ell/ \delta_n^\ell \to 0}}  \tilde u_n^\ell $ is of the type~$ {\mathcal G}_n^{L,1}$.

 \noindent
 Now let us estimate~$\displaystyle \sum_{\ell = 1 }^{    L_0 }  u_n^\ell  $ and~$\displaystyle   \sum_{1 \leq \ell \leq \tilde L_0\atop \eta_n^\ell/ \delta_n^\ell \to 0} \tilde u_n^\ell $ . There is of course no uniformity problem in~$L$ and we simply
 use the uniform bound in~${\mathcal Y}$ provided in Lemmas~\ref{globalfirstprofile} and~\ref{globalsecondprofile}.
The terms~$\displaystyle  \sum_{\tilde L_0+1 \leq \ell \leq L\atop{ 1 \leq \kappa(\ell) \leq L_0 \atop\eta_n^\ell/ \delta_n^\ell \to 0}} \tilde u_n^\ell $ and~$\displaystyle\sum_{1 \leq \kappa(\ell) \leq L_0 \atop {\ell >L\atop \eta_n^\ell/ \delta_n^\ell \to 0}} \tilde u_n^\ell$ are dealt with similarly and all those three terms are also of the type~$ {\mathcal G}_n^{L,1}$.
Choosing~$ {\mathcal G}_n^{L,2}:=  {\mathcal V}_n^L
$
and using~(\ref{ytilde}) and~(\ref{smallpsiY}) concludes the proof of Lemma~{\rm \ref{lemdrift}}.
  \end{proof}

  \smallskip

  \begin{rmk}\label{uniformbound}{\rm
This argument shows that~${\mathcal U}_n^L$ is uniformly bounded in the space~${\mathcal S}_{3,1}   $.}
\end{rmk}

  \smallskip

\begin{rmk}\label{p=q=1}{\rm
It is important to have chosen the initial data bounded in a space of the type~$\dot B^{s,s'}_{p,q}$ with~$p=1 > q$ (hence in particular with~$p=1 = q$ by embedding), as it enables us to prove easily  the uniform bound on~$ F_n^{L,1}$.
As seen for instance  in~\cite{gkp}, it is indeed possible to prove such a bound when~$p = q$ and it is not clear how to prove it in the general case,  when~$p \neq q$. Then it is very natural to pick~$q\leq1$ as explained in the introduction in order to have a good Cauchy theory for the Navier-Stokes equations in anisotropic spaces, and finally the choice~$q<1$ implies by interpolation that the remainders are small precisely in a space where the Cauchy theory for (NS) is satisfactory (namely~$q=1$).}
\end{rmk}
       \subsection{Study of the forcing term}\label{studyforcingterm}
               \begin{proof}[Proof of Lemma~{\rm \ref{lemforce}}]
      We recall that
$$
\begin{aligned}
Z_n^L  := \sum_{\ell \neq k}  u_n^\ell \cdot \nabla u_n^k +
 \sum_{\ell \neq k}
\tilde u_n^\ell ({\mathbf 1}_{1 \leq \ell \leq L} + {\mathbf 1}_{1 \leq \kappa(\ell) \leq L \atop \ell >L} ) \cdot \nabla  \tilde u_n^k ({\mathbf 1}_{1 \leq k \leq L} + {\mathbf 1}_{1 \leq \kappa(k) \leq L \atop \ell >L} )\\
{}  +  \sum_{\ell \neq k}
\big(\tilde u_n^\ell  ({\mathbf 1}_{1 \leq \ell \leq L} + {\mathbf 1}_{1 \leq \kappa(\ell) \leq L \atop \ell >L} )  \cdot \nabla u_n^k + u_n^\ell \cdot \nabla \tilde u_n^k ({\mathbf 1}_{1 \leq k \leq L} + {\mathbf 1}_{1 \leq \kappa(k) \leq L \atop \ell >L} )  \big) \\
{} + u \cdot \nabla  F_n^L  +  F_n^L \cdot \nabla u  +  {\mathcal U}_n^L\cdot  \nabla {\mathcal V}_n^L + {\mathcal V}_n^L \cdot  \nabla {\mathcal U}_n^L +
  {\mathcal V}_n^L \cdot  \nabla {\mathcal V}_n^L  \, .
\end{aligned}
$$
We define
$$
\begin{aligned}
H_n^{L,1}  := \sum_{\ell \neq k}  u_n^\ell \cdot \nabla u_n^k +
 \sum_{\ell \neq k}
\tilde u_n^\ell ({\mathbf 1}_{1 \leq \ell \leq L} + {\mathbf 1}_{1 \leq \kappa(\ell) \leq L \atop \ell >L} ) \cdot \nabla  \tilde u_n^k ({\mathbf 1}_{1 \leq k \leq L} + {\mathbf 1}_{1 \leq \kappa(k) \leq L \atop \ell >L} )\\
{}  +  \sum_{\ell \neq k}
\big(\tilde u_n^\ell  ({\mathbf 1}_{1 \leq \ell \leq L} + {\mathbf 1}_{1 \leq \kappa(\ell) \leq L \atop \ell >L} )  \cdot \nabla u_n^k + u_n^\ell \cdot \nabla \tilde u_n^k ({\mathbf 1}_{1 \leq k \leq L} + {\mathbf 1}_{1 \leq \kappa(k) \leq L \atop \ell >L} )  \big)  \, ,
\end{aligned}
$$
$$
   H_n^{L,2}:= {\mathcal U}_n^L\cdot  \nabla {\mathcal V}_n^L + {\mathcal V}_n^L \cdot  \nabla {\mathcal U}_n^L +
  {\mathcal V}_n^L \cdot  \nabla {\mathcal V}_n^L  \quad \mbox{and}\quad
 H_n^{L,3} :=  u \cdot \nabla   F_n^{L}+ F_n^{L}
\cdot \nabla u  \, .
 $$
Let start by discussing~$H_n^{L,1}$.
We shall actually only deal with
$$
    \sum_{1 \leq \ell \neq k \leq L} \tilde u_n^\ell \cdot \nabla\tilde u_n^k=   \sum_{1\leq\ell \neq k\leq L}  \mbox{div} \:  \big ( \tilde u_n^\ell \otimes\tilde u_n^k \big)\, ,
$$
as all the other terms   in~$H_n^{L,1} $ can be dealt with similarly.
Referring to Lemma~{\rm \ref{globalfirstprofile}}, we know that  this term can  in turn be split into two parts, defining
$$
\begin{aligned}
   H_n^{L,1, 1}&:=   \sum_{1 \leq \ell \neq k\leq L \atop \eta_n^\ell / \delta_n ^\ell \to \infty} \mbox{div} \: \big( \tilde u_n^\ell \otimes\tilde u_n^k +\tilde u_n^k \otimes  \tilde u_n^\ell \big) +   \sum_{1 \leq \ell \neq k\leq L  \atop \eta_n^\ell / \delta_n ^\ell +  \eta_n^k / \delta_n ^k \to 0} \mbox{div} \: \Big(  \tilde R_n^\ell \otimes\tilde u_n^k + \tilde u_n^k \otimes  \tilde R_n^\ell\Big)\, , \\
    H_n^{L,1, 2}&:=   \sum_{1 \leq \ell \neq j\leq L \atop \eta_n^\ell / \delta_n ^\ell +  \eta_n^j / \delta_n ^j  \to 0}
 \mbox{div}  \: \Big(\widetilde \Lambda^n_{\boldsymbol{\eta^{\ell}} ,\boldsymbol{\delta^{\ell}},\boldsymbol{\tilde x^{\ell}} }
 \big(  \tilde  U^{h,\ell} + \frac{ \eta_n^{\ell} }{\delta_n^{\ell} }  U_n^{\kappa(\ell),h}, U_n^{\kappa(\ell),3} \big) \\
&{} \quad\quad\quad\quad\quad\quad\quad\quad\quad\quad \otimes  \widetilde \Lambda^n_{\boldsymbol{\eta^{j}} ,\boldsymbol{\delta^{j}},\boldsymbol{\tilde x^{j}} }\big(
  \tilde  U^{h,j} + \frac{ \eta_n^{j} }{\delta_n^{j} } U_n^{\kappa(j),h},U_n^{\kappa(j),3} \big)  \Big)
 \, .
\end{aligned}
$$
     The first term $   H_n^{L,1, 1}$ is dealt with using product laws in anisotropic Besov spaces  (see Appendix~{\rm \ref{appendixlp}}).
   On the one hand we have   for any~$j \in \{1,2 \}$, by~(\ref{algebra}),
          \begin{equation}\label{productanisoh}
    \begin{aligned}
\|\partial_j (f g) \|_{\widetilde{L^1}(\R^+; \dot B^{-\frac13,\frac13 }_{3,1} )}
&\lesssim \| f g \|_{\widetilde{L^1}(\R^+;\dot B^{\frac23,\frac13}_{3,1})}
\\
& \lesssim \|f\|_{\widetilde{L^{2}}(\R^+; \dot B^{\frac23,\frac13}_{3,1})}
 \|g\|_{\widetilde{L^{2}}(\R^+; \dot B^{\frac23,\frac13 }_{3,1})} \, ,
\end{aligned}
     \end{equation}
     and on the other hand  estimate~(\ref{quasialgebra}) gives
%        \begin{equation}\label{productaniso31}
%    \begin{aligned}
%\|\partial_3 (f  g) \|_{\widetilde{L^{2}}(\R^+; \dot B^{-\frac13 ,-\frac23 }_{3,1} )}
%&\lesssim \| f  g \|_{\widetilde{L^{2}}(\R^+;\dot B^{-\frac13,\frac13}_{3,1})}
%\\
%& \lesssim \|f \|_{\widetilde{L^{4}}(\R^+; \dot B^{\frac16 ,\frac13}_{3,1})}
% \|g\|_{\widetilde{L^{4}}(\R^+; \dot B^{\frac16,\frac13 }_{3,1})} \, .
%\end{aligned}
%     \end{equation}
       \begin{equation}\label{productaniso31}
    \begin{aligned}
\|\partial_3 (f  g) \|_{\widetilde{L^{2}}(\R^+; \dot B^{-\frac13 ,-\frac23 }_{3,1} )}
&\lesssim \| f  g \|_{\widetilde{L^{2}}(\R^+;\dot B^{-\frac13,\frac13}_{3,1})}
\\
& \lesssim \|f \|_{\widetilde{L^{\infty}}(\R^+; \dot B^{-\frac13 ,\frac13}_{3,1})}
 \|g\|_{\widetilde{L^{2}}(\R^+; \dot B^{\frac23,\frac13 }_{3,1})} \, .
\end{aligned}
     \end{equation}
     and by~(\ref{algebra}) again
          \begin{equation}\label{productaniso3}
    \begin{aligned}
\|\partial_3 (f  g) \|_{\widetilde{L^{1}}(\R^+; \dot B^{\frac23 ,-\frac23 }_{3,1} )}
&\lesssim \| f  g \|_{\widetilde{L^{1}}(\R^+;\dot B^{\frac23,\frac13}_{3,1})}
\\
&\lesssim \|f \|_{\widetilde{L^{2}}(\R^+; \dot B^{\frac23,\frac13}_{3,1})}
 \|g\|_{\widetilde{L^{2}}(\R^+; \dot B^{\frac23,\frac13 }_{3,1})}  \, .
\end{aligned}
     \end{equation}
     So using~(\ref{smallkato}) along with the uniform bounds provided by Lemma~\ref{globalfirstprofile} gives  
  \begin{equation}\label{firstestimateH11}
  \forall L, \quad \lim_{n \to \infty} \Big \|  \sum_{1 \leq \ell \neq k\leq L \atop \eta_n^\ell / \delta_n ^\ell \to \infty } \mbox{div} \: \big( \tilde u_n^\ell \otimes\tilde u_n^k\big) \Big \|_{ {\mathcal X}} = 0 \, .
 \end{equation}
   \smallskip
 \noindent
The   terms~$ \tilde R_n^\ell \otimes \tilde u_n^k $ are dealt with in the same way using  Lemma~\ref{globalfirstprofile}: we find that~$\widetilde H_n^{L,1,1}$ satisfies the bound~(\ref{calhn1}).

  \medskip
 \noindent
%
%    \begin{equation}\label{productaniso31}
%    \begin{aligned}
%\|\partial_3 (f^3 g) \|_{\widetilde{L^{2}}(\R^+; \dot B^{-\frac13 ,-\frac23 }_{3,1} )}
%&\lesssim \| f^3 g \|_{\widetilde{L^{2}}(\R^+;\dot B^{-\frac13,\frac13}_{3,1})}
%\\
%& \lesssim \|f^3\|_{\widetilde{L^{\infty}}(\R^+; \dot B^{-\frac13 ,\frac13}_{3,1})}
% \|g\|_{\widetilde{L^{2}}(\R^+; \dot B^{\frac23,\frac13 }_{3,1})} \, .
%\end{aligned}
%     \end{equation}
    The same product laws  (using the structure of the nonlinear term) enable us to deal with~$ H_n^{L,2}$, recalling that
 $$
   H_n^{L,2}:= {\mathcal U}_n^L\cdot  \nabla {\mathcal V}_n^L + {\mathcal V}_n^L \cdot  \nabla {\mathcal U}_n^L +
  {\mathcal V}_n^L \cdot  \nabla {\mathcal V}_n^L
 $$
 using~(\ref{ytilde})-(\ref{smallpsiY})
 to estimate~${\mathcal V}_n^L$, and Remark~\ref{uniformbound} for~${\mathcal U}_n^L$. To control~$  {\mathcal V}_n^L \cdot  \nabla {\mathcal V}_n^L  $ for instance, we notice that the horizontal component does not belong a priori to~$\widetilde{L^{\infty}}(\R^+; \dot B^{-\frac13 ,\frac13}_{3,1})$ (see~(\ref{ytilde})) but that is not a problem as in~(\ref{productaniso31}), due to the structure of the nonlinear term,
one of the two functions is necessarily a third component, which does belong to~$\widetilde{L^{\infty}}(\R^+; \dot B^{-\frac13 ,\frac13}_{3,1})$. We argue similarly for all the other terms.

  \medskip
 \noindent
Next let us consider the term~$  H_n^{L,1,2}$ and prove it satisfies the bounds~(\ref{calhn2})-(\ref{calhn3}).
  Let us define a typical term
  $$
    U_{n}^{j,\ell}:=  \widetilde \Lambda^n_{\boldsymbol{\eta^{\ell}} ,\boldsymbol{\delta^{\ell}},\boldsymbol{\tilde x^{\ell}} }
  \tilde  U^{h,\ell}  \otimes \widetilde \Lambda^n_{\boldsymbol{\eta^{j}} ,\boldsymbol{\delta^{j}},\boldsymbol{\tilde x^{j}} }
  \tilde  U^{h,j}   \, ,
  $$
and first show that
\begin{equation}\label{unifboundtypical}
\mbox{ div }   U_{n}^{j,\ell} \mbox{ is bounded in }  L^1(\R^+; \dot B^{1, 1 }_{1,1} ) \cap  \widetilde{L^\infty}(\R^+; \dot B^{-1,1}_{1,1})
+ L^1(\R^+;  \dot B^{2, 0 }_{1,1}  )
\cap \widetilde{L^\infty}(\R^+;   \dot B^{0, 0 }_{1,1}  )
 \, .
\end{equation}
This follows from the fact that~$  \tilde  U^{h,\ell} $ belongs to~$L^2(\R^+; \dot B^{2, 1 }_{1,1})
\cap  \widetilde{L^\infty}(\R^+; \dot B^{1, 1 }_{1,1} )
$ (see Lemma~\ref{globalfirstprofile} for that result): we know indeed that~$ \dot B^{2, 1 }_{1,1}$ is an  algebra
and that the product of two functions in~$ \dot B^{1, 1 }_{1,1}$ belongs to~$ \dot B^{0, 1 }_{1,1}$
(see Appendix~\ref{appendixlp}). Since~$\widetilde{L^2}(\R^+; \dot B^{2, 1 }_{1,1})
\cap  \widetilde{L^\infty}(\R^+; \dot B^{1, 1 }_{1,1} )
$ is invariant through the action of~$\widetilde \Lambda^n_{\boldsymbol{\eta^{\ell}} ,\boldsymbol{\delta^{\ell}},\boldsymbol{\tilde x^{\ell}} }$
(see Remark~\ref{tildelambdainvariant})
the result~(\ref{unifboundtypical}) follows. 

\smallskip
\noindent Now let us   prove that~$ U_{n}^{j,\ell} $ goes to zero in~$L^1(\R^+; \dot B^{2, 1 }_{1,1})\cap \widetilde{L^\infty}(\R^+; \dot B^{0,1}_{1,1} ) $, as in~(\ref{calhn2})-(\ref{calhn3}): ~$\mbox{div} \,  U_{n}^{j,\ell} $ will then go to zero in~$ L^1(\R^+; \dot B^{1, 1 }_{1,1} ) \cap  \widetilde{L^\infty}(\R^+; \dot B^{-1,1}_{1,1})
+ L^1(\R^+;  \dot B^{2, 0 }_{1,1}  )
\cap \widetilde{L^\infty}(\R^+;   \dot B^{0, 0 }_{1,1}  )$ which is contained in~${\mathcal X}$.

\smallskip
\noindent Let us start by the~$L^1(\R^+; \dot B^{2, 1 }_{1,1})$ norm. By the equivalent formulation in terms of the heat flow~(\ref{defbesovanisoheat}), we know that~$ \tau^{ -2} \tau'^{- \frac32}
 K_h (\tau) K_v (\tau')  U_{n}^{j,\ell} (t,x)$ is uniformly bounded in~$L^1$ in all variables.
To prove the result, by Lebesgue's dominated convergence theorem we     shall therefore  prove the pointwise convergence of~$ \tau^{ -2} \tau'^{- \frac32}
K_h (\tau) K_v (\tau')  U_{n}^{j,\ell}(t,x)$ to zero for almost every~$(\tau,\tau',t,x)$,  as~$n$ goes to infinity.

\medskip
\noindent  We shall use the well-known bounds
\begin{equation}\label{boundKhKv}
\begin{aligned}
\| K_h (\tau) K_v (\tau')  f(t,x)\|_{L^\infty_{t,x}} \leq \tau ^{-  1}  {\tau' } ^{-\frac12 } \|f(t,x)\|_{L^\infty_t L^1_x} \quad \mbox{and} \\
 \| K_h (\tau) K_v (\tau')  f(t,x)\|_{L^\infty_{t,x}} \leq  \|f(t,x)\|_{L^\infty_{t,x}} \, ,
 \end{aligned}
\end{equation}
as well as their interpolates, in the horizontal and vertical space variables: for instance denoting~$L^p_h L^r_v:=L^p(\R^2;L^r(\R))$ we have also
$$
\| K_h (\tau) K_v (\tau')  f(t,x)\|_{L^\infty_{t,x}} \leq \tau ^{-  1}   \|f(t,x)\|_{L^\infty_t L^1_h L^\infty_v} \, .
$$

 \smallskip
\noindent We first notice that
$$
\begin{aligned}
\| U_{n}^{j,\ell}\|_{L^\infty_t L^1_x} & \leq \big\| \widetilde \Lambda^n_{\boldsymbol{\eta^{\ell}} ,\boldsymbol{\delta^{\ell}},\boldsymbol{\tilde x^{\ell}} }
  \tilde  U^{h,\ell} \big\|_{L^\infty_t L^2_h L^1_v }\big \| \widetilde \Lambda^n_{\boldsymbol{\eta^{j}} ,\boldsymbol{\delta^{j}},\boldsymbol{\tilde x^{j}} }
  \tilde  U^{h,j}\big \|_{L^\infty_t L^2_h L^\infty_v } \\
   & \leq C \delta_n^\ell \,
\end{aligned}
$$
so the a.e. pointwise convergence of~$ \tau^{ -2} \tau'^{- \frac32}
K_h (\tau) K_v (\tau')  U_{n}^{j,\ell}(t,x)$  to zero follows, using~(\ref{boundKhKv}), if (by symmetry in~$\ell$ and~$j$) either~$ \delta_n^\ell $ or~$ \delta_n^j $ go to zero. So from now on we assume that~$ \delta_n^\ell $ and~$ \delta_n^j $ go to infinity or 1.
Next we write
$$
\begin{aligned}
\| U_{n}^{j,\ell}\|_{L^\infty_t L^1_h L^\infty_v} & \leq \| \widetilde \Lambda^n_{\boldsymbol{\eta^{\ell}} ,\boldsymbol{\delta^{\ell}},\boldsymbol{\tilde x^{\ell}} }
  \tilde  U^{h,\ell} \|_{L^\infty_t L^1_h L^\infty_v } \| \widetilde \Lambda^n_{\boldsymbol{\eta^{j}} ,\boldsymbol{\delta^{j}},\boldsymbol{\tilde x^{j}} }
  \tilde  U^{h,j} \|_{L^\infty_t L^\infty_h L^\infty_v } \\
   & \leq C\frac{ \eta_n^\ell }{\eta_n^j }\,
\end{aligned}
$$
so again from now on we may assume that~$ \eta_n^\ell =  \eta_n^j$, if not the result is proved (if one or the other ratio goes to zero).
But in that case
$$
\| U_{n}^{j,\ell}\|_{L^\infty_t L^\infty_x}   \leq C\frac{1 }{(\eta_n^\ell)^2 }\,
$$
hence from now on we restrict our attention to the case when~$ \eta_n^\ell =  \eta_n^j \to 0$ or~1.
We notice that by the change of variables
$$
y_h: = \frac {x_h - \tilde x_{n,h}^\ell}{\eta_n^\ell}  \, , \quad y_3: = \frac {x_3 - \tilde x_{n,3}^\ell}{\delta_n^\ell}  \, ,\quad \sigma := (\eta_n^\ell)^{-2}\tau \, , \quad \sigma' := (\delta_n^\ell)^{-2}\tau' \, , \quad s:= (\eta_n^\ell)^{-2}t \, ,
$$
we have after an easy computation
$$\int \! \! \tau^{   -2} \tau'^{- \frac 32} \Big | K_h (\tau) K_v (\tau')  U_{n}^{j,\ell} (t,x)\Big |   d\tau d\tau' dxdt
\!  = \! \!
\int \! \! \sigma^{   -2} \sigma'^{- \frac 32} \Big | K_h (   \sigma) K_v (   \sigma') \tilde U_{n}^{j,\ell} (s,y)\Big |  d\sigma d\sigma' dsdy
$$
where
$$
\tilde U_{n}^{j,\ell} (s,y) :=    \tilde  U^{h,\ell} (s,y)  \otimes    \tilde  U^{h,j}  \Big(
 s ,  y_h + \frac {\tilde x_{n,h}^\ell - \tilde x_{n,h}^j}{\eta_n^j}
 , \frac{\delta_n^\ell}{\delta_n^j} y_3 + \frac {\tilde x_{n,3}^\ell - \tilde x_{n,3}^j}{\delta_n^j}
 \Big)
 \, ,
 $$
 so if~$ \delta_n^\ell = \delta_n^j$  then the orthogonality assumption on the cores of concentration implies the result, so we may assume for instance that~$ \delta_n^\ell /  {\delta_n^j}$ goes to infinity, and since neither goes to zero, that in particular~$ \delta_n^\ell $ goes to infinity. The same argument   lets us  assume that~${(\tilde x_{n,h}^\ell - \tilde x_{n,h}^j)} / {\eta_n^j}$ is bounded.

\smallskip
\noindent Next we notice that the change of variables
$$
y_h: = \frac {x_h - \tilde x_{n,h}^\ell}{\eta_n^\ell} , \quad y_3: = \frac {x_3 - \tilde x_{n,3}^\ell}{\delta_n^j} ,\quad \sigma := (\eta_n^\ell)^{-2}\tau, \quad \sigma' := (\delta_n^j)^{-2}\tau', \quad s:= (\eta_n^\ell)^{-2}t \, ,
$$
gives
$$\int \! \! \tau^{   -2} \tau'^{- \frac 32} \Big | K_h (\tau) K_v (\tau')  U_{n}^{j,\ell} (t,x)\Big |   d\tau d\tau' dxdt
\!  = \! \!
\int \! \! \sigma^{   -2} \sigma'^{- \frac 32} \Big | K_h (   \sigma) K_v (   \sigma') \widetilde V_{n}^{j,\ell} (s,y)\Big |  d\sigma d\sigma' dsdy
$$
where
$$
\widetilde V_{n}^{j,\ell} (s,y) :=    \tilde  U^{h,\ell} (s,y_h, \frac{\delta_n^j}{\delta_n^\ell} y_3 )  \otimes    \tilde  U^{h,j}  \Big(
 s ,  y_h + \frac {\tilde x_{n,h}^\ell - \tilde x_{n,h}^j}{\eta_n^j}
 ,  y_3 + \frac {\tilde x_{n,3}^\ell - \tilde x_{n,3}^j}{\delta_n^j}
 \Big)
 \, .
 $$
So
  if~${(\tilde x_{n,3}^\ell  - \tilde x_{n,3}^j)} /  {\delta_n^j}$ is not bounded,
 then for each fixed~$y_3$ the limit of~$\widetilde V_{n}^{j,\ell} (s,y)$ is zero hence we may from now on assume that~${(\tilde x_{n,3}^\ell  - \tilde x_{n,3}^j)} /  {\delta_n^j}$ is   bounded, and similarly for~$\tilde x_{n,3}^j /  \delta_n^j$ and~$\tilde x_{n,3}^\ell /  \delta_n^j$
by translation invariance. Notice that repeating the argument~(\ref{limxelldeltaell}) we get that~$\tilde x_{n,3}^\ell /  \delta_n^\ell$ must go  to zero.    According to Assumption~\ref{sousdecompo}, we may therefore now assume that
  $$
    \tilde \varphi^{h,\ell} (\cdot,0) \equiv 0\, ,
  $$
  which implies by Lemma~{\rm \ref{globalfirstprofile}}, (\ref{smallerthanalphaatz3}), that
  \begin{equation}\label{smallatzeroUhl}
 \Big  |\tilde  U^{h,\ell} \big(t,y_h, \frac{ \delta_n^j} {\delta_n^\ell} y_3 \big)\Big  |\leq  \Big(   \frac{ \delta_n^j} {\delta_n^\ell}  |y_3  | + \alpha\Big) f(t,y_h)
  \end{equation}
  where~$f (t,y_h)$ is a smooth function in~$L^\infty(\R^+;L^2 \cap L^\infty(\R^2))$.
We obtain finally that
$$
 \|\widetilde V_n^{j, \ell} (\cdot,\cdot,y_3)\|_{L^\infty_t L^1_h }   \lesssim\Big| \alpha +  \frac{\delta_n^j} {\delta_n^\ell} \Big | \, .
$$
   The result in~$L^1(\R^+; \dot B^{2, 1 }_{1,1})$ follows.

     \smallskip
 \noindent  The same argument gives actually also the result in~$\widetilde{L^\infty}(\R^+; \dot B^{0, 1 }_{1,1})$  since all convergences to zero above are uniform in~$t$.

  \smallskip
 \noindent All other terms of~$  H_n^{L,1,2} $   are dealt with in a similar fashion hence~$  H^{L,1}_n$ satisfies the bounds~(\ref{calhn2}) and~(\ref{calhn3}).

    \medskip
\noindent
Recalling that~$ H_n^{L,2} $ was already dealt with, let us finally consider~$ H_n^{L,3} $ with
 $$
 H_n^{L,3} :=  u \cdot \nabla   F_n^{L}+ F_n^{L}
\cdot \nabla u  \, .
 $$
 Using the decomposition~{\rm (\ref{decomposeF})} of~$F_n^{L}$ and the same arguments as above give
   $$
\forall L, \quad \limsup_{n \to \infty} \big(u \cdot \nabla   \widetilde F_n^{L,1} +   \widetilde F_n^{L,1}\cdot \nabla u\big)
= 0 \, \mbox{ in } \,
 {L^1(\R^+; \dot B^{1, 1 }_{1,1} + \dot B^{2, 0 }_{1,1}  )\cap \widetilde{L^\infty}(\R^+; \dot B^{-1,1}_{1,1} + \dot B^{0, 0 }_{1,1}  )
}
 $$
 where
 $$
  \widetilde F_n^{L,1}  := \sum_{1 \leq \ell \leq L\atop \eta_n^\ell/ \delta_n^\ell \to 0} \tilde u_n^\ell+ \sum_{\ell = 1 }^{    L }  u_n^\ell \quad
   \mbox{and} \quad  F_n^{L,2}  :=\sum_{\ell = 1 \atop \eta_n^\ell / \delta_n ^\ell \to \infty}^{    L } \tilde u_n^\ell       \, .
$$
 while the terms~$   F_n^{L,1} - \widetilde F_n^{L,1}$ and~$ F_n^{L,2} $
are dealt with using the product laws~(\ref{productanisoh})-(\ref{productaniso3}). We leave the details to the reader.
\noindent Lemma~{\rm \ref{lemforce}} is proved. \end{proof}

\subsection{Proof of Corollaries~{\rm \ref{corthm1}} and~{\rm \ref{corthm}}}
\subsubsection{Proof of Corollary~{\rm \ref{corthm1}}}
If the solution~$u$ associated with~$u_0$ only has a finite life span~$T^*$, then we can retrace the following steps, replacing everywhere~$\R^+$ by~$[0,T]$ for~$T < T^*$ and it is obvious that the result of Corollary~{\rm \ref{corthm1}} holds as soon as~$n$ is large enough (depending on~$T$).

\subsubsection{Proof of Corollary~{\rm \ref{corthm}}}
The proof of that corollary is very close to the proof of a similar result in the isotropic context (see~\cite{gprofils}, Theorem 2(ii)).
Under the assumptions of Corollary~{\rm \ref{corthm}}, we can apply the previous results (in particular Corollary~{\rm \ref{corthm1}}) to write that as long as the solution~$u$ associated with~$u_0$ exists, it may be decomposed into
$$
  u = u_n -  {\mathcal U}_n^L   - {\mathcal V}_n^{L} - w_n^{L} \, ,
 $$
 and we know that for all~$T < T^*$, denoting by~$\displaystyle {\mathcal L}_{2}(T) :=
  \widetilde   {L^2}([0,T];\dot B^{\frac23 ,\frac13 }_{3,1}) $,
 \begin{equation}\label{smallwnL}
 \lim_{\alpha \to 0} \big( \limsup_{n \to \infty} \|w_n^{L}\|_{  {\mathcal L}_{2}(T)} \big)\to 0 \, , \quad L \to \infty \, .
 \end{equation}
 Moreover we also have, for~$n$ large enough, $\alpha$ small enough and all~$L$ (due to the assumption on~$u_n$ and to Lemma~{\rm \ref{lemdrift}}),
  $$
 \|u +  w_n^{L}\|_{ {\mathcal L}_{2}(T)} \leq C \, ,
 $$
 uniformly in~$L$, $\alpha$ and~$n$. Next
    recalling that if a solution blows up at time~$T^*$, then its norm in~$  {\mathcal L}_{2}(T)$ blows up when~$T$ goes to~$T^*$ (see Appendix~{\rm \ref{globalsmallaniso}}), we can therefore choose~$T < T^*$ such that
 $$
 \|u\|_{  {\mathcal L}_{2}(T)} \geq 2C \, .
 $$
We conclude by noticing that
$$
\|u\|_{ {\mathcal L}_{2}(T)} \leq C + \|w_n^{L}\|_{ {\mathcal L}_{2}(T)}
$$
so choosing~$n$ and~$L$ large enough and~$\alpha$ small enough gives a contradiction due to~{\rm (\ref{smallwnL})}, whence the result.

    %%%%%%%%%%%%%%%%%%%%%%%%%%%%%%%%%%%%%%%%%%%%%

 \section{Profile decompositions in~${\mathcal B}^1_q$}\label{profile}

\subsection{Introduction and statement of the theorem}   After the pioneering works of P. -L.  Lions \cite{Lions1} and
\cite{Lions2},  the lack of compactness in critical Sobolev embeddings was
investigated for different types of examples through several angles.   For instance, in \cite{Ge1} the lack of  compactness  in the critical  Sobolev
embedding~$\displaystyle
\dot H^s(\R^d)\hookrightarrow L^{p}(\R^d)
$
in the case
where $d \geq 3$ with~$0\leq s<d/2$ and $p=2d/(d-2s)$ is described in terms of microlocal defect measures and in \cite{G}, it is characterized  by means of profiles.  More generally for Sobolev spaces in the~$ L^q$ framework, this question  is treated in \cite{jaffard}  (see also the more recent work~\cite{koch}) by  the use of
nonlinear wavelet approximation theory. In \cite{BMM}, the authors look into the lack of compactness of the critical
embedding~$\displaystyle
 H_{rad}^1(\R^2)\hookrightarrow {\mathcal L} \, , $
where~${\mathcal L}$ denotes the Orlicz space associated to the function $\phi(s)={\rm e}^{ s^2}-1$.
 Other studies were   conducted in various works (see among others~\cite{Benameur,BC,tintarev,ST,So,St}) supplying us with a large amount of information on
solutions of nonlinear  partial differential equations, both in the
elliptic  or the evolution framework;  among other applications, one can mention~\cite{BG,gprofils,GG,gkp,km,ker,Tao}.
 Recently  in \cite{BCG}, the
wavelet-based profile decomposition introduced by S. Jaffard in \cite{jaffard} was revisited in order to treat a larger range of examples of critical embedding of function spaces~$\displaystyle
 X \hookrightarrow Y
 $
including Sobolev, Besov, Triebel-Lizorkin, Lorentz, H\"older and BMO spaces.   For that purpose, two generic properties on the spaces $X$ and $Y$ were  identified
to  build the profile decomposition in a unified way. These properties concern wavelet decompositions in the spaces~$X$ and $Y$ supposed to have the same  scaling, and endowed  with an unconditional wavelet basis $(\psi_\lambda)_{\lambda\in\Lambda}$.

\medskip
\noindent
The first  property   is related to the existence of a {\it nonlinear projector}  $Q_M$ satisfying
$$
\lim_{M\to +\infty}\max_{\|f\|_X \leq 1}\|f-Q_M f\|_Y=0\, .$$
More precisely, if~$f$ may be decomposed in the following way (the notation will be made precise below):
$\displaystyle f=\sum_{\lambda\in \nabla} d_\lambda\psi_\lambda  ,$
then  $Q_M f$, sometimes called the {\it best $M$-term approximation},    takes the general form
\begin{equation}
\label{nonlinproj}
Q_M f:=\sum_{\lambda \in E_M} d_\lambda\psi_\lambda \, ,
\end{equation}
where the sets $E_M=E_M(f)$ of cardinality $M$ depend on $f$ and
satisfy~$\displaystyle E_M(f)\subset E_{M+1}(f) .$
The existence of such a nonlinear projector was extensively studied in nonlinear approximation theory and for many cases, like Sobolev   spaces, it turns out that the set  $E_M=E_M(f)$ can be chosen as
 the subset of $\nabla$  that corresponds to the $M$ largest
values of $|d_\lambda|$. It is in fact known (see \cite{Me} for instance) that in homogeneous Besov spaces $\dot B^{\sigma}_{r,r}$,
we have the following norm equivalence :
\begin{equation}
\|f\|_{\dot B^{\sigma}_{r,r}}\sim \|(d_\lambda)_{\lambda\in\nabla}\|_{\ell^r}  \, ,
\label{besov}
\end{equation}
for $\displaystyle f=\sum_{\lambda\in\nabla} d_\lambda\psi_\lambda$ with wavelets normalized in $\dot B^{\sigma}_{r,r}$. Therefore, in the
 particular case where~$X=\dot B^s_{p,p}$ and~$Y=\dot B^t_{q,q}$,
with~$\frac 1 p-\frac 1 q=\frac {s-t} d$, the  nonlinear projector  $Q_M$ defined by {\rm (\ref{nonlinproj})},
where~$E_M=E_M(f)$
is the subset of $\nabla$ of cardinality $M$ that corresponds to the $M$ largest
values of $|d_\lambda|$, is appropriate and satisfies (see~\cite{BCG} for instance):
\begin{equation}
\sup_{\|f\|_{\dot B^s_{p,p}} \leq 1}\|f-Q_Mf\|_{\dot B^t_{q,q}}\leq CM^{-\frac {s-t} d} \, .
\label{particularbesov}
\end{equation}
The second property  concerns the {\it stability}  of
wavelet expansions in the  function space~$X$  with respect to certain operations
such as ``shifting'' the indices of wavelet coefficients, as well
as disturbing the value of these coefficients.  In practice and for most cases of interest,  this property  derives from the fact that the $X$ norm
of a function is equivalent to the norm
of its wavelet coefficients in a certain sequence space, by invoking Fatou's lemma.

\medskip
\noindent Under these assumptions, it is
proved in~\cite{BCG}  that, as in the previous works \cite{Ge1} and \cite{jaffard}, translation and scaling invariance
are the sole responsible for the defect of compactness of the
embedding of~$ X \hookrightarrow Y $.

\medskip
\noindent In what follows, we shall apply the same lines of reasoning,  taking advantage of an anisotropic  wavelet setting  to describe  the lack of compactness of the Sobolev embedding~${\mathcal B}^1_q \hookrightarrow
 \dot B_{p,p}^{- 1 + \frac 2 p,\frac1p}$ with  $p > \max(1,q) $
in terms of an asymptotic anisotropic profile decomposition.  We recall that as defined in the introduction of this paper,~$ {\mathcal B}^1_q:= \dot B_{1,q}^{1,1}$. Our presentation is essentially based on ideas and methods developed for the isotropic setting   in \cite{BCG}. Because of the anisotropy, we  use a two-parameter  wavelet basis. More precisely,  wavelet decompositions of a function have the form
\begin{equation}
\label{anisotropicdec}
f=\sum_{\lambda =(\lambda_1,\lambda_2)\in \nabla} d_\lambda\psi_\lambda \, ,
\end{equation}
where  the wavelets $\psi_\lambda$    are assumed to be  normalized in the space $X=  {\mathcal B}^1_q$, and where the notation~$\lambda_1 = (j_1,k_1) \in \ZZ\times \ZZ^2$ (resp. $\lambda_2 = (j_2,k_2) \in \ZZ\times \ZZ$)  concatenates  the scale index~$j_1= j_1(\lambda_1)$ (resp.~$j_2= j_2(\lambda_2)$) and the space index~$k_1=k_1(\lambda_1)$ (resp.~$k_2=k_2(\lambda_2)$)  for the horizontal variable (resp. the vertical variable).  Thus the index set $\nabla$ in {\rm (\ref{anisotropicdec})} is  defined as~$\displaystyle\nabla:=(\ZZ\times \ZZ^2)\times (\ZZ\times \ZZ)$
and  the wavelets~$\psi_\lambda$  write under the form
$$ \psi_\lambda= \psi_{(\lambda_1,\lambda_2)}= 2^{j_1}\psi (2^{j_1}\cdot-k_1, 2^{j_2}\cdot-k_2)$$
where $\psi$  the so-called ``mother wavelet''  is generated by a finite
dimensional inner product of one  variable  functions $\psi^e$,  for $e\in E$ a finite set. It is known (see for instance \cite{BN})  that  wavelet bases are unconditional bases, i.e.   there exists a constant~$D$ such that
for any finite subset $E\subset \nabla$ and coefficients
vectors $(c_\lambda)_{\lambda\in E}$ and $(d_\lambda)_{\lambda\in E}$
such that~$|c_\lambda| \leq |d_\lambda|$ for all $\lambda$, one has
\begin{equation}
\label{uncond}
\big\|\sum_{\lambda\in E} c_\lambda\psi_\lambda\big\|_{{\mathcal B}^1_q} \leq D\,\big\|\sum_{\lambda\in E} d_\lambda\psi_\lambda\big\|_{{\mathcal B}^1_q}
\end{equation}
 and similarly
for $ \dot B_{p,p}^{- 1 + \frac 2 p,\frac1p}$.
In addition ${\mathcal B}^1_q$ and $ \dot B_{p,p}^{- 1 + \frac 2 p,\frac1p}$ may be characterized by simple properties on wavelet coefficients: for~$
\displaystyle f=\sum_{\lambda\in\nabla} d_\lambda\psi_\lambda = \sum_{(\lambda_1,\lambda_2)\in\nabla} d_{(\lambda_1,\lambda_2)} \psi_{(\lambda_1,\lambda_2)}
$
with normalized wavelets, we have the following norm equivalences:
\begin{equation}
\label{norme1}
\|f\|^q_{{\mathcal B}^1_q}\sim \sum_{j_1\in\ZZ} 
\Big(\sum_{|\lambda_1|=j_1}\Big(\sum_{j_2\in\ZZ} \big(\sum_{|\lambda_2|=j_2}  |d_{(\lambda_1,\lambda_2)}|\big)^q
\Big)^{1/ q}
\Big)^q
\end{equation}
and
\begin{equation}
\label{norme2}
\|f\|_{\dot B_{p,p}^{ - 1 + \frac 2 p}(\R^2 ; \dot B_{p, p}^\frac1p(\R))}\sim \|(d_\lambda)_{\lambda\in\nabla}\|_{\ell^p} \, .
\end{equation}
Moreover as  proved in \cite{BCG, K}, there exists a nonlinear projector $Q_M$ of the form {\rm (\ref{nonlinproj})} such that
\begin{equation}
\label{nonlincondaniso}
\lim_{M\to +\infty}\max_{\|f\|_{ {\mathcal B}^1_q} \leq 1}\|f-Q_M f\|_{ \dot B_{p,p}^{- 1 + \frac 2 p,\frac1p}}=0 \, .
\end{equation}
We refer to \cite{ah,BN,bowho,Co,Dau,De,DJP,ght, Hr,christopher,Tem} and the references therein for more details on the construction
of wavelet bases and on the characterization of function spaces by expansions in such bases.

\noindent In the sequel,  for any function $\phi$, not necessarily a wavelet, and any scale-space index~$\lambda$ defined by~$
\displaystyle \lambda=(\lambda_1,\lambda_2) = ((j_1,k_1),(j_2,k_2)) \in \nabla , $ we shall use the notation
$$
\phi_\lambda (x):=2^{j_1}\phi(2^{j_1}x_h-k_1, 2^{j_2}x_3-k_2),
\label{rescale}
$$
 and to avoid heaviness, we shall define for ~$i \in \{1,2\}$ and $\lambda=(\lambda_1,\lambda_2) = ((j_1,k_1),(j_2,k_2))$, by~$j_i = j_i (\lambda)$ and $k_i = k_i (\lambda)$.

\smallskip
\noindent We shall prove the following theorem, characterizing    the lack of compactness in the critical
embedding~$
 {\mathcal B}^1_q\hookrightarrow \dot B_{p,p}^{- 1 + \frac 2 p,\frac1p}$,~$p > \max(q,1)$. The   result actually holds for many   such embeddings, but for the sake of readability we choose to only state and prove it in this particular case.
 \begin{thm}
\label{anisocompthm}
Let $(u_n)_{n\geq 0}$ be a bounded sequence in ${\mathcal B}^1_q$. Then, up to a
subsequence extraction, there exists a family of functions $(\phi^\ell)_{\ell\geq0}$ in ${\mathcal B}^1_q$
and sequences of scale-space indices~$(\lambda_\ell(n))_{n\geq0}$ for each $\ell>0$
such that for all~$p > \max(q,1)$,
$$
u_n=\sum_{\ell=1}^L \phi_{\lambda_\ell(n)}^\ell+ \psi_{n}^L \, ,\quad
\mbox{where} \quad
\limsup_{n\to\infty}\;\|\psi_{n}^L\|_{\dot B_{p,p}^{- 1 + \frac 2 p,\frac1p}}\to 0\quad
\mbox{as} \quad
  {L\to\infty} \, .$$
The decomposition
is asymptotically orthogonal in the sense that  for any $k\neq \ell$, as $n\to +\infty$, either
\begin{equation}\label{orthogonalscales}
|j_1(\lambda_k(n))-j_1(\lambda_\ell(n))| + |j_2(\lambda_k(n))-j_2(\lambda_\ell(n))|\to +\infty\;\;\;\;
\end{equation}
or $$ |k_1(\lambda_k(n))-2^{j_1(\lambda_k(n))-j_1(\lambda_\ell(n))}k_1(\lambda_\ell(n))| + |k_2(\lambda_k(n))-2^{j_2(\lambda_k(n))-j_2(\lambda_\ell(n))}k_2(\lambda_\ell(n))| \to +\infty \,  .$$

\noindent
Moreover, we have
the following stability estimates
\begin{equation}
\label{ortogonal}
 \sum_{\ell=1}^{\infty}\,\|\phi^\ell\|_{{\mathcal B}^1_q}  \leq  C \sup_{n\geq 0} \|
u_n\|_{ {\mathcal B}^1_q} \, ,
 \end{equation}
where $C$ is a constant which only depends on  the choice of the wavelet basis.
\end{thm}
\begin{rmk}\label{equalscales}
{\rm
Up to rescaling the profiles, if~(\ref{orthogonalscales}) does not hold then one may assume that~$j_i (\lambda_\ell(n)) = j_i (\lambda_k(n))$ for~$i \in \{1,2\}$.
}
\end{rmk}

 \subsection{Proof of Theorem~\ref{anisocompthm}}
Along the same lines as in \cite{BCG},  the anisotropic profile  decomposition construction   proceeds
in several steps.

\subsubsection{Step 1: rearrangements} According to the notation {\rm (\ref{anisotropicdec})},  we first introduce the
wavelet decompositions of the sequence $u_n$, namely~$\displaystyle u_n=\sum_{\lambda\in\nabla} d_{\lambda,n} \psi_\lambda.$
Then we use the  nonlinear projector $Q_M$ to write for each $M>0$
$$
u_n = Q_M u_n + R_M u_n\, ,
\quad \mbox{with} \quad
\lim_{M\to +\infty} \sup_{n>0}\|R_M u_n\|_{\dot B_{p,p}^{- 1 + \frac 2 p,\frac1p}}=0 \, ,
$$
in view of {\rm (\ref{nonlincondaniso})} and the boundedness
of the sequence $u_n$ in $ {\mathcal B}^1_q$.
Noting
$$ Q_M u_n = \sum_{m=1}^M d_{m,n} \psi_{\lambda(m,n)} \, ,$$
it is obvious that the coefficients $d_{m,n}$ are uniformly bounded in $n$ and $m$, so up to a diagonal subsequence extraction procedure in $n$,   we can  reduce to the case where for all $m$,
 the sequence~$(d_{m,n})_{n>0}$ converges towards a finite limit that depends on $m$,
$$
d_m:=\lim_{n\to +\infty} d_{m,n} \, .
$$
We may thus write
$$u_n=\sum_{m=1}^M d_{m} \psi_{\lambda(m,n)}+t_{n,M} \, ,\quad \mbox{where}
\quad t_{n,M}:=\sum_{m=1}^M(d_{m,n}-
 d_m)\psi_{\lambda(m,n)}+ R_M u_n \, .
$$

\subsubsection{  Step 2: construction of approximate profiles}  The profiles $\phi^\ell$ will be built  as limits of
sequences $\phi^{\ell,i}$ resulting
 by the following algorithm. At the first iteration $i=1$, we define
$$
\phi^{1,1}=d_1\psi \, ,\;\; \lambda_1(n):=\lambda(1,n)\, ,\;\; \varphi_1(n):=n \, .
$$
Now, supposing
 that after iteration step~$i-1$, we have constructed $L-1$ functions denoted by~$(\phi^{1,i-1},\dots,\phi^{L-1,i-1})$ and scale-space index sequences $(\lambda_1(n),\dots,\lambda_{L-1}(n))$
with $L\leq i$,
as well as an increasing sequence of positive integers $\varphi_{i-1}(n)$ such that
$$
\sum_{m=1}^{i-1}d_m\psi_{\lambda(m,\varphi_{i-1}(n))}
= \sum_{\ell=1}^{L-1}\phi^{\ell,i-1}_{\lambda_\ell(\varphi_{i-1}(n))} \, ,
$$
we shall use the $i$-th
component $d_i\psi_{\lambda(i,\varphi_{i-1}(n))}$ to
either modify one of these functions or construct a new one at iteration $i$  according to the
following dichotomy. \\

\noindent (i) First case: assume that we can extract $\varphi_i(n)$ from $\varphi_{i-1}(n)$
such that for $\ell =1,\dots,L-1$ at least one of the following holds:
\begin{equation}
\label{firstas}
\lim_{n\to +\infty} |j_1(\lambda(i,\varphi_{i}(n)))-j_1(\lambda_\ell(\varphi_{i}(n)))| + |j_2(\lambda(i,\varphi_{i}(n)))-j_2(\lambda_\ell(\varphi_{i}(n)))|=+\infty \, ,
\end{equation}
or
\begin{equation}\label{secas}
\begin{aligned}
\lim_{n\to +\infty}
\Big |k_1\big(\lambda(i,\varphi_{i}(n))\big)-2^{j_1 (\lambda(i,\varphi_{i}(n)) )-j_1 (\lambda_\ell(\varphi_{i}(n)) )}
k_1\big (\lambda_\ell(\varphi_{i}(n)) \big ) \Big|  \\
+\Big |k_1\big(\lambda(i,\varphi_{i}(n))\big)-2^{j_1(\lambda(i,\varphi_{i}(n)))-j_1(\lambda_\ell(\varphi_{i}(n)))}
k_1\big(\lambda_\ell(\varphi_{i}(n)) \big)\Big |=+\infty \, .
\end{aligned}
\end{equation}
In such a case, we create a new profile and scale-space index sequence by defining
$$
\phi^{L,i}:=d_i\psi \, , \quad\lambda_L(n):=\lambda(i,n) \,  , \quad \phi^{\ell ,i}:=\phi^{\ell,i-1} \, \forall \ell \in \{1,\dots,L-1\} \, .
$$

\noindent (ii) Second case: assume that for some subsequence $\varphi_{i}(n)$
of $\varphi_{i-1}(n)$ and for some~$\ell$ belonging to~$ \{1,\dots,L-1\}$ neither  {\rm (\ref{firstas})} nor {\rm (\ref{secas})}   holds.
 Then
it follows  that  for~$i $ in~$ \{1,2\}$, the quantities~$\displaystyle
j_i(\lambda_\ell(\varphi_{i}(n)))-j_i(\lambda(i,\varphi_{i}(n)))$   and~$\displaystyle k_i \big (\lambda(i,\varphi_{i}(n))\big)-2^{j_i(\lambda(i,\varphi_{i}(n)))-j_i(\lambda_\ell(\varphi_{i}(n)))}
k_i\big(\lambda_\ell(\varphi_{i}(n))\big)
$
 only take a finite number of values as $n$ varies.
Therefore, up to an additional subsequence extraction, we
may assume that there exists numbers $a_1$, $a_2$, $b_1$ and $b_2$ such that for all $n>0$ and for~$i \in \{1,2\}$,
$$
j_i(\lambda(i,\varphi_{i}(n)))- j_i(\lambda_\ell(\varphi_{i}(n)))=a_i \, ,
$$
and
$$
k_i(\lambda(i,\varphi_{i}(n)))-2^{j_i(\lambda(i,\varphi_{i}(n)))-j_i(\lambda_\ell(\varphi_{i}(n)))}
k_i(\lambda_\ell(\varphi_{i}(n)))=b_i \, .
$$
We then update the function $\phi^{\ell,i-1}$ according to
$$
\phi^{\ell,i}:=\phi^{\ell,i-1}+d_i 2^{a_1}\psi(2^{a_1}\cdot-b_1, 2^{a_2}\cdot-b_2) \, , \quad \phi^{\ell',i}:=\phi^{\ell',i-1} \, \forall\ell'\in \{1,\dots,L-1\} \, , \, \, \ell'\neq \ell.
$$
Up to a diagonal subsequence extraction procedure in $n$, it derives from this construction that for each value of  $M$ there exists $L=L(M)\leq M$ such that
$$
\sum_{m=1}^M d_{m} \psi_{\lambda(m,n)}=\sum_{\ell=1}^L \phi^{\ell,M}_{\lambda_\ell(n)}
$$
with for each $\ell=1,\dots,L$
$$
\phi^{\ell,M}_{\lambda_\ell(n)}=\sum_{m\in E(\ell,M)}d_{m} \psi_{\lambda(m,n)}\, ,
$$
and where the sets $E(\ell,M)$ for $\ell=1,\dots,L$ form
a  partition of $\{1,\dots,M\}$.
Moreover,  for ~$i \in \{1,2\}$ and for any $m,m'\in E(\ell,M)$ we have
\begin{equation}
\label{cond1}
j_i(\lambda(m,n))-j_i(\lambda(m',n))=a_i(m,m') \, ,
\end{equation}
and
\begin{equation}
\label{cond2}
k_i(\lambda(m,n))-2^{j_i(\lambda(m,n))-j_i(\lambda(m',n))}k_i(\lambda(m',n))=b_i(m,m') \, ,
\end{equation}
where $a_i(m,m')$ and $b_i(m,m')$ do not depend on $n$.

\subsubsection{Step 3: construction of the exact profiles}  The profiles $\phi^{\ell}$ will be obtained as the limits in $ {\mathcal B}^1_q$
of $\phi^{\ell,M}$ as $M\to +\infty$. To this end, we shall use {\rm (\ref{norme1})} and the fact that the wavelet basis~$(\psi_\lambda)_{\lambda\in\nabla}$ is an unconditional basis of ${\mathcal B}^1_q $. So let us define for fixed $  \ell$ and $M$ such that~$ \ell\leq L(M)$  the functions
$\displaystyle 
g^{\ell,M}:=\sum_{m\in E(\ell,M)}d_{m} \psi_{\lambda(m)} $ and~$\displaystyle f^{\ell,M,n}:=\sum_{m\in E(\ell,M)}d_{m,n} \psi_{\lambda(m)}  ,
$
with $\lambda(m):=\lambda(m,1)$. In view of   {\rm (\ref{cond1})}, {\rm (\ref{cond2})} and the scaling invariance of the space ${\mathcal B}^1_q $,  we have
$$
\|f^{\ell,M,n}\|_{{\mathcal B}^1_q}=\Big \|\sum_{m\in E(\ell,M)}d_{m,n} \psi_{\lambda(m,n)} \Big\|_{{\mathcal B}^1_q} \, .$$
Since $\displaystyle \sum_{m\in E(\ell,M)}d_{m,n} \psi_{\lambda(m,n)}$ is a part of the
expansion of $u_n$,   we deduce the existence of a constant~$C$  which depends neither on~$n$ nor on  $\ell$ and $M$ such that
$$
\|f^{\ell,M,n}\|_{{\mathcal B}^1_q} \leq C \, . $$
Now, according to the first step of the proof of the theorem, the coefficients $d_{m}$ are the limits of $d_{m,n}$ when~$ n$ tends to infinity. Therefore, {\rm (\ref{norme1})} and  Fatou's lemma  imply  that
$$
\|g^{\ell,M}\|_{{\mathcal B}^1_q} \leq \liminf_{n\to +\infty}\|f^{\ell,M,n}\|_{{\mathcal B}^1_q} \, ,
$$
which ensures the convergence in ${\mathcal B}^1_q$ of the sequence $g^{\ell,M}$   towards a limit $g^\ell$ as
$M\to +\infty$.

\noindent
Finally, since by construction  the $g^{\ell,M}$ are
rescaled versions of the $\phi^{\ell,M}$, there exists numbers~$A_1>0$,~$A_2>0$, $B_1\in\R^2$  and $B_2\in\R$ such that
$$
\phi^{\ell,M}= 2^{A_1}g^{\ell,M}(2^{A_1}\cdot  - B_1, 2^{A_2}\cdot  - B_2) \, .
$$
Therefore  $\phi^{\ell,M}$ converges in ${\mathcal B}^1_q$ towards  $\phi^\ell:=2^{A_1} g^{\ell}(2^{A_1}\cdot  - B_1, 2^{A_2}\cdot  - B_2)$ as
$M\to +\infty$.

\medskip
\noindent
To conclude the construction,  we argue exactly as in the proof of Theorem~1.1 in \cite{BCG}.

\medskip
\noindent Finally, let us prove that the decomposition derived in Theorem {\rm \ref{anisocompthm}} is stable.
The argument is again similar to the one followed in~\cite{BCG}, we reproduce it here for the convenience of the reader.
We shall use the  following property:
  if $E_1,\dots,E_L$ are disjoint finite sets
in $\nabla$, then for any coefficient sequence $(d_\lambda)$, one has
\begin{equation} \label{condpartition}
\sum_{\ell=1}^{L} \| \sum_{\lambda \in E_\ell} d_\lambda \psi_\lambda\|_{{\mathcal B}^1_q}  \leq C \| \sum_{\ell=1}^{L}  \sum_{\lambda \in E_\ell} d_\lambda \psi_\lambda  \|_{{\mathcal B}^1_q} \, .
 \end{equation}
Such an estimate was proved in \cite{BCG} for Besov spaces $\dot B_{p,a}^s (\R^d)$ and generalizes
 easily to our framework.
 Let us then consider for $\ell=1,\dots,L$ the functions
$$
\phi^{\ell,M,n}:=\sum_{m\in E(\ell,M)}d_{m,n}\psi_{\lambda(m,n)} \, ,
$$
where $E(\ell,M)$ are the sets introduced in the second step of the proof of the decomposition. These functions  are linear combinations of wavelets with indices in disjoint
finite sets~$E_1,\dots,E_L$ (that vary with $n$), which implies by~{\rm (\ref{condpartition})} that
$$
\sum_{\ell=1}^L  \|\phi^{\ell,M,n}\|_{{\mathcal B}^1_q} \leq C
\Big \| \sum_{\ell=1}^L \phi^{\ell,M,n}\Big \|_{{\mathcal B}^1_q} \, .
$$
Since the functions $\phi^{\ell,M,n}$ are part of the wavelet expansion of $u_n$,  we deduce that
$$
\sum_{\ell=1}^L  \|\phi^{\ell,M,n}\|_{{\mathcal B}^1_q}  \leq C\, \sup_{n\geq 0}\|
u_n\|_{{\mathcal B}^1_q} \, .
$$
Now, by construction the sequence $(\phi^{\ell,M,n})_{n>0}$ converges in ${\mathcal B}^1_q$ towards the approximate profiles~$ \displaystyle
\phi^{\ell,M}_{\lambda_\ell(n)}=\sum_{m\in E(\ell,M)}d_m\psi_{\lambda(m,n)}
$
as $n\to \infty$. It follows that for any
$\e>0$ we have
$$
\sum_{\ell=1}^L  \|\phi^{\ell,M}_{\lambda_\ell(n)}\|_{{\mathcal B}^1_q}\leq C \,\sup_{n\geq 0}\|
u_n\|_{{\mathcal B}^1_q}+\e \, ,
$$
for $n$ large enough.
Thanks to  the scaling invariance,  we thus find that
$$
\sum_{\ell=1}^L  \|\phi^{\ell,M}\|_{{\mathcal B}^1_q} \leq C\, \sup_{n\geq 0}\|
u_n\|_{{\mathcal B}^1_q}.
$$
Letting $M$ go to $+\infty$, we obtain the same inequality for
the exact profiles and we   conclude  by letting $L\to +\infty$. The theorem is proved. \hfill $\qed$

\subsection{Some additional properties}
The following result is very useful.
\begin{lem}\label{orthoaniso}
Let~$(u_n)_{n \in \N}$ be a  bounded sequence in~${\mathcal B}^1_q$, which does not converge strongly to zero in~${\mathcal B}^1_q$ and which may be decomposed with the notation of  Theorem~{\rm \ref{anisocompthm}} into
\begin{equation}\label{sum1}
u_n=\sum_{\ell=1}^L \phi_{\lambda_\ell(n)}^\ell + \psi_{n}^L \, .
\end{equation}
Let~$p \geq 2$ be given. For any~$\ell \in \{1,..., L\}$, there are three  constants~$C  > 0$ and~$(a_\ell^1,a_\ell^2) \in \ZZ^2$ such that
\begin{equation}\label{limsup>0}
 \limsup_{n \to \infty}  \, 2^{  j_1 (\lambda_\ell (n))   (-1+\frac2p)+ \frac{j_2 (\lambda_\ell (n)) }p } \Big\|\Delta_{j_1 (\lambda_\ell (n)) + a_\ell^1}^{h} \Delta_{j_2 (\lambda_\ell (n))+ a_\ell^2}^v \, u_n \Big \|_{L^p(\R^3)} = C \, .
\end{equation}
\end{lem}
\begin{proof} [Proof of Lemma~{\rm \ref{orthoaniso}}]
We start by noticing that the existence of~$C< \infty$ satisfying~(\ref{limsup>0}) is obvious, the only difficulty is to prove that~$C>0$.

\medskip

\noindent $\bullet \,  $ Let us first estimate one individual contribution, meaning let us show that there is~$ C^{\ell,p} > 0$ and~$(a_\ell^1,a_\ell^2) \in \ZZ^2$ such that
\begin{equation}
\label{oneind}
\limsup_{n \to \infty} \, 2^{  j_1 (\lambda_\ell (n))   (-1+\frac2p)+ \frac{j_2 (\lambda_\ell (n)) }p } \Big\|\Delta_{j_1 (\lambda_\ell (n)) + a_\ell^1}^{h} \Delta_{j_2 (\lambda_\ell (n))+ a_\ell^2}^v \,  \phi_{\lambda_\ell(n)}^\ell
\Big \|_{L^p(\R^3)}= C^{\ell,p} \, .
\end{equation}
 By definition $\Delta_{j_1+a_1}^{h} u = 2^{2 (j_1+a_1)} \Psi(2^{ j_1+a_1} \cdot)\ast_h u$ and $\Delta_{j_2+a_2}^{v} u = 2^{ j_2+a_2} \Psi(2^{ j_2+a_2} \cdot)\ast_v u$, where~$\Psi$ is the frequency localization function introduced in Appendix~{\rm \ref{appendixlp}}
and~$\ast_h$ (resp.~$\ast_v$) denotes the convolution operator in the horizontal (resp. vertical) variable. Writing
 $$ \phi_{\lambda_\ell(n)}^\ell = 2^{j_1 (\lambda_\ell (n))}\phi^\ell\Big(2^{j_1 (\lambda_\ell (n))} (\cdot \, _h- x_{n,h}^\ell), 2^{j_2(\lambda_\ell (n))}(\cdot \, _3-x_{n,3}^\ell)\Big),$$
 we  easily prove
 that
 $$
 \Delta_{j_1 (\lambda_\ell (n))+ a_\ell^1}^{h} \Delta_{j_2 (\lambda_\ell (n))+ a_\ell^2}^v  \phi_{\lambda_\ell(n)}^\ell
  = 2^{j_1 (\lambda_\ell (n))}(\widetilde\Psi^\ell  \ast \phi^\ell) \big( 2^{j_1 (\lambda_\ell (n))} (\cdot \, _h- x_{n,h}^\ell), 2^{j_2(\lambda_\ell (n))}(\cdot \, _3-x_{n,3}^\ell )\big)
 $$
where~$\displaystyle
\widetilde \Psi^\ell (x) := 2^{2 a_\ell^1+ a_\ell^2} \Psi (2^{a_\ell^1}x_h) \Psi (2^{a_\ell^2}x_3),
$
which ensures  that
\begin{equation}\label{psiellphiell}
 \limsup_{n \to \infty}  \, 2^{j_1 (\lambda_\ell (n)) (-1+\frac2p)+ \frac{j_2 (\lambda_\ell (n))}p} \Big\|\Delta_{j_1 (\lambda_\ell (n))}^{h} \Delta_{j_2 (\lambda_\ell (n))}^v u_n \Big \|_{L^p(\R^3)}  = \|\widetilde \Psi^\ell \ast \phi^\ell\|_{L^p(\R^3)} \neq 0 \, ,
 \end{equation}
 as soon as~$(a_\ell^1,a_\ell^2)$ are conveniently chosen so that the supports of~$\widehat {\widetilde  \Psi^\ell} $ and~$\widehat \phi^\ell$ are not disjoint.

\medskip

\noindent $\bullet \,  $ Next  let us prove that  for $\ell' \neq \ell$
$$
2^{j_1 (\lambda_\ell (n))(-1+\frac2p)+ \frac{j_2 (\lambda_\ell (n))}p} \Big\|\Delta_{j_1 (\lambda_\ell (n)) +a_\ell^1}^{h} \Delta_{j_2 (\lambda_\ell (n)) +a_\ell^2}^v \phi_{\lambda_{\ell'}(n)}^{\ell'} \Big \|_{L^p(\R^3)} \to 0 \quad \mbox{as} \quad n \to \infty \, ,
$$
when the scales $j(\lambda_\ell(n))$ and $j(\lambda_{\ell'}(n))$ are orthogonal, meaning~$2^{j_i(\lambda_\ell(n)) - j_i(\lambda_{\ell'}(n))} \to 0$ or~$\infty$ as~$n \to \infty$, for~$i$ equal either to~$1$ or~2.
Noticing that
$$
\Delta_{k}^{h}   \Delta_{j}^v \big(\phi (2^{k'} x_h, 2^{j'} x_3) \big)=
  (\Delta_{k-k'}^{h}   \Delta_{j-j'}^v \phi  )(2^{k'} x_h, 2^{j'} x_3)
$$
we deduce that
$$\begin{aligned}
2^{j_1 (\lambda_\ell (n))(-1+\frac2p)+ \frac{j_2 (\lambda_\ell (n))}p} \Big\|\Delta_{j_1 (\lambda_\ell (n)) +a_\ell^1}^{h} \Delta_{j_2 (\lambda_\ell (n)) +a_\ell^2}^v \,  \phi_{\lambda_{\ell'}(n)}^{\ell'} \Big \|_{L^p(\R^3)}
\\
 = 2^{j^{\ell,\ell'}_1 (n)(-1+\frac2p)+ \frac{j^{\ell,\ell'}_2 (n) }p} \Big\|\Delta_{j^{\ell,\ell'}_1 (n)+a_\ell^1}^{h}   \Delta_{j^{\ell,\ell'}_2 (n) +a_\ell^2 }^v\,  \phi^{\ell'} \Big \|_{L^p(\R^3)}
\end{aligned}
$$
where
$$
j^{\ell,\ell'}_1 (n) := j_1 (\lambda_\ell (n))-j_1 (\lambda_{\ell'} (n))\quad \mbox{and}¬¨‚Ä†\quad
j^{\ell,\ell'}_2 (n) := j_2 (\lambda_\ell (n))-j_2 (\lambda_{\ell'} (n)) \, .
$$
Since $\phi^{\ell'} \in \dot B^{-1+\frac2p,\frac1p}_{p,q}$, we deduce that
\begin{equation}\label{differentscalesOK}
 2^{j_1 (\lambda_\ell (n))(-1+\frac2p)+ \frac{j_2 (\lambda_\ell (n))}p}\Big\|\Delta_{j_1 (\lambda_\ell (n)) + a_\ell^1}^{h} \Delta_{j_2 (\lambda_\ell (n))+ a_\ell^2}^v \phi_{\lambda_{\ell'}(n)}^{\ell'} \Big \|_{L^p(\R^3)} \to 0,\quad \mbox{as} \quad n \to \infty \, .
 \end{equation}

\medskip

\noindent
$\bullet \,  $ Finally, let us regroup in \eqref{sum1}
all the profiles corresponding to the same scales:  namely let us write, for a given~$\ell \in \N$
$$
u_n - \psi_{n}^L = u_{n,1} ^\ell+ u_{n,2}^\ell \, ,
$$
where (up to conveniently re-ordering the profiles~$\phi_{\lambda_{\ell_1}(n)}^{\ell_1},\dots,\phi_{\lambda_{\ell_L}(n)}^{\ell_L}$),
$$
 u_{n,1}^\ell := \sum^{L_\ell}_{k=1} \phi_{\lambda_{\ell_k}(n)}^{\ell_k} \quad \mbox{with} \quad
 j_i (\lambda_{\ell_k} (n)) =  j_i (\lambda_{\ell} (n)) \, , \quad  \forall i \in \{ 1,2\} \, ,
 $$
and on the other hand, writing to simplify~$j_i (\lambda_{\ell} (n)) =: j_i (n)$,
$$
u_{n,2} ^\ell : = \sum^{L}_{k= L_{\ell} +1} \phi_{\lambda_{\ell_k}(n)}^{\ell_k} \, ,
$$
with scales $j_i (\lambda_{\ell_k}(n))$  orthogonal to the scale $j_i(n)$ for every $k \in  \{L_\ell +1,\dots, L\}$. The result~(\ref{differentscalesOK}) enables us to take care of the term~$u_{n,2} ^\ell $ which satisfies
$$
2^{j_1 (\lambda_\ell (n))(-1+\frac2p)+ \frac{j_2 (\lambda_\ell (n))}p}\Big\|\Delta_{j_1 (\lambda_\ell (n)) + a_\ell^1}^{h} \Delta_{j_2 (\lambda_\ell (n))+ a_\ell^2}^v u_{n,2}^\ell \Big \|_{L^p(\R^3)} \to 0,\quad \mbox{as} \quad n \to \infty \, ,$$
so let us prove that
$$
\limsup_{n \to \infty} 2^{j_1 (n) (-1+\frac2p)+\frac{ j_2 (n)}p}  \Big\|\Delta_{j_1 (n)+a_\ell^1}^{h} \Delta_{j_2 (n)+a_\ell^2}^v  u_{n,1}^\ell
\Big \|_{L^p(\R^3)}= C > 0 \, .
$$
By H\"older's inequality if~$2 \leq p \leq \infty$, we have
\begin{equation}\label{holder}
\begin{aligned}
2^{\frac{ j_2 (n)} {2}}  \Big\|\Delta_{j_1 (n)}^{h} \Delta_{j_2 (n)}^v   \,  u_{n,1}^\ell
\Big \|_{L^2(\R^3)}\leq
\left(2^{j_1 (n) + j_2 (n)}  \Big\|\Delta_{j_1 (n)}^{h} \Delta_{j_2 (n)}^v  \,   u_{n,1}^\ell
\Big \|_{L^1(\R^3)}\right) ^\frac{p-2}{2(p-1)}
\\
\times \left(
2^{j_1 (n) (-1+\frac2p)+\frac{ j_2 (n)}p}  \Big\|\Delta_{j_1 (n)}^{h} \Delta_{j_2 (n)}^v \,  u_{n,1}^\ell
\Big \|_{L^p(\R^3)} \right)^{\frac p{2(p-1)}}
\end{aligned}
\end{equation}
and since both terms on the right-hand side are bounded,
 the result will follow if we prove that
$$
\limsup_{n \to \infty} 2^{\frac{ j_2 (n)} {2}}  \Big\|\Delta_{j_1 (n)+a_\ell^1}^{h} \Delta_{j_2 (n)+a_\ell^2}^v    \,u_{n,1}^\ell
\Big \|_{L^2(\R^3)}= C > 0.
$$
But this is a simple orthogonality argument, noticing that
\begin{equation}\label{theL2estimate}
 \begin{aligned}
& \Big\|\Delta_{j_1 (n)+a_\ell^1}^{h} \Delta_{j_2 (n)+a_\ell^2}^v  \,  u_{n,1}
\Big \|_{L^2(\R^3)}^2  = \sum_{k = 1}^{L_\ell} \|
\Delta_{j_1 (n)+a_\ell^1}^{h} \Delta_{j_2 (n)+a_\ell^2}^v   \, \phi_{\lambda_{\ell_k}(n)}^{\ell_k}
\|_{L^2(\R^3)}^2 \\
& \quad + \sum_{k \neq k'} (\Delta_{j_1 (n)+a_\ell^1}^{h} \Delta_{j_2 (n)+a_\ell^2}^v  \,  \phi_{\lambda_{\ell_k}(n)}^{\ell_k} |¬†\Delta_{j_1 (n)+a_\ell^1}^{h} \Delta_{j_2 (n)+a_\ell^2}^v    \,\phi_{\lambda_{\ell_{k'}}(n)}^{\ell_{k'}})_{L^2(\R^3)} \, .
\end{aligned}
\end{equation}
Indeed we know from~(\ref{oneind}) that
\begin{eqnarray}
  2^{\frac{ j_2 (n)} {2}} \left( \sum_{k = 1}^{L_\ell} \|
\Delta_{j_1 (n)+a_\ell^1}^{h} \Delta_{j_2 (n)+a_\ell^2}^v  \, \phi_{\lambda_{\ell_k}(n)}^{\ell_k}
\|_{L^2(\R^3)}^2\right)^\frac12 & \geq &
  2^{\frac{ j_2 (n)} {2}}
 \|
\Delta_{j_1 (n)+a_\ell^1}^{h} \Delta_{j_2 (n)+a_\ell^2}^v   \, \phi_{\lambda_{\ell}(n)}^{\ell}
\|_{L^2(\R^3)} \nonumber \\
& \geq & C^{\ell,2}>0\label{theotherL2estimate}
\end{eqnarray}
so it is enough to prove that
\begin{equation}\label{nointeraction}
 2^{j_2 (n)} \sum_{k \neq k'} (\Delta_{j_1 (n)+a_\ell^1}^{h} \Delta_{j_2 (n)+a_\ell^2}^v   \, \phi_{\lambda_{\ell_k}(n)}^{\ell_k} |¬†\Delta_{j_1 (n)+a_\ell^1}^{h} \Delta_{j_2 (n)+a_\ell^2}^v    \,\phi_{\lambda_{\ell_{k'}}(n)}^{\ell_{k'}})_{L^2(\R^3)} \to 0 \, .
\end{equation}
This is a finite sum so it suffices to prove the result for each individual term, which writes after a change of variables
$$
\int_{\R^3} (\Delta_{a_\ell^1}^h\Delta_{a_\ell^2}^v  \phi^{\ell_k})(x) \times (\Delta_{a_\ell^1}^h\Delta_{a_\ell^2}^v  \phi^{\ell_{k'}}) (x +2^{j_2 (n)} (x_{n,h}^{\ell_k} -x_{n,h}^{\ell_{k'}} )) \, dx
$$
which goes to zero when~$n$ goes to infinity, due to the orthogonality of the cores of concentration (see Theorem~\ref{anisocompthm}), so~(\ref{nointeraction}) holds.

\medskip

\noindent $\bullet \,  $  Finally we need to take the remainder into account. But a reverse triangle inequality gives trivially the result, since the remainder~$\psi_{n}^L$ may be made arbitrarily small in~$\dot B^{-1+\frac2p,\frac1p}_{p,\infty}$ as soon as~$L$ is large enough, uniformly in~$n$, whereas~(\ref{theL2estimate})-(\ref{theotherL2estimate})     guarantee that making~$L$ larger does not decrease the norm of the sum of the profiles.

\medskip
\noindent
The lemma is proved.
\end{proof}
%\begin{rmk}{\rm
%The proof of Lemma~\ref{orthoaniso} (namely estimates~(\ref{holder})-(\ref{theotherL2estimate})) shows that the constant~$C$ in the statement of the lemma is related to the bound on~$(u_n)$. In particular if~$(u_n)$ converges strongly towards zero in~${\mathcal B}^1_q$, then
%}
%\end{rmk}
%
%
\begin{lem}\label{orthoanisodiv1}
Let us consider a sequence~$(v_n)_{n \in \N}$, bounded in~${\mathcal B}^1_q$, which may be decomposed with the notation of  Theorem~{\rm \ref{anisocompthm}} into
$$ v_n=\sum_{\ell=1}^L \phi_{\lambda_\ell(n)}^\ell + \psi_{n}^L \, .
$$
Assume moreover that~$\displaystyle
\lim_{n \to \infty}2^{-j_1 (\lambda_\ell (n)) +j_2 (\lambda_\ell (n))} \in  \{0,\infty\} .$
If~$(\partial_3 v_n)_{n \in \N}$ is  bounded  in~$
\dot B^{0,1}_{1,q}$, then
 $$\displaystyle
 \lim_{n \to \infty}2^{-j_1 (\lambda_\ell (n)) +j_2 (\lambda_\ell (n))} = 0 \,.$$
\end{lem}
\begin{proof} [Proof of Lemma~{\rm \ref{orthoanisodiv1}}]
By definition of~$
\dot B^{0,1}_{1,q}$, we have
$$
 \|\partial_3 v_n\|_{\dot B^{0,1}_{1,q}}= \Big(\sum_{j,k \in \ZZ} 2^{ j q} \|\Delta_k^h \Delta_j^v \partial_3 v_n\|^q_{L^1(\R^3)} \Big)^{1/q}< \infty\quad \mbox{ uniformly in } \, n \, .
 $$
 In particular, for any $\ell \in \{1,..., L\}$, we have
 \begin{equation}\label{bounded}
 2^ {j_2 (\lambda_\ell (n))} \Big\|\Delta_{j_1 (\lambda_\ell (n))}^{h} \Delta_{j_2 (\lambda_\ell (n))}^v \partial_3 v_n \Big \|_{L^1(\R^3)} < \infty\quad \mbox{ uniformly in } \, n \, .
 \end{equation}
Now reasoning as  in the proof of Lemma {\rm \ref{orthoaniso}} and taking into account that $\partial_3 v_n$ is also bounded in $\dot B^{1,0}_{1,q}$,
we find that there are two integers~$a_\ell^1$ and~$a_\ell^2$ such that
 $$
  \limsup_{n \to \infty} 2^{ j_1 (\lambda_\ell (n))}  \Big\|\Delta_{j_1 (\lambda_\ell (n)) +a_\ell^1}^{h} \Delta_{j_2 (\lambda_\ell (n))+a_\ell^2}^v  \partial_3 \phi_{\lambda_\ell(n)}^\ell
\Big \|_{L^1(\R^3)}= C > 0\,,
 $$
and for any $\ell' \neq \ell$
 $$
  2^{j_1 (\lambda_\ell (n))}\Big\|\Delta_{j_1 (\lambda_\ell (n))+a_\ell^1}^{h} \Delta_{j_2 (\lambda_\ell (n))+a_\ell^2}^v \partial_3\phi_{\lambda_{\ell'}(n)}^{\ell'} \Big \|_{L^1(\R^3)} \to 0 \quad\mbox{as}\quad n \to \infty \,.
  $$
Finally,  we argue as in the proof of Lemma {\rm \ref{orthoaniso}} and write
$$
 v_n= v_{n,1} + v_{n,2}+ \psi_{n}^L\, ,
$$
where $v_{n,1}$
contains all the profiles with scale~$j_i (\lambda_{\ell} (n))$, meaning (up to re-ordering the profiles)
$$
v_{n,1} := \sum^{L_\ell}_{k=1} \phi_{\lambda_{\ell_k}(n)}^{\ell_k}\,,
$$
with $\phi_{\lambda_{\ell_k}(n)}^{\ell_k} = 2^{j_1 (\lambda_{\ell} (n))}\phi^{\ell_k}\Big(2^{j_1 (\lambda_{\ell} (n))} (x_h- x_{n,h}^{\ell_k}), 2^{j_2(\lambda_{\ell} (n))}(x_3-x_{n,3}^{\ell_k})\Big)$  and where, denoting~$j_i (n):=j_i (\lambda_{\ell} (n)) $,
$$
v_{n,2} := \sum^{L}_{k= L_ \ell+1} \phi_{\lambda_{\ell_k}(n)}^{\ell_k}\,,
$$
with scales $j(\lambda_{\ell_k}(n))$  orthogonal to the scale $j_i (n)$ for any $k \in  \{L_1 +1,\dots, L\}$. Using the same argument as in the proof of Lemma~\ref{orthoaniso}, we easily prove that for any $\ell \in \{1,\dots, L\}$
$$2^ {j_2 (\lambda_\ell (n))}\Big\|\Delta_{j_1 (\lambda_\ell (n)) +a_\ell^1}^{h} \Delta_{j_2 (\lambda_\ell (n)) +a_\ell^2}^v \partial_3 v_n \Big \|_{L^1(\R^3)} \sim 2^{-j_1 (\lambda_\ell (n)) +j_2 (\lambda_\ell (n))}  \,C, \quad\mbox{as}\quad n \to \infty \,,
 $$
with $C>0$, which concludes the proof of the lemma due to~(\ref{bounded}).
\end{proof}

\begin{lem}\label{orthoanisodiv2} Let us consider ~$(v_n^h = (v^1_n, v^2_n))_{n \in \N}$ a bounded sequence of vector fields in~${\mathcal B}^1_q$ and let us suppose, with the notation of  Theorem~{\rm \ref{anisocompthm}}, that
$$
 v_n^h=\sum_{\ell=1}^L \tilde \phi_{\lambda_\ell(n)}^{\ell,h} + \psi_n^{L,h} \,.
$$
If~${\rm div}_h \:  v_n ^h= 0$,  then for any $\ell \in \{1,..., L\}$  we have
 $\displaystyle
{\rm div}_h \:  \tilde \phi_{\lambda_\ell(n)}^{\ell,h} = 0.$
\end{lem}
\begin{proof}[Proof of Lemma~{\rm \ref{orthoanisodiv2}}]
 We use the notation of the proof of Lemma {\rm \ref{orthoaniso}}. Taking advantage of the fact that the operator  $\mbox{div}_h$ is continuous from ~${\mathcal B}^1_q$ into $\dot B^{0,1}_{1,q}$,  we get,  along the same lines as~(\ref{psiellphiell}) in the proof of Lemma {\rm \ref{orthoaniso}} and recalling that~$\dot B^{0,1}_{1,q}$ embeds in~$\dot B^{\frac2p-2,\frac1p}_{p,q}$,
 $$
  \limsup_{n \to \infty} 2^{ j_1 (\lambda_\ell (n)) (\frac2p-2 )} 2^{ \frac{j_2 (\lambda_\ell (n))}p}  \Big\|\Delta_{j_1 (\lambda_\ell (n))+a_\ell^1}^{h} \Delta_{j_2 (\lambda_\ell (n))+a_\ell^2}^v  \mbox{div}_h \:  \tilde \phi_{\lambda_\ell(n)}^{\ell,h}
\Big \|_{L^p}= \|\widetilde \Psi^\ell \ast \mbox{div}_h \:  \tilde \phi_{\lambda_\ell(n)}^{\ell,h} \|_{L^p }
$$
and  for any $\ell' \neq \ell$, as in~(\ref{differentscalesOK}),
 $$
  2^{j_2 (\lambda_\ell (n))}\Big\|\Delta_{j_1 (\lambda_\ell (n))+a_\ell^1}^{h} \Delta_{j_2 (\lambda_\ell (n))+a_\ell^2}^v \mbox{div}_h \:  \tilde \phi_{\lambda_{\ell'}(n)}^{\ell'} \Big \|_{L^1(\R^3)} \to 0 \quad\mbox{as}\quad n \to \infty \,.
 $$
 Moreover as in~(\ref{nointeraction}),
 $$
 2^{-2j_1(\lambda_\ell (n))} 2^{j_2 (\lambda_\ell (n))} \sum_{k \neq k'} (\Delta_{j_1 (n)+a_\ell^1}^{h} \Delta_{j_2 (n)+a_\ell^2}^v   \, \phi_{\lambda_{\ell_k}(n)}^{\ell_k} |¬†\Delta_{j_1 (\lambda_\ell (n))+a_\ell^1}^{h} \Delta_{j_2 (\lambda_\ell (n))+a_\ell^2}^v    \,\phi_{\lambda_{\ell_{k'}}(n)}^{\ell_{k'}})_{L^2 } \to 0 \, .
 $$
 Then we follow the method giving Lemma~\ref{orthoaniso} which yields  $$
 \begin{aligned}
 0 & = 2^{-2j_1(\lambda_\ell (n))}2^{j_2 (\lambda_\ell (n))} \Big\|\Delta_{j_1 (\lambda_\ell (n))+a_\ell^1}^{h} \Delta_{j_2 (\lambda_\ell (n))+a_\ell^2}^v  \,
{\rm div}_h \:  v_n ^h\Big \|_{L^2(\R^3)}^2  \\
& \geq 2^{-2j_1(\lambda_\ell (n))}2^{j_2 (\lambda_\ell (n))}  \sum_{k = 1}^{L_\ell} \|
\Delta_{j_1 (\lambda_\ell (n))+a_\ell^1}^{h} \Delta_{j_2 (\lambda_\ell (n))+a_\ell^2}^v   \,  \mbox{div}_h \:  \tilde \phi_{\lambda_\ell(n)}^{\ell,h}\|_{L^2(\R^3)}^2+o(1), \quad n \to \infty \\
& \geq  \|\widetilde \Psi^\ell \ast \mbox{div}_h \:  \tilde \phi_{\lambda_\ell(n)}^{\ell,h} \|_{L^2} ^2+o(1), \quad n \to \infty
 \end{aligned}
 $$
so finally~$\widetilde \Psi^\ell \ast \mbox{div}_h \:  \tilde \phi_{\lambda_\ell(n)}^{\ell,h} \equiv 0$ for all couples~$(a_\ell^1,a_\ell^2)$, hence~$ \mbox{div}_h \:  \tilde \phi_{\lambda_\ell(n)}^{\ell,h} \equiv 0$.
\end{proof}
%%%%%%%%%%%%%%%%%%%%%%%%%%%%%%%%%%%%%%%%%%%%%

 \appendix

%%%%%%%%%%%%%%%%%%%%%%%%%%%%%%%%%%%%%%%%%%%%%

 \section{The (perturbed) Navier-Stokes equation in~$\dot B^{-1 + \frac2p, \frac1p}_{p,1} $}\label{globalsmallaniso}
\subsection{Statement of the results}

 In this appendix  it proved that (NS) is globally wellposed for small data in~$\dot B^{-1 + \frac2p, \frac1p}_{p,1} $, using anisotropic techniques  (note that in~\cite{dragos} such a study was undertaken   in the framework of Sobolev spaces). We also study a perturbed Navier-Stokes equation in such spaces.

 \medskip
\noindent  We use the following notation:
   $$
 \begin{aligned}
            {\mathcal S}_{p,q}:=  \widetilde{L^\infty}(\R^+;\dot B^{-1+\frac2p  ,\frac1p  }_{p,q} )  \cap \widetilde{L^1}
         (\R^+;\dot B^{1+\frac2p  ,\frac1p  }_{p,q} \cap \dot B^{ -1+\frac2p  ,2+\frac1p  }_{p,q} ) \, , \\
             {\mathcal S}_{p,q}(T):= \widetilde{L^\infty_{loc}}([0,T [;\dot B^{-1+\frac2p  ,\frac1p  }_{p,q} )  \cap \widetilde{L^1_{loc}}
         ([0,T [;\dot B^{1+\frac2p  ,\frac1p  }_{p,q} \cap \dot B^{ -1+\frac2p  ,2+\frac1p  }_{p,q} )  \, , \\
         {\mathcal X}_{p,q}:=    {\widetilde{L^1} }(\R^+; \dot B^{-1+\frac2p  ,  \frac1p }_{p,q})   +   {\widetilde{L^2} }(\R^+; \dot B^{-1+\frac2p  ,-1+\frac1p  }_{p,q} ) \cap
\widetilde{L^1}(\R^+; \dot B^{\frac2p  ,-1+\frac1p  }_{p,q}) \, ,\\
  {\mathcal Y}_{p,q}:=    {{L^2} }(\R^+; \dot B^{\frac2p  ,\frac1p  }_{p,q} ) \cap  { {L^1} }(\R^+; \dot B^{\frac2p  ,1+\frac1p  }_{p,q} \cap \dot B^{1+\frac2p  ,\frac1p  }_{p,q} )   \, .
\end{aligned}
     $$
       \begin{thm}\label{globalaniso}
Let~$1 \leq p < \infty$ be given.  There is a constant~$c_0$ such that the following result holds. Let $u_0\in \dot B^{-1+\frac 2p,\frac 1p}_{p,1}$ verifying the smallness condition $\|u_0\|_{\dot B^{-1+\frac 2p,\frac 1p}_{p,1}}\leq c_0$. Then, there exists a unique,   global   solution
 $ u$ to~{\rm(NS)} in~$ {\mathcal Y}_{p,1} $, and it satisfies
 $$
 \|u\|_{{\mathcal Y}_{p,1}}  \leq 2 \|u_0 \|_{\dot B^{-1+\frac 2p,\frac 1p}_{p,1}} \, .
 $$

 \noindent If the initial data belongs to~$\dot B^{-1+\frac 2p,\frac 1p}_{p,1}$ with no smallness condition, then there is a maximal time of existence~$T^*>0$  such that there is a unique solution in~$   {\mathcal Y}_{p,1}(T^*)$ and if~$T^*<\infty $ then
  \begin{equation}\label{critereexplosion}
 \lim_{T \to T^*} \|u\|_{\widetilde{L^2} ([0,T] ; \dot B^{\frac 2p,\frac 1p}_{p,1}  )} = \infty \, .
\end{equation}

\noindent If the initial data belongs moreover to~$\dot B^{-1+\frac 2p,\frac 1p}_{p,q}$ with~$q<1$ then the solution belongs to the space~$   {\mathcal Y}_{p,q}(T^*)$, on the same life span. 
%The same goes if the initial data belongs moreover to~$\dot B^{-1+\frac 2{p_1},\frac 1{p_1}}_{p_1,1}$  with~$p_1<p$.

\medskip
\noindent Moreover if~$p<4$ then the spaces~${\mathcal Y}_{p,q}$ can be replaced by~${\mathcal S}_{p,q}$ everywhere. 
  \end{thm}
    \noindent  The next result deals with a perturbed Navier-Stokes system:
 $$
\rm{(NSP)} \ \begin{cases}
\partial_t u+{\mathbb P} (u\cdot\nabla u +U\cdot\nabla u+u\cdot\nabla U) - \Delta u=F \quad\text{in}\quad \R^+\times \R^3\\
u_{|t=0}=u_{0} \, , \quad  \\
\mbox{div} \:  u_0= \mbox{div} \: F = 0 \, .
\end{cases}
$$
 \begin{thm}\label{globalanisoperturbed}
Let~$1 \leq p < 4$ be given.  There is a constant~$c_0$ such that the following result holds.
  Consider  three divergence free vector fields~$u_0 \in \dot B^{-1+\frac2p,\frac 1p}_{p,1}$,~$F \in  {\mathcal X}_{p,1} $ and~$U   \in   {\mathcal Y}_{p,1} $.
 If
$$
\|u_0\|_{\dot B^{-1+\frac2p,\frac 1p}_{p,1}} + \|F\|_{ \mathcal X_{p,1} }  \leq c_0 \exp \big(
- c_0^{-1}  \|U \|_{  {\mathcal Y}_{p,1} }
\big) \, ,
$$
then there is a unique, global solution to~{\rm(NSP)}, in the space
$$
 \widetilde{L^2}
         (\R^+;\dot B^{ \frac2p  ,\frac1p  }_{p,1} \cap \dot B^{ -1+\frac2p  ,1+\frac1p  }_{p,1} ) \cap  \widetilde{L^1}
         (\R^+;\dot B^{1+\frac2p  ,\frac1p  }_{p,1} \cap \dot B^{  \frac2p  ,1+\frac1p  }_{p,1} ) \, .
         $$
          \end{thm}
\noindent The proofs of those two theorems allow to obtain the following strong stability result, which to simplify we only state in the case~$p=1$ since it is the setting of the stability result by weak convergence proved in this paper. We recall that~$ {\mathcal B}^{1}_{1} = \dot B^{1,1}_{1,1}$.
\begin{cor}[Strong stability in $ {\mathcal B}^{1}_{1}$]\label{strongstability}
  Let~$u_0 \in {\mathcal B}^{1}_{1}$ be a divergence free vector field generating a unique solution~$u$ in~$ \widetilde{L^\infty_{loc}}(\R^+; {\mathcal B}^{1}_{1})  \cap \widetilde{L^1_{loc}}
         (\R^+;\dot B^{3  ,1 }_{1,1} \cap \dot B^{1 ,3}_{1,1} )$. Then~$u $ belongs to~${\mathcal S}_{1,1}$ and~$\|u(t)\|_{{\mathcal B}^{1}_{1}} \to 0$ as~$t\to \infty$.

\noindent Moreover
         there is~$\e_0$ such that  any~$v_0 \in {\mathcal B}^{1}_{1}$ satisfying~$\|u_0 - v_0\|_{ {\mathcal B}^{1}_{1}} \leq \e_0$ generates a unique global solution in~${\mathcal S}_{1,1}$.
\end{cor}
\subsection{Proof of Theorem~\ref{globalaniso}}
We shall proceed in several steps:
\begin{enumerate}
\item  If~$u_0$ belongs to~$\dot B^{-1+\frac2p  ,\frac1p  }_{p,1}$, we prove that a fixed point may be performed in the Banach space~${\widetilde {L^2}}(\R^+;\dot B^{\frac2p  ,\frac1p  }_{p,1}  )$, which   implies the existence and uniqueness of a solution in that space for small data.

\smallskip
\item  We then prove  that the  solution constructed in the previous step actually  belongs to~${\mathcal Y}_{p,1}$, and to~${\mathcal S}_{p,1}$ if~$p<4$, and we prove that any "almost global solution" belongs to~${\mathcal S}_{p,1}$ and decays to zero at infinity. 

\smallskip
\item We deduce from the  estimates leading to the above steps the result for large data.

\smallskip
\item We   prove the propagation of regularity in~${\mathcal S}_{p,q}$ for $q<1$.
%, and in~${\mathcal S}_{p_1,q}$ for~$p_1<p$.

\end{enumerate}

 \medskip
\noindent $(1) $ $ $  Let us start by applying a fixed point theorem in the Banach space~${\widetilde {L^2}}(\R^+;\dot B^{\frac2p  ,\frac1p  }_{p,1}  )$, to~(NS) written in integral form:
$$
u(t) = e^{t\Delta} u_0 - \int_0^t e^{(t-t') \Delta}\,  {\mathbb P} \mbox{div} \, (u \otimes u) (t') \, dt' \, ,
$$
recalling that~$ {\mathbb P} := \rm{I} - \nabla \Delta^{-1}  \rm{div} $ is the Leray projector onto divergence free vector fields.
We first notice that (see Proposition~\ref{propaniso})
$$
\|e^{t\Delta}\Delta_k^h\Delta_j^v u_0\|_{L^p}\lesssim e^{-ct(2^{2k}+2^{2j})}\|\Delta_k^h\Delta_j^v u_0\|_{L^p} \, ,
$$
so one sees immediately that for any~$1 \leq r \leq \infty$ and for any~$0 \leq \sigma \leq 2/r$,
\begin{equation}\label{heatflow}
\|e^{t\Delta} u_0\|_{{\widetilde {L^r}}(\R^+;\dot B^{-1+\frac2p+\sigma  , \frac2r - \sigma + \frac1p  }_{p,1}   )} \lesssim \|u_0\|_{\dot B^{-1+\frac2p  ,\frac1p  }_{p,1}} \, .
\end{equation}
Now let us turn to the non linear term. Defining
$$
B(u,u) (t):= -\int_0^t e^{(t-t') \Delta} \, {\mathbb P} \mbox{div} \, (u \otimes u) (t') \, dt' \, ,
$$
we have
$$
2^{\frac{2k}p + \frac jp} \|\Delta_k^h\Delta_j^v B(u,u)(t) \|_{L^p} \lesssim
\int_0^t e^{-c(t-t') (2^{2k}+2^{2j})} (2^k + 2^j)2^{\frac{2k}p + \frac jp} \|\Delta_k^h\Delta_j^v  (u \otimes u) (t') \|_{L^p} \, dt' \, .
$$
The space~$ \dot B^{\frac2p  ,\frac1p  }_{p,1}$ is an algebra according to~(\ref{algebra}) so we have
\begin{equation}\label{uotimesu}
\|u \otimes u\|_{{  L^1}(\R^+;\dot B^{\frac2p  ,\frac1p  }_{p,1})} \lesssim  \|u\|_{{\widetilde {L^2}}(\R^+;\dot B^{\frac2p  ,\frac1p  }_{p,1})} ^2 \, .
\end{equation}
It follows that
\begin{equation}\label{firstestimateL2}
2^{\frac{2k}p + \frac jp} \|\Delta_k^h\Delta_j^v B(u,u)(t) \|_{L^p} \lesssim \|u\|_{{\widetilde {L^2}}(\R^+;\dot B^{\frac2p  ,\frac1p  }_{p,1})} ^2
\int_0^t e^{-c(t-t') (2^{2k}+2^{2j})} (2^k + 2^j) c_{jk} (t') \, dt' \, ,
\end{equation}
where~$c_{jk} (t')$ belongs to~$\ell^1_{jk}(L^1_{t'})$ and  Young's inequality in time gives
\begin{equation}\label{estimateBuu}
\|B(u,u) \|_{{\widetilde {L^2}}(\R^+;\dot B^{\frac2p  ,\frac1p  }_{p,1})
}
 \lesssim \|u\|_{{\widetilde {L^2}}(\R^+;\dot B^{\frac2p  ,\frac1p  }_{p,1})}^2 \, .
\end{equation}
The small data result follows classically from~(\ref{heatflow}) and~(\ref{estimateBuu})
by a fixed point in~${\widetilde {L^2}}(\R^+;\dot B^{\frac2p  ,\frac1p  }_{p,1})$.

 \medskip
 \noindent $(2) $ $ $ Now let us prove that
  the solution actually belongs to~${\mathcal Y}_{p,1}$.  We first notice that the above computations actually imply that the solution~$u$ belongs to~$L^1   (\R^+;\dot B^{1+\frac2p  ,\frac1p  }_{p,1} \cap \dot B^{  \frac2p  , 1+\frac1p  }_{p,1} )$.  Indeed that   holds for the term~$e^{t\Delta} u_0$ due to~(\ref{heatflow}) so we just need to concentrate on the bilinear term.  We return to estimate~(\ref{firstestimateL2}) and  consider any real number~$r \in [1,\infty]$. Using~(\ref{uotimesu}), we can write
 for any~$\sigma \in \R$
  $$
  \begin{aligned}
I_{jk}(t) &:=   2^{k(-1+\frac2p  +\sigma)}   2^{j( \frac2r - \sigma +\frac1p ) } \|\Delta_k^h\Delta_j^v B(u,u) \|_{L^p} \\
& \lesssim \|u\|_{{\widetilde {L^2}}(\R^+;\dot B^{\frac2p  ,\frac1p  }_{p,1})}^2
\int_0^t e^{-c(t-t') (2^{2k}+2^{2j})} (2^k + 2^j) 2^{k(-1+\frac2p  +\sigma - \frac{2}p)}   2^{j( \frac2r - \sigma +\frac1p - \frac1p ) }  c_{jk} (t') \, dt'  \, ,
\end{aligned}
  $$
  where again~$c_{jk} (t')$ belongs to~$\ell^1_{jk}(L^1_{t'})$. We want to prove that~$I_{jk}(t)$ belongs to~$\ell^1_{jk}(L^r_{t'})$. We apply a Young inequality in the time variable, which produces
  \begin{equation}\label{estimateijkjk}
  \|I_{jk}\|_{L^r}\lesssim
   \|u\|_{{\widetilde {L^2}}(\R^+;\dot B^{\frac2p  ,\frac1p  }_{p,1})}^2(2^{2k}+2^{2j})^{-\frac1r}(2^k + 2^j) 2^{k(-1+\frac2p  +\sigma - \frac{2}p)}   2^{j( \frac2r - \sigma +\frac1p - \frac1p ) } d_{jk}    \, ,
  \end{equation}
  with~$d_{jk}\in \ell^1_{jk}$. An easy computation shows that the sequence bounding~$  \|I_{jk}\|_{L^r}$ is bounded in~$ \ell^1 _{jk}  $ as soon as one has~$1 \leq\sigma \leq 2/r    $. This implies in particular that~$u$ belongs to the space~$L^1   (\R^+;\dot B^{1+\frac2p  ,\frac1p  }_{p,1} \cap \dot B^{  \frac2p  , 1+\frac1p  }_{p,1} )$ as claimed.

\begin{rmk}\label{rmkjk}{\rm
Note in passing that if~$2^k + 2^j$ was replaced by~$2^k$   on the right-hand side of~(\ref{estimateijkjk}), then one would recover directly the whole range~$0\leq\sigma \leq 2/r    $. Here we need an extra step because of the presence of~$2^j$.
}
\end{rmk}

 \smallskip
 \noindent  From now  on we assume that~$p<4$, and we want to extend this result to any degree of integrability in time, as well as to the space~$L^1(\R^+;\dot B^{-1+\frac2p   , 2   +\frac1p  }_{p,1} )$. Let us start with the case~$r = \infty$. Due to the smallness of~$u_0$ and to the result we just found,
  it is enough to prove that
\begin{equation}\label{linfty}
  \|B(u,u)\|_{{\widetilde {L^\infty}}(\R^+;\dot B^{-1+\frac2p    , \frac1p  }_{p,1} )} \lesssim  \|u\|_{{\widetilde {L^\infty}}(\R^+;\dot B^{-1+\frac2p    , \frac1p  }_{p,1} )} \|u\|_{{  {L^1}}(\R^+; \dot B^{1+\frac2p    , \frac1p  }_{p,1} \cap \dot B^{ \frac2p    ,1+ \frac1p  }_{p,1} )} \,
  \end{equation}
  since~(\ref{heatflow}) takes care of~$e^{t\Delta}u_0$.  But we have, if~$p<4$,
\begin{equation}\label{estimatewithwidetilde}
  \begin{aligned}
 \| u \cdot \nabla u \|_{{  {L^1}}(\R^+; \dot B^{-1+\frac2p    , \frac1p  }_{p,1}  )}
  & \leq \| u^h \cdot \nabla^h u\|_{  {L^1}(\R^+; \dot B^{-1+\frac2p    , \frac1p  }_{p,1}  )} +
   \| u^3   \partial_3 u\|_{{  {L^1}}(\R^+; \dot B^{-1+\frac2p    , \frac1p  }_{p,1}  )} \\
    & \lesssim \|u\|_{{\widetilde  {L^\infty}}(\R^+;\dot B^{-1+\frac2p    , \frac1p  }_{p,1} )} \Big( \|  u\|_{L^1(\R^+; \dot B^{1+\frac2p    , \frac1p  }_{p,1}   )}
   +  \|  u\|_{{  {L^1}}(\R^+;  \dot B^{ \frac2p    ,1+ \frac1p  }_{p,1} )} \Big)
\end{aligned}
  \end{equation}
  by the product laws~(\ref{quasialgebra}) recalled in Appendix~\ref{appendixlp}, and the result follows exactly as above: on the one hand~(\ref{estimatewithwidetilde}) gives
   $$
  \begin{aligned}
J_{jk}(t) &:=   2^{k(-1+\frac2p  )}   2^{  \frac jp   } \|\Delta_k^h\Delta_j^v B(u,u) \|_{L^p} \\
& \lesssim
\int_0^t e^{-c(t-t') (2^{2k}+2^{2j})}  2^{k(-1+\frac2p  )}   2^{ \frac jp  }  2^{-k(-1+\frac2p) } 2^{- \frac jp } c_{jk} (t') \, dt' \\
& \quad \quad \times  \|u\|_{{\widetilde  {L^\infty}}(\R^+;\dot B^{-1+\frac2p    , \frac1p  }_{p,1} )} \Big( \|  u\|_{{\widetilde {L^1}}(\R^+; \dot B^{1+\frac2p    , \frac1p  }_{p,1}   )}
   +  \|  u\|_{{  {L^1}}(\R^+;  \dot B^{ \frac2p    ,1+ \frac1p  }_{p,1} )} \Big)   \, ,
\end{aligned}
  $$
  with~$ c_{jk} (t) \in \ell^1_{jk}(L^1_t) $, hence
  $$
  \begin{aligned}
 \| B(u,u)\|_{{
  \widetilde  {L^\infty}
  }
  (\R^+;\dot B^{-1+\frac2p    , \frac1p  }_{p,1} )}& \leq \|J_{jk} \|_{\ell^1_{jk}((L^\infty_t)} \\
  & \lesssim  \|u\|_{{\widetilde  {L^\infty}}(\R^+;\dot B^{-1+\frac2p    , \frac1p  }_{p,1} )} \Big( \|  u\|_{L^1(\R^+; \dot B^{1+\frac2p    , \frac1p  }_{p,1}   )}
   +  \|  u\|_{L^1(\R^+;  \dot B^{ \frac2p    ,1+ \frac1p  }_{p,1} )} \Big)   \, ,
   \end{aligned}
$$
which proves~(\ref{linfty}). On the other hand
    $$
  \begin{aligned}
K_{jk}(t) &:=   2^{k(-1+\frac2p  )}   2^{ j(2+ \frac 1p)   } \|\Delta_k^h\Delta_j^v B(u,u) \|_{L^p} \\
& \lesssim
\int_0^t e^{-c(t-t') (2^{2k}+2^{2j})}  2^{k(-1+\frac2p  )}   2^{  j(2+ \frac 1p)    }  2^{-k(-1+\frac2p) } 2^{- \frac jp } c_{jk} (t') \, dt' \\
& \quad \quad \times  \|u\|_{{\widetilde  {L^\infty}}(\R^+;\dot B^{-1+\frac2p    , \frac1p  }_{p,1} )} \Big( \|  u\|_{L^1(\R^+; \dot B^{1+\frac2p    , \frac1p  }_{p,1}   )}
   +  \|  u\|_{L^1(\R^+;  \dot B^{ \frac2p    ,1+ \frac1p  }_{p,1} )} \Big)   \, ,
\end{aligned}
  $$
  with~$ c_{jk} (t) \in \ell^1_{jk}(L^1_t) $, hence
   $$
  \begin{aligned}
 \| B(u,u)\|_{{
     {L^1}
  }
  (\R^+;\dot B^{-1+\frac2p    , 2+\frac1p  }_{p,1} )}& \leq \|K_{jk} \|_{\ell^1_{jk}(L^1_t)} \\
  & \lesssim  \|u\|_{{\widetilde  {L^\infty}}(\R^+;\dot B^{-1+\frac2p    , \frac1p  }_{p,1} )} \Big( \|  u\|_{L^1(\R^+; \dot B^{1+\frac2p    , \frac1p  }_{p,1}   )}
   +  \|  u\|_{L^1(\R^+;  \dot B^{ \frac2p    ,1+ \frac1p  }_{p,1} )} \Big)   \, .
   \end{aligned}
$$
We conclude that if the initial data is small enough, then the solution belongs to~${\mathcal S}_{p,1}$.
\begin{rmk}\label{anicercase}{\rm
It is easy to see, using Remark~\ref{rmkjk} for instance, that one could add an exterior force, small enough in~$ {\widetilde{L^1} }(\R^+; \dot B^{-1+\frac2p  ,  \frac1p }_{p,q})  $, and the small data result would be identical.
}
  \end{rmk}
  \begin{rmk}{\rm
Note that all the estimates can be restricted to a time interval~$[a,b] $ of~$\R^+$.
}
  \end{rmk}
  \begin{rmk}\label{removetilde}{\rm
The~$\widetilde {L^\infty}(\R^+;\dot B^{-1+\frac2p    , \frac1p  }_{p,1} )$ norm on the right-hand side of~(\ref{estimatewithwidetilde}) can be replaced by the (smaller)~$ L^\infty(\R^+;\dot B^{-1+\frac2p    , \frac1p  }_{p,1} )$ norm. The same goes for the~${\widetilde {L^2}}(\R^+;\dot B^{\frac2p  ,\frac1p  }_{p,1}) $ norm in~(\ref{uotimesu}), which can be replaced by the~${ {L^2}}(\R^+;\dot B^{\frac2p  ,\frac1p  }_{p,1}) $ norm. This will be useful in the proof of Theorem~\ref{globalanisoperturbed}.
}
  \end{rmk}
 
    \bigskip
 \noindent $(3) $ $ $ It is classical that the    previous estimates can be adapted  to the case of large initial data (for instance by solving first the heat equation and then a perturbed Navier-Stokes equation, of the same type as in the proof of Theorem~\ref{globalanisoperturbed} below) and we leave this     to the reader.

 \bigskip

 \noindent $(4) $ $ $
\noindent  Now we are left with the proof of the propagation of regularity result.
Again this is an easy exercise based on the fact that Young's inequality for sequences are true in~$\ell^q$ with~$q >0$ so we can simply   copy  the above arguments. 
%The propagation of regularity in~$\dot B^{-1+\frac2{p_1}    , \frac1{p_1}   }_{{p_1} ,1}$ follows from  classical arguments.

 \bigskip

\noindent Theorem~\ref{globalaniso} is proved.
  \qed

\subsection{Proof of Theorem~\ref{globalanisoperturbed}}
We shall follow the proof of Theorem~\ref{globalaniso} above, writing~(NSP) under the integral form
$$
u(t) = e^{t\Delta }u_0 - \int_0^t e^{(t-t') \Delta} \,  {\mathbb P} \Big( \mbox{div} \, (u \otimes u + U\otimes u +  u \otimes  U )+ F\Big) (t')  \, dt' \, .
$$
The linear term~$ e^{t\Delta }u_0$ and the term involving~$\mbox{div} \, (u \otimes u )$ (called~$B(u,u)$ in the previous proof) have already been dealt with and we know that in particular for any~$a < b$ and any~$1 \leq r \leq \infty$,
\begin{equation}\label{1}
\forall \, 0 \leq \sigma \leq \frac2r \, , \quad \|e^{t\Delta} u_0\|_{{\widetilde {L^r}}([a,b];\dot B^{-1+\frac2p+\sigma  , \frac2r - \sigma + \frac1p  }_{p,1}   )} \lesssim \|u_0\|_{\dot B^{-1+\frac2p  ,\frac1p  }_{p,1}} \, .
\end{equation}
We have as well
\begin{equation}\label{2}
\|B(u,u) \|_{{\widetilde {L^2}}([a,b];\dot B^{\frac2p  ,\frac1p  }_{p,1} )}
+  \|B(u,u)\|_{ {L^1}([a,b];\dot B^{ \frac2p    , 1+\frac1p  }_{p,1}
\cap \dot B^{1+\frac2p  , \frac1p  }_{p,1} )} \lesssim \|u\|_{{\widetilde {L^2}}([a,b];\dot B^{\frac2p  ,\frac1p  }_{p,1})}^2 \, ,
\end{equation}
and if~$1 \leq p < 4$,
\begin{equation}\label{3}
   \begin{aligned}
\|B(u,u)\|_{{ \widetilde{L^\infty}}(\R^+;\dot B^{-1+\frac2p    , \frac1p  }_{p,1} )}
 +  \|B(u,u)\|_{{ \widetilde{L^2}}(\R^+;\dot B^{-1+\frac2p    ,1+ \frac1p  }_{p,1} )}
 +  \|B(u,u)\|_{{  {L^1}}(\R^+;\dot B^{-1+\frac2p    , 2+\frac1p  }_{p,1} )} \\
\quad \lesssim  \|u\|_{{ {L^\infty}}(\R^+;\dot B^{-1+\frac2p    , \frac1p  }_{p,1} )} \|u\|_{L^1(\R^+; \dot B^{1+\frac2p    , \frac1p  }_{p,1} ¬†\cap \dot B^{ \frac2p    ,1+ \frac1p  }_{p,1} )} \, .
  \end{aligned}
\end{equation}
Note that the estimate in~${\widetilde {L^2}}(\R^+;\dot B^{-1+\frac2p    ,1+ \frac1p  }_{p,1} )$ appearing in~(\ref{3}) is a consequence of an interpolation between the spaces~${\widetilde {L^\infty}}(\R^+;\dot B^{-1+\frac2p    , \frac1p  }_{p,1} )$ and~${  {L^1}}(\R^+;\dot B^{-1+\frac2p    , 2+\frac1p  }_{p,1} )$.

\smallskip
\noindent
Now let us study the term containing the force~$F$. We define
$$
\begin{aligned}
{\mathcal F}(t):=\int_0^t e^{(t-t') \Delta} \, {\mathbb P}F(t') \, dt' \, , \quad \mbox{with} \quad F_1 \in   { {L^1} }(\R^+; \dot B^{-1+\frac2p  ,  \frac1p }_{p,1})
 \quad \mbox{and} \\
 F_2 \in {\widetilde{L^2} }(\R^+; \dot B^{-1+\frac2p  ,-1+\frac1p  }_{p,1} ) \cap
 {L^1}(\R^+; \dot B^{\frac2p  ,-1+\frac1p  }_{p,1}) \, .
\end{aligned}
$$
On the one hand the above arguments (see the estimates of~$I_{jk}$ and~$K_{jk}$, or simply Remark~\ref{rmkjk}) enable us to write directly that for all~$\sigma \in [0,2]$,
\begin{equation}\label{41}
\| {\mathcal F} \|_{{\widetilde {L^\infty}}([a,b];\dot B^{-1+\frac2p  ,\frac1p  }_{p,1}  )} + \| {\mathcal F} \|_{ {L^1}([a,b];\dot B^{ -1+\sigma+ \frac2p    , 2- \sigma +\frac1p  }_{p,1} ) }
\lesssim\|F_1\|_{ {L^1}([a,b];\dot B^{-1+\frac2p  ,\frac1p  }_{p,1} )}
\end{equation}
while
 for all~$1 \leq \sigma \leq 2$,
\begin{equation}\label{4}
\| {\mathcal F} \|_{{\widetilde {L^2}}([a,b];\dot B^{\frac2p  ,\frac1p  }_{p,1}  )} + \| {\mathcal F} \|_{ {L^1}([a,b];\dot B^{ -1+\sigma+ \frac2p    , 2- \sigma +\frac1p  }_{p,1} ) }
\lesssim\|F_2\|_{ {L^1}([a,b];\dot B^{\frac2p  ,-1+\frac1p  }_{p,1}) } \, .
\end{equation}
On the other hand the same
computations as in the proof of Theorem~\ref{globalaniso}  give easily
\begin{equation}\label{5}
\| {\mathcal F} \|_{{\widetilde {L^\infty}}([a,b];\dot B^{-1+\frac2p  ,\frac1p  }_{p,1} )\cap {\widetilde {L^2}} ([a,b];\dot B^{-1+\frac2p    , 1+\frac1p  }_{p,1} )}
 \lesssim  \|F_2\|_{   {\widetilde{L^2} }(\R^+; \dot B^{-1+\frac2p  ,-1+\frac1p  }_{p,1} ) } \, .
\end{equation}
Finally let us turn to the contribution of~$U$. We define
$$
{\mathcal U}(t):=-\int_0^t e^{t \Delta} \, {\mathbb P}\mbox{div} \, (u \otimes U + U\otimes u) (t') \, dt' \, .
$$
%Since~$ \dot B^{\frac2p  ,\frac1p  }_{p,1}$ is an algebra we have, as in~(\ref{uotimesu}) (see also Remark~\ref{removetilde}),
%$$
%\begin{aligned}
%\|u \otimes U\|_{{  L^1}([a,b];\dot B^{\frac2p  ,\frac1p  }_{p,1})} & \lesssim  \|u\|_{{  {L^2}}([a,b];\dot B^{\frac2p  ,\frac1p  }_{p,1})}  \|U\|_{{ {L^2}}([a,b];\dot B^{\frac2p  ,\frac1p  }_{p,1})} \\
%& \lesssim  \|u\|_{{\widetilde {L^2}}([a,b];\dot B^{\frac2p  ,\frac1p  }_{p,1})}  \|U\|_{{ {L^2}}([a,b];\dot B^{\frac2p  ,\frac1p  }_{p,1})} \, ,
%  \end{aligned}
%  $$
%and exactly the same estimates as in the first step of the proof of Theorem~\ref{globalaniso} give, for all indexes~$1 \leq \sigma \leq 2$,
%\begin{equation}\label{6}
%\|{\mathcal U}\|_{{\widetilde {L^2}}([a,b];\dot B^{\frac2p  ,\frac1p  }_{p,1} )\cap {L^1}([a,b];\dot B^{ -1+\sigma+ \frac2p    , 2- \sigma+ \frac1p  }_{p,1} ) }
%\lesssim  \|u\|_{{\widetilde {L^2}}([a,b];\dot B^{\frac2p  ,\frac1p  }_{p,1})}  \|U\|_{{  {L^2}}([a,b];\dot B^{\frac2p  ,\frac1p  }_{p,1})}\, .
%\end{equation}
%On the other hand w
We can write  using~(\ref{quasialgebra}) (and Remark~\ref{removetilde})
$$
\begin{aligned}
\|u^h \cdot \nabla^h U + u^3   \partial_3 U\|_{L^1([a,b]; \dot B^{-1+\frac2p  ,\frac1p  }_{p,1})} & \lesssim \|u\|_{{  {L^\infty}}([a,b];\dot B^{-1+\frac2p  ,\frac1p  }_{p,1} )}   \|U\|_{L^1([a,b]; \dot B^{ 1+\frac2p  ,\frac1p  }_{p,1} \cap  \dot B^{ \frac2p  ,1+\frac1p  }_{p,1})}\\
& \lesssim \|u\|_{{\widetilde {L^\infty}}([a,b];\dot B^{-1+\frac2p  ,\frac1p  }_{p,1} )}   \|U\|_{L^1([a,b]; \dot B^{ 1+\frac2p  ,\frac1p  }_{p,1} \cap  \dot B^{ \frac2p  ,1+\frac1p  }_{p,1})}
  \end{aligned}
$$
and using~(\ref{quasialgebra}) again,
$$
\|U^h \cdot \nabla^h u + U^3   \partial_3 u\|_{L^1([a,b]; \dot B^{-1+\frac2p  ,\frac1p  }_{p,1})} \lesssim \|u\|_{{\widetilde {L^2}}([a,b];\dot B^{-1+\frac2p  ,1+\frac1p  }_{p,1} \cap \dot B^{  \frac2p  , \frac1p  }_{p,1})}   \|U\|_{{  {L^2}}([a,b]; \dot B^{  \frac2p  ,\frac1p  }_{p,1})}\, .
$$
This enables us to write 
\begin{eqnarray}\label{7}
\|{\mathcal U}\|_{{\widetilde {L^\infty}}([a,b];\dot B^{-1+\frac2p  ,\frac1p  }_{p,1}  )\cap { {\widetilde {L^2}}([a,b];\dot B^{-1+\frac2p  ,1+\frac1p  }_{p,1}})}
\lesssim\Big(  \|u\|_{{\widetilde {L^\infty}}([a,b];\dot B^{-1+\frac2p  ,\frac1p  }_{p,1} )}   \|U\|_{L^1([a,b]; \dot B^{ 1+\frac2p  ,\frac1p  }_{p,1} \cap  \dot B^{ \frac2p  ,1+\frac1p  }_{p,1})} \nonumber&\\
+  \|u\|_{{\widetilde {L^2}}([a,b];\dot B^{-1+\frac2p  ,1+\frac1p  }_{p,1} \cap \dot B^{  \frac2p  , \frac1p  }_{p,1})}   \|U\|_{{  {L^2}}([a,b]; \dot B^{  \frac2p  ,\frac1p  }_{p,1})}\Big) \, .&
\end{eqnarray}
Putting   estimates~(\ref{1}), (\ref{2}), (\ref{41}), (\ref{4}), (\ref{7}) together we infer that
\begin{eqnarray}\label{eq:alpha}
\|u\|_{{\widetilde {L^2}}([a,b];\dot B^{\frac2p  ,\frac1p  }_{p,1})\cap  {L^1}([a,b];\dot B^{ \frac2p    , 1+\frac1p  }_{p,1}
\cap \dot B^{1+\frac2p  , \frac1p  }_{p,1} )} \leq C \Big(
 \|u\|_{\widetilde {L^2}([a,b];\dot B^{\frac2p  ,\frac1p  }_{p,1})}^2& \\
+     \|u\|_{{\widetilde {L^2}}([a,b];\dot B^{\frac2p  ,\frac1p  }_{p,1})}  \|U\|_{{  {L^2}}([a,b];\dot B^{\frac2p  ,\frac1p  }_{p,1})}+ \|u(a)\|_{\dot B^{-1+\frac2p  ,\frac1p  }_{p,1}}  + \|F\|_{{\mathcal X}_{p,1}}
\Big) \,,&\nonumber
\end{eqnarray}
while estimates~(\ref{1}), (\ref{3}), (\ref{5}), (\ref{7}) give
\begin{eqnarray}\label{eq:beta}
\|u\|_{{\widetilde {L^\infty}}([a,b];\dot B^{-1+\frac2p  ,\frac1p  }_{p,1}  )\cap { {\widetilde {L^2}}([a,b];\dot B^{-1+\frac2p  ,1+\frac1p  }_{p,1}})}
\leq C \Big(
 \|u\|_{{\widetilde {L^\infty}}([a,b];\dot B^{-1+\frac2p    , \frac1p  }_{p,1} )} \|u\|_{L^1([a,b]; \dot B^{1+\frac2p    , \frac1p  }_{p,1} ¬†\cap \dot B^{ \frac2p    ,1+ \frac1p  }_{p,1} )} & \nonumber\\
  +   \|u\|_{{\widetilde {L^\infty}}([a,b];\dot B^{-1+\frac2p  ,\frac1p  }_{p,1} )}   \|U\|_{L^1([a,b]; \dot B^{ 1+\frac2p  ,\frac1p  }_{p,1} \cap  \dot B^{ \frac2p  ,1+\frac1p  }_{p,1})} &\nonumber \\
 +  \|u\|_{{\widetilde {L^2}}([a,b];\dot B^{-1+\frac2p  ,1+\frac1p  }_{p,1} \cap \dot B^{  \frac2p  , \frac1p  }_{p,1})}   \|U\|_{{  {L^2}}([a,b]; \dot B^{  \frac2p  ,\frac1p  }_{p,1})}+ \|u(a)\|_{\dot B^{-1+\frac2p  ,\frac1p  }_{p,1}}  + \|F\|_{{\mathcal X}_{p,1}}
\Big) \,.&
 \end{eqnarray}
To conclude we resort to a   Gronwall-type argument (see for instance~\cite{gip} for a similar argument):
there exist~$ N$ real numbers~$ (T_i)_{1 \leq i \leq N}$
such that~$ T_1 = 0$ and~$T_N = +\infty $, such that~$\displaystyle \R_+ = \bigcup_{i = 1}^{N-1} [T_i,T_{i+1}] $ and satisfying
\begin{equation}\label{eq:usmall}
\|U\|_{{  {L^2}}([T_i,T_{i+1}];\dot B^{\frac2p  ,\frac1p  }_{p,1})} +  \|U\|_{L^1([T_i,T_{i+1}]; \dot B^{ 1+\frac2p  ,\frac1p  }_{p,1} \cap  \dot B^{ \frac2p  ,1+\frac1p  }_{p,1})} \leq
\frac{1}{8C} \quad \forall i \in \{1,\dots,N-1\} \, .
\end{equation}
Then suppose that
\begin{equation}
\label{eq:initialdata}
\begin{aligned}
 \|u_0\|_{\dot B^{-1+\frac2p  ,\frac1p  }_{p,1}} + \|F\|_{{\mathcal X}_{p,1}}
   \leq \frac{1}{8CN(2C)^N } \,\cdotp
   \end{aligned}
\end{equation}
By time continuity we can define a maximal time~$ T \in \R^+ \cup
\{\infty\}$ such that
\begin{equation}
\label{eq:defT}
\|u\|_{{\widetilde {L^2}}([0,T];\dot B^{\frac2p  ,\frac1p  }_{p,1})}+  \|u\|_{ {L^1}([0,T];\dot B^{ \frac2p    , 1+\frac1p  }_{p,1}
\cap \dot B^{1+\frac2p  , \frac1p  }_{p,1} )}  \leq
\frac{1}{4C} \, \cdotp
\end{equation}
If~$ T = \infty$ then the theorem is proved. Suppose now that~$ T <
+\infty$. Then we can define an integer~$ k \in \{1,\dots,N-1\}$ such
that
$$
T_k \leq T < T_{k+1} \, ,
$$
and plugging~(\ref{eq:usmall}) and~(\ref{eq:defT})
into~(\ref{eq:alpha})  we get for any~$ i \leq k-1$
$$
\begin{aligned}
\|u\|_{{\widetilde {L^2}}([T_i,T_{i+1}];\dot B^{\frac2p  ,\frac1p  }_{p,1})}¬†+   \|u\|_{ {L^1}([T_i,T_{i+1}];\dot B^{ \frac2p    , 1+\frac1p  }_{p,1}
\cap \dot B^{1+\frac2p  , \frac1p  }_{p,1} )}
\leq C \|u(T_{i})\|_{\dot B^{-1+\frac2p  ,\frac1p  }_{p,1}} \\ {}+ C \|F\|_{{\mathcal X}_{p,1}}
 + \frac{1}{4}
\|u\|_{\widetilde{L^{2}}([T_i,T_{i+1}];\dot B^{\frac2p  ,\frac1p  }_{p,1})}
 +
\frac{1}{4} \|u\|_{\widetilde{L^{2}}([T_i,T_{i+1}];\dot B^{\frac2p  ,\frac1p  }_{p,1})}
\,  ,
  \end{aligned}$$
so finally
\begin{equation}
  \label{di-r}
\begin{aligned}
\|u\|_{{\widetilde {L^2}}([T_i,T_{i+1}];\dot B^{\frac2p  ,\frac1p  }_{p,1})}¬†+   \|u\|_{ {L^1}([T_i,T_{i+1}];\dot B^{ \frac2p    , 1+\frac1p  }_{p,1}
\cap \dot B^{1+\frac2p  , \frac1p  }_{p,1} )}\\
  \leq 2 C  \Big(\|u(T_{i})\|_{\dot B^{-1+\frac2p  ,\frac1p  }_{p,1}} + \|F\|_{{\mathcal X}_{p,1}}   \Big)  \, .
 \end{aligned}
 \end{equation}
From relations~(\ref{eq:alpha}) and~\eqref{eq:beta}   we also get
\begin{equation}
  \label{di-inf}
 \begin{aligned}\|u\|_{{\widetilde {L^\infty}}([T_i,T_{i+1}];\dot B^{-1+\frac2p    , \frac1p  }_{p,1} )} \leq 2
C \Big( \|u(T_{i})\|_{\dot B^{-1+\frac2p    , \frac1p  }_{p,1} }
+ \|F\|_{{\mathcal X}_{p,1}}    \Big).
 \end{aligned}
 \end{equation}
Since~${\widetilde {L^\infty}}(\R^+;\dot B^{-1+\frac2p    , \frac1p  }_{p,1} )\subset {  {L^\infty}}(\R^+;\dot B^{-1+\frac2p    , \frac1p  }_{p,1} )$, we further infer
that
$$
\|u(T_{i+1})\|_{\dot B^{-1+\frac2p    , \frac1p  }_{p,1} }\leq 2
C \Big( \|u(T_{i})\|_{\dot B^{-1+\frac2p    , \frac1p  }_{p,1} } + \|F\|_{ {\mathcal X}_{p,1} }  \Big) \, .$$
A trivial induction now shows that for all~$ i\in\{1,\dots,k-1\}$,
\begin{equation*}
\|u(T_{i})\|_{\dot B^{-1+\frac2p    , \frac1p  }_{p,1}}\leq  (2C)^{i-1} \Big(\|u_0\|_{\dot B^{-1+\frac2p    , \frac1p  }_{p,1}}+ \|F\|_{{\mathcal X}_{p,1} }  \Big)\,   .
\end{equation*}
We conclude from \eqref{di-r} and \eqref{di-inf} that
$$\begin{aligned}
\|u\|_{{\widetilde {L^2}}([T_i,T_{i+1}];\dot B^{\frac2p  ,\frac1p  }_{p,1})}¬†+   \|u\|_{ {L^1}([T_i,T_{i+1}];\dot B^{ \frac2p    , 1+\frac1p  }_{p,1}
\cap \dot B^{1+\frac2p  , \frac1p  }_{p,1} )}
\leq (2 C)^i \Big ( \|u_0\|_{\dot B^{-1+\frac2p    , \frac1p  }_{p,1}}
+ \|F\|_{ {\mathcal X}_{p,1} } \Big )  \end{aligned}
$$
and
$$
\|u\|_{\widetilde{L^{\infty}}([T_i,T_{i+1}];\dot B^{-1+\frac2p    , \frac1p  }_{p,1}} \leq
(2C)^i \Big ( \|u_0\|_{\dot B^{-1+\frac2p    , \frac1p  }_{p,1}} + \|F\|_{ {\mathcal X}_{p,1} }  \Big)
$$
for all $i\leq k-1$. The same arguments as above also apply on the
interval $[T_k,T]$ and yield
$$
\begin{aligned}
& \|u\|_{\widetilde{L^{2}}([T_k,T ];\dot B^{\frac2p  ,\frac1p  }_{p,1})}\leq  (2 C)^N \Big ( \|u_0\|_{\dot B^{-1+\frac2p    , \frac1p  }_{p,1}}  + \|F\|_{{\mathcal X}_{p,1} } \Big ) \end{aligned}
$$
and
$$
\|u\|_{\widetilde{L^{\infty}}([T_k,T];\dot B^{-1+ \frac2p  ,\frac1p  }_{p,1})} \leq
(2C)^N\Big ( \|u_0\|_{\dot B^{-1+\frac2p    , \frac1p  }_{p,1}}  + \|F\|_{ {\mathcal X}_{p,1} }    \Big) \, . $$
Then it is easy to see that (see for instance~\cite{gip})
\begin{align*}
\|u\|_{
\widetilde{L^{2}}
([0,T];\dot B^{\frac2p    , \frac1p  }_{p,1})
}
&\leq  \|u\|_{\widetilde{L^{2}}([T_1,T_2];
 \dot B^{\frac2p    , \frac1p  }_{p,1})}+\dots+\|u\|_{
 \widetilde{L^{2}}([T_k,T];\dot B^{\frac2p    , \frac1p  }_{p,1})
 }\\
&\leq N (2C)^N
  \|
  u_0
  \|_{
  \dot B^{-1+\frac2p    , \frac1p  }_{p,1}
  } + N (2C)^N
 \|F\|_{{\mathcal X}_{p,1}  }  \, .
\end{align*}
Under
assumption~(\ref{eq:initialdata}) this contradicts the maximality of~$
T$ as defined in~(\ref{eq:defT}). Since the integer $N$ can be chosen of size equivalent to $\|U\|_{{  {L^2}}(\R^+;\dot B^{\frac2p  ,\frac1p  }_{p,1})} +  \|U\|_{L^1(\R^+; \dot B^{ 1+\frac2p  ,\frac1p  }_{p,1} \cap  \dot B^{ \frac2p  ,1+\frac1p  }_{p,1})} $,  the theorem is proved.
 \qed

\begin{rmk}{\rm
 Note that we have    obtained also that~$u $ belongs to~$ \widetilde{L^{\infty}}(\R^+,\dot B^{-1+ \frac2p  ,\frac1p  }_{p,1})$.}
\end{rmk}

\subsection{Proof of Corollary~\ref{strongstability}}
Let~$u \in \widetilde{L^\infty_{loc}}(\R^+; {\mathcal B}^{1}_{1} )  \cap  {L^1_{loc}}
         (\R^+;\dot B^{3  ,1  }_{1,1} \cap \dot B^{ 1  ,3  }_{1,1} ) $ solve~(NS) with initial data~$u_0 \in  {\mathcal B}^{1}_{1} $. Let us start by proving that~$u \in {\mathcal S}_{1,1}$ and that~$\|u(t)\|_{\ {\mathcal B}^{1}_{1}} \to 0$ as~$t$ goes to~$\infty$.  Actually it is enough to prove the convergence to zero result in large times, since the fact that~$u \in {\mathcal S}_{1,1}$ is then a consequence of Theorem~\ref{globalaniso} since for~$T$ large enough we have~$\|u(T)\|_{{\mathcal B}^{1}_{1}} \leq c_0$.

\smallskip
\noindent
We shall only sketch the proof as it is very similar to the same result in the isotropic case, proved in~\cite{gip}. 
                 The idea is to use a frequency truncation to decompose~$u_0 = v_0 + w_0$ with~$\|w_0\|_{ {\mathcal B}^{1}_{1} } \leq \e_0$ for some arbitrarily small~$\e_0$ and with~$v_0 \in  {\mathcal B}^{1}_{1}\cap L^2$. We then solve globally (NS) in~$ {\mathcal S}_{1,1}$ with data~$w_0$, and we know from Theorem~\ref{globalaniso} that
         $$
         \|w\|_{ {\mathcal S}_{1,1}} \leq 2 \e_0\, .
         $$
  It is easy to see (using the same arguments as in Proposition~\ref{propanisoiso}) that~$\|w_0\|_{\dot B^{-1}_{\infty,\infty}} \lesssim \e_0$ so the arguments of Proposition A.2 of~\cite{gip} imply that
 \begin{equation}\label{winfty}
         \sup_{t>0}\sqrt t \|w(t)\|_{L^\infty}\lesssim \e_0 \, .
\end{equation}
      Now let us consider~$v$: it satisfies the perturbed (NSP) equation with data~$v_0$, with~$F = 0$ and with~$U = w$, and it belongs to~$ \widetilde{L^\infty_{loc}}(\R^+; {\mathcal B}^{1}_{1} )  \cap  {L^1_{loc}}
         (\R^+;\dot B^{3  ,1  }_{1,1} \cap \dot B^{ 1  ,3  }_{1,1} ) $ since that holds for~$u$ and~$w$. We claim that there is~$T>0$ such that
         $$
         v \in \widetilde {L^\infty}([0,T];L^2) \cap {L}^2([0,T];\dot H^1) \, .
          $$
          Indeed we have by product laws the following
          analogue of~(\ref{estimatewithwidetilde}):
      $$  \begin{aligned}
 \| u \cdot \nabla u \|_{L^1([0,T]; \dot B^{1    , \frac12  }_{1,1}  )}
  & \leq \| u^h \cdot \nabla^h u\|_{L^1([0,T]; \dot B^{1    , \frac12  }_{1,1}  )} +
   \| u^3   \partial_3 u\|_{L^1([0,T]; \dot B^{1   , \frac12  }_{1,1}  )} \\
    & \lesssim \|u\|_{{\widetilde  {L^\infty}}([0,T];\dot B^{1  , \frac12  }_{1,1} )} \Big( \|  u\|_{L^1([0,T]; \dot B^{3  , 1 }_{1,1}   )}
   +  \|  u\|_{L^1([0,T];  \dot B^{ 2   ,2  }_{1,1} )} \Big) \, ,
\end{aligned}
$$
which implies as in~(\ref{linfty}) that
   $$
       \|B(u,u)\|_{{\widetilde {L^\infty}}([0,T];\dot B^{1    , \frac12  }_{1,1}  )} \lesssim  \|u\|_{{\widetilde {L^\infty}}([0,T];\dot B^{1    , \frac12  }_{1,1}  )} \|u\|_{L^1([0,T]; \dot B^{3   ,1}_{1,1} \cap \dot B^{ 2   ,1+ 1  }_{1,1} )}
      $$
      so as in  Lemma A.2 of~\cite{gip} we get~$v \in {\widetilde {L^\infty}}([0,T];\dot B^{1    , \frac12  }_{1,1}  ) \subset {\widetilde {L^\infty}}([0,T];\dot B^{0    , 0 }_{2,1}  ) \subset {\widetilde {L^\infty}}([0,T];L^2 ) $.
 The bound in~$ {L}^2([0,T];\dot H^1)$   is obtained in a similar way, noticing that if~$f$ is in~$L^1([0,T]; \dot B^{1    , \frac12  }_{1,1}  )$, then 
 $\displaystyle
 F := \int_0^t e^{(t-t') \Delta} {\mathbb P} (f)(t') \, dt'
 $
 satisfies, by similar computations to the proof of Theorem~\ref{globalaniso},
 $$
 \|F\|_{\widetilde {L^2} ([0,T]; \dot B^1_{2,1}(\R^3))} \lesssim \|f\|_{L^1([0,T]; \dot B^{0    , 0  }_{2,1}  )} \lesssim \|f\|_{L^1([0,T]; \dot B^{1    , \frac12  }_{1,1}  )} \, .
 $$
  Then we conclude exactly as in the proof of Theorem 2.1 in~\cite{gip}: we find, writing an energy estimate in~$L^2$ and using~(\ref{winfty}) that~$v$ can be made arbitrarily small in~$\dot H^\frac12$ as time goes to infinity, 
hence by Proposition~\ref{propanisoiso} the same holds in~$\dot B^{0,\frac 12}_{2,1}$. It follows that  $u(t) = v(t) + w(t)$ is arbitrarily small in~$\dot B^{0,\frac 12}_{2,1}$  (say smaller than~$c_0$, if~$  \e_0  $ is small enough) for~$t$ large enough, hence there is a global solution in~${\mathcal S}_{2,1}$ associated with~$u_0$, which can be shown to also belong to ~${\mathcal S}_{1,1}$ by a propagation of regularity argument. We know indeed by Theorem~\ref{globalaniso} that~$u$ belongs to~${\mathcal S}_{1,1}(T)$ for some time~$T$ so we just need to check that the~${\widetilde {L^2}}([0,T];\dot B^{2  ,1 }_{1,1}) $ norm of~$u$ remains bounded uniformly in~$T$. But
  product laws give
$$
\|u\otimes u\|_{L^1([0,T]; \dot B^{2,1}_{1,1})} \lesssim \|u\|_{\widetilde{L^2}([0,T]; \dot B^{2,1}_{1,1})}
 \|u\|_{\widetilde{L^2}([0,T]; \dot B^{1,\frac12}_{2,1})}
$$
 so as in~(\ref{estimateBuu}) we get
 $$
 \|B(u,u) \|_{{\widetilde {L^2}}([0,T];\dot B^{2  ,1 }_{1,1}) } \lesssim \|u\|_{\widetilde{L^2}([0,T]; \dot B^{2,1}_{1,1})}
 \|u\|_{\widetilde{L^2}([0,T]; \dot B^{1,\frac12}_{2,1})} \, ,
 $$
 which allows to prove the result.
 
 \medskip
 \noindent
 Then the strong stability result is obtained using Theorem~\ref{globalanisoperturbed}. Indeed we can solve (NS) with initial data~$v_0$ for a short time and the solution~$v$ can be written as~$u-w$. The vector field~$w$ then satisfies~(PNS) with initial data~$w_0$, with  forcing term zero, and with~$U = u$. We know that~$u \in {\mathcal S}_{1,1} \subset {\mathcal Y}_{1,1}$   so the result is a direct consequence of Theorem~\ref{globalanisoperturbed}. 
 
 \medskip
 \noindent
 Corollary~\ref{strongstability} is proved. \qed

 \section{Anisotropic Littlewood-Paley decomposition}\label{appendixlp}
In this section we recall the definition of the isotropic and anisotropic Littlewood-Paley decompositions and associated function spaces, and give their main properties that are used in this paper.
We refer for instance to~\cite{BCD}, \cite{cheminzhang},  \cite{gz}, \cite{ghz}, \cite{dragos},  \cite{paicu} and \cite{triebel}  for all necessary details.

\subsection{Isotropic decomposition and function spaces}
Let~$\widehat{\chi}$ (the Fourier transform of $\chi$) be a radial
function in~$\mathcal{D}(\mathbb{R})$ such
that~$\widehat\chi(t) = 1$ for~$|t|\leq 1$
and~$\widehat\chi(t) = 0$ for~$|t|>2$, and we define (in~$d$ space dimensions)~$
\chi_{\ell} := 2^{d \ell}\chi(2^{\ell}|\cdot|).$ Then the frequency
localization operators used in this paper are defined by
$$
S_{ \ell} := \chi_{\ell}\ast\cdot \quad\mbox{and}\quad \Delta_{\ell}
:= S_{\ell +1} - S_{\ell} =:  \Psi_{\ell}\ast\cdot\, .
$$

\medskip
\noindent
Now let us define Besov spaces on~$\R^d$ using this decomposition. We start by defining, as in~\cite{BCD},
\begin{equation}\label{defS'h}
{\mathcal S}'_h :=  \left\{f \in {\mathcal S}' (\R^d)\: / \:   \|\Delta_j f\|_{L^\infty} \to 0, \: j \to -\infty \right\} \, .
\end{equation}
Let~$f$ be in~$\mathcal{S}'(\mathbb{
  R}^{d})$, let~$p $ belong to~$[1,\infty]$ and~$q$ to~$]0,\infty]$, and let~$s \in \R, s < d/p$. We say that~$f$ belongs to~$\dot
B^{s}_{p,q}(\R^d)$ if the sequence $\e_{\ell} := 2^{\ell s}\| \Delta_{\ell} f\|_{L^{p}}$
  belongs to $\ell ^{q} (\ZZ)$, and we have
  $$
  \|f\|_{\dot
B^{s}_{p,q}(\R^d)} := \|\e_{\ell} \|_{\ell^q (\ZZ)}\, .
  $$
%  Equivalent norm for the Besov space $\dot B^{s}_{p,q}(\R^d)$  can as well  be obtained with the
%discrete frequencies $2^\ell$ replaced by a continuous
%one: introducing for any $\lambda>0$ the operators
%$$
%\Delta_\lambda f = \lambda^d\Psi(\lambda\cdot) * f,
%$$
%we then find that $\|f\|_{\dot B^s_{p,q}}$
%is equivalent to
%$$ \left(\int^\infty_0 \lambda^{s\,q}\|\Delta_\lambda f\|^q_{L^p} \frac {d\lambda} \lambda\right)^{\frac 1 q}\cdot$$ 
  If~$s = d/p $ and~$q = 1$, then the same definition holds as soon as one assumes moreover that~$f \in {\mathcal S}'_h $ --- or equivalently after taking the quotient with polynomials. Finally in all other cases then~$\dot
B^{s}_{p,q}(\R^d)$ is defined by the above norm, after taking the quotient with polynomials (see~\cite{bourdaud} and the references therein for a discussion).

\medskip
\noindent It is  well-known that an equivalent norm is given by
\begin{equation}\label{defbesovheat}
\forall s\in \R, \:\forall (p,q) \in [1,\infty], \quad \|f\|_{\dot B^{s}_{p,q}(\R^d)} = \left\| t^{-\frac s2 }\|K(t) f\|_{L^p(\R^d)} \right\|_{L^q(\R^+;\frac{dt}t)}
\end{equation}
with~$K(t) := t \partial_t e^{t\Delta}$.
We recall also that Sobolev spaces are defined by the norm $\|\cdot\|_{\dot B^{s}_{2,2}(\R^d)}$
and
$$
\forall s < \frac d 2, \quad \|f\|_{\dot H^s(\R^d)} := \left(\int |\xi|^{2s} |\hat f (\xi)|^2 \, d \xi\right)^\frac12
$$
where~$\hat f$ is the Fourier transform of~$f$.

\medskip
\noindent Finally it is useful, in the context of the Navier-Stokes equations, to introduce the following space-time norms (see~\cite{cheminlerner}):
$$
 \|f\|_{\widetilde{L^r}([0,T];\dot B^{s}_{p,q})} := \big\| 2^{js} \|\Delta_j f \|_{L^r([0,T];L^p(\R^d))}\big\|_{\ell^q}
$$
or equivalently
$$
 \|f\|_{\widetilde{L^r}([0,T];\dot B^{s}_{p,q})} =  \left\| t^{-\frac s2} \|K(t) f\|_{L^r([0,T];L^p(\R^d))} \right\|_{L^q(\R^+;\frac{dt}t)} \, .
$$
The following proposition lists a few useful inequalities related to those spaces.
\begin{prop}\label{propiso}
If~$1 \leq p \leq q \leq \infty$, then
$$
\begin{aligned}
 \|\partial^\alpha \Delta_j f\|_{L^q(\R^d)} & \lesssim 2^{j(|\alpha|+d(1/p - 1/q))} \| \Delta_j f\|_{L^p(\R^d)}\, ,
\\
\mbox{and} \quad  \|e^{t\Delta} \Delta_j f\|_{L^q(\R^d)} & \lesssim e^{-ct 2^{2j}} \| \Delta_j f\|_{L^q(\R^d)} \, .
\end{aligned}
$$
\end{prop}
\noindent Finally let us recall product laws in   Besov spaces:
$$
\|fg\|_{\dot  B^{s_1+s_2 - \frac dp}_{p,q}(\R^d)} \lesssim \|f\|_{\dot  B^{s_1}_{p,q}(\R^d)} \|g\|_{\dot  B^{s_2  }_{p,q}(\R^d)} \, ,
$$
as soon as
$$
  s_1+s_2 >0 \quad \mbox{and} \quad s_j < \frac dp\, , \: j \in \{1,2\}\, .
$$

\subsection{Anisotropic decomposition and function spaces}
Similarly we define a three dimensional, anisotropic decomposition as follows. For~$(j,k) \in \ZZ^2$, we define the horizontal decomposition as
$$
S_k^h f := {\mathcal F}^{-1} \big(
 \widehat {\chi} (2^{-k}|\xi_h|) \hat f(\xi)
\big) \,  \mbox{and} \,  \Delta_k^h :=
 S_{k+1}^h  - S_{k}^h \, , \,  \mbox{which writes} \quad
 {\mathcal F} ( \Delta_k^h f):= \widehat\Psi (2^{-k}|\xi_h|) \hat f(\xi)
 $$
and the vertical decomposition as
$$
S_j^v f := {\mathcal F}^{-1} \big(
 \widehat {\chi} (2^{-j}|\xi_3|) \hat f(\xi)
\big)  \, \mbox{and}  \, \Delta_j^v :=
 S_{j +1}^v - S_{j}^v\, , \,  \mbox{which writes} \quad
 {\mathcal F} ( \Delta_j^v f):=\widehat \Psi (2^{-j}|\xi_3|) \hat f(\xi)
 \, .
$$

  \medskip
\noindent
Now let us define anisotropic Besov spaces.  We   define, for all~$(s,s')\in \R^2, s < 2/p, s' < 1/p$  and all~$p \in  [1,\infty]$ and~$q \in ]0, \infty]$,
  $$
 \dot B^{s,s'}_{p,q} :=  \left\{f \in {\mathcal S}' \: / \: \|f\|_{ \dot B^{s,s' } _{p,q}}:= \Big\| 2^{ks + js'}\|\Delta_k^h\Delta_j^v f\|_{L^p} \Big\|_{\ell^q}< \infty \right\}.
 $$
  In all other cases one defines the same norm, and one needs to take the quotient with polynomials.
  
    \medskip
\noindent
%As with  the isotropic framework, equivalent norm for the Besov space $ \dot B^{s,s'}_{p,q}(\R^3)$  can  be derived by  continuous   frequencies  instead  of the
%discrete 
%ones: defining for any positive $\lambda $ and $\mu $ the operators
%$$
%\Delta^h_\lambda \Delta^v_\mu f  = \mu \,\lambda^2 \int_{\R^3}\Psi(\lambda \, y_h)  \Psi(\mu \, y_3) \, f (\cdot - y_h, \cdot -y_3)  \,dy ,
%$$
%we get the following equivalent norm for   $ \dot B^{s,s'}_{p,q}(\R^3)$  using a continuous Littlewood-Paley decomposition: 
%$$ \left(\int^\infty_0 \int^\infty_0\lambda^{s\,q} \,  \mu^{s' q} \|\Delta^h_\lambda \Delta^v_\mu f\|^q_{L^p} \frac {d\lambda} \lambda \, \frac {d\mu} \mu\right)^{\frac 1 q}\cdot $$
%In particular, the  norm \eqref{firstnorm} is equivalent to 
%\begin{equation}\label{defbesovanisocont} \left(\int^\infty_0 \int^\infty_0\lambda^{q} \,  \mu^{q} \|\Delta^h_\lambda \Delta^v_\mu f\|^q_{L^1} \frac {d\lambda} \lambda \, \frac {d\mu} \mu\right)^{\frac 1 q}\cdot \end{equation} 
 As in~{\rm (\ref{defbesovheat})} an equivalent definition using the heat flow is
\begin{equation}\label{defbesovanisoheat}
 \|f\|_{\dot B^{s,s'}_{p,q} } = \left\| t^{-\frac{s}2}t'^{-\frac{s'}2} \| K_h(t)K_v(t') f\|_{L^p} \right\|_{   L^q(\R^+\times \R^+;\frac{dt}{t} \frac{dt'}{t'})}
 \end{equation}
 where~$K_h(t) : = t \partial_t e^{t\Delta_h^2}$ and~$K_v(t) := t \partial_t e^{t\partial_3^2}$.

 \medskip
 \noindent
 As in the isotropic case we introduce the following space-time norms:
$$
 \|f\|_{\widetilde{L^r}([0,T];\dot B^{s,s'}_{p,q} )} := \big\|2^{ks + js'}\|\Delta_k^h\Delta_j^v f\|_{L^r([0,T];L^p)}\big\|_{\ell^q}
$$
or equivalently
$$
 \|f\|_{\widetilde{L^r}([0,T];\dot B^{s,s'}_{p,q} )} =  \left\|  t^{-\frac{s}2}t'^{-\frac{s'}2} \| K_h(t)K_v(t') f\|_{L^r([0,T];L^p)} \right\|_{L^q(\R^+\times \R^+;\frac{dt}{t} \frac{dt'}{t'})} \, .
$$
Notice that of course~$\widetilde{L^r}([0,T];\dot B^{s,s'}_{p,r} ) =  {L^r}([0,T];\dot B^{s,s'}_{p,r} )$, and by Minkowski's inequality, we have the embedding~$\widetilde{L^r}([0,T];\dot B^{s,s'}_{p,q} ) \subset  {L^r}([0,T];\dot B^{s,s'}_{p,q} )$ if~$r \geq q$.

 \medskip
\noindent The anisotropic counterpart of Proposition~{\rm \ref{propiso}} is the following.
 \begin{prop}\label{propaniso}
If~$1 \leq p_1 \leq p_2 \leq \infty$, then
$$
\begin{aligned}
 \|\partial_{x_h}^\alpha \Delta_k^h  f\|_{L^{p_2}(\R^2; L^r(\R))} &\lesssim 2^{k(|\alpha|+2(1/p_1 - 1/p_2))}
 \| \Delta_k^h f\|_{L^{p_1}(\R^2; L^r(\R))}\, ,
 \\
 \|\partial_{x_3}^\alpha \Delta_j^v   f\|_{L^r(\R^2; L^{p_2}(\R))} & \lesssim 2^{j(|\alpha|+1/{p_1}- 1/{p_2})} \| \Delta_j ^vf\|_{L^r(\R^2; L^{p_1}(\R))} \, ,
\\
 \|e^{t\Delta}  \Delta_k^h\Delta_j^v f\|_{L^q} &\lesssim e^{-ct (2^{2k} + 2^{2j})} \|  \Delta_k^h\Delta_j^v f\|_{L^q} \, .
\end{aligned}
$$
\end{prop}
\noindent In this paper we use product laws in anisotropic Besov spaces, which read as follows:
$$
\|fg\|_{\dot  B^{s_1+s_2 - \frac2p,s'_2}_{p,q}} \lesssim \|f\|_{\dot  B^{s_1 ,\frac1p}_{p,1}} \|g\|_{\dot  B^{s_2  ,s'_2}_{p,q}} +\|f\|_{\dot  B^{s_1  ,s'_2}_{p,q}}  \|g\|_{\dot  B^{s_2 ,\frac1p}_{p,1}} \, ,
$$
as soon as
$$
  \frac 1p  \leq s'_2, \quad  s_1+s_2 >0 \quad \mbox{and} \quad s_j < \frac 2p\, , \: j \in \{1,2\}\, ,
$$
and
$$
\|fg\|_{\dot  B^{s_1+s_2 - \frac2p,s'_1+s'_1-\frac1p}_{p,q}} \lesssim \|f\|_{\dot  B^{s_1 ,s'_1}_{p,q}} \|g\|_{\dot  B^{s_2  ,s'_2}_{p,q}} \, ,
$$
as soon as
$$
  s'_1+s'_2 >0 \quad \mbox{and} \quad s'_j < \frac 1p\, , \: j \in \{1,2\}\,
$$
and with the same conditions on~$s_1,s_2$.
Finally
\begin{equation}\label{algebra}
\|fg\|_{\dot  B^{ \frac2p, \frac1p}_{p,1}} \lesssim \|f\|_{\dot  B^{ \frac2p, \frac1p}_{p,1}} \|g\|_{\dot  B^{ \frac2p, \frac1p}_{p,1}} \,,
\end{equation}
and if~$p <4$,
\begin{equation}\label{quasialgebra}
\|fg\|_{\dot  B^{ -1+\frac2p, \frac1p}_{p,1}} \lesssim \|f\|_{\dot  B^{ -1+\frac2p, \frac1p}_{p,1}} \|g\|_{\dot  B^{ \frac2p, \frac1p}_{p,1}} \,.\end{equation}

\medskip
\noindent
The following result     compares some isotropic and anisotropic Besov spaces.
 \begin{prop}\label{propanisoiso}
Let~$s$ and~$t$ be two nonnegative real numbers. Then for any~$(p,q) \in [1,\infty]^2 $ one has
$$
\|f\|_{\dot B^{s,t}_{p,q}} \lesssim \|f\|_{\dot B^{s+t}_{p,q}} \, .
$$
\end{prop}
\begin{proof} [Proof of Proposition~{\rm \ref{propanisoiso}}]
We recall that
$$
\|f\|_{\dot B^{s,t}_{p,q}}^q = \sum_{j,k} 2^{ksq} 2^{jtq} \| \Delta_k^h\Delta_j^v f\|_{L^p}^q \, .
$$
We separate the sum into two parts, depending on whether~$j < k$ or~$j \geq k$ and we shall only detail the first case (the second one is identical). We  notice indeed that if~$j<k$, then
$$
\begin{aligned}
 \| \Delta_k^h\Delta_j^v f \|_{L^p}&=\| \sum_\ell   \Delta_\ell \Delta_k^h\Delta_j^v f \|_{L^p}\\
 &\sim \|    \Delta_k \Delta_k^h\Delta_j^v f \|_{L^p}\, .
\end{aligned}
$$
It follows that
$$
\begin{aligned}
\sum_{j < k} 2^{ksq} 2^{jtq} \| \Delta_k^h\Delta_j^v f\|_{L^p}^q &\lesssim  \sum_{j < k} 2^{ksq} 2^{jtq} \| \Delta_k   f\|_{L^p}^q \\
&\lesssim  \sum_{k} 2^{k(s+t)q} \| \Delta_k   f\|_{L^p}^q
\end{aligned}
$$
and the result follows.
\end{proof}
\noindent Finally let us prove the following easy lemma, which implies 
   that~$(u_{0,n})_{n \in \N}$ is bounded in~${\mathcal B}^1_q$ if  it is bounded in a space of the type~$\dot B^{1\pm \e_1,1\pm\e_2}_{1,1}$ for some~$\e_1,\e_2 >0$.
\begin{lem}\label{epsilon-epsilonq}
Let~$s_1,s_2 \in \R$, $p \in [1,\infty]$, $0 < q_1 \leq q_2 \leq\infty$ be given, as well as two positive real numbers~$\varepsilon_1$ and~$\varepsilon_2$. The space~$\dot B^{s_1\pm \e_1,s_2\pm\e_2}_{1,q_2}$ is continuously embedded in~$\dot B^{s_1,s_2}_{p,q_1}$.
\end{lem}
 \begin{proof}
 Let~$f$ be an element of~$\dot B^{s_1\pm \e_1,s_2\pm\e_2}_{1,q_2}$ and let us prove that~$f$ belongs to~$\dot B^{s_1,s_2}_{p,q_1}$.
We write
$$
\|f\|_{\dot B^{s_1,s_2}_{p,q_2}}^{q_1} = \sum_{j,k} 2^{ks_1q_1} 2^{js_2q_1}  \| \Delta_k^h\Delta_j^v f\|_{L^p}^{q_1} 
$$
and
we decompose the sum into four terms, depending on the sign of~$j$ and~$k$. For instance we have
$$
\begin{aligned}
{\mathcal F}_1& :=  \sum_{j \leq 0 \atop k\geq 0} 2^{ks_1q_1} 2^{js_2q_1}  \| \Delta_k^h\Delta_j^v f\|_{L^p}^{q_1} \\
& \leq  \sum_{j \leq 0 \atop k\geq 0} 2^{-k   \e_1 q_1} 2^{j \e_2 q_1}  2^{k (s_1 + \e_1) q_1} 2^{j(s_2-\e_2)q_1}  \| \Delta_k^h\Delta_j^v f\|_{L^p}^{q_1}
\end{aligned}
$$
and we apply H\"older's inequality for sequences which gives
$$
{\mathcal F}_1 \lesssim \|f\|_{\dot B^{s_1+ \e_1,s_2-\e_2}_{1,q_2}} \, .
$$
The other terms are dealt with similarly.
\end{proof}
\subsection{On the role of anisotropy in the Navier-Stokes equations}
In this final short paragraph, we shall prove Theorem~\ref{anisominimal} stated in the introduction.

\begin{proof} [Proof of Theorem~\ref{anisominimal}]The proof   follows
from the small data theory  recalled in Appendix~\ref{globalsmallaniso}.
Let us first consider~$\displaystyle v_0 := \sum_{j-k < - N_0} \Delta_k^h\Delta_j^v u_0$.
We have
   $$
  \begin{aligned}
  \|v_0\|_{\dot B^ {0, \frac12}_{2,1}}& \sim \sum_{j-k < - N_0} 2^\frac j2 \|\Delta_k^h \Delta_j^v u_0\|_{L^2(\R^3)} \\
   & \sim    \sum_{j-k < - N_0} 2^{\frac{j-k}2}  2^\frac k2
   \|\Delta_k^h \Delta_j^v u_0\|_{L^2(\R^3)}
    \leq C 2^{-\frac{N_0}2}\rho
      \end{aligned}
  $$
due to Proposition~\ref{propanisoiso} which states in particular that~$\dot B^\frac12_{2,1} \subset \dot B^{\frac12, 0}_{2,1}$. So~$v_0$ can be made arbitrarily small in~$\dot B^ {0, \frac12}_{2,1}$, for~$N_0$ large enough (depending only on~$\rho$).

\noindent Now  let us consider~$\displaystyle w_0 = \sum_{j-k > N_0} \Delta_k^h\Delta_j^v u_0$. We shall prove that in this case~$\|w_0\|_{L^3}$ is small. Indeed we know (see for instance~\cite{BCD}) that~$\dot B^0_{3,1} \subset L^3$, and moreover   we have as soon as~$N_0$ is large enough (depending only on the choice of the Littlewood-Paley decomposition)
$$
\|\Delta_\ell w_0\|_{L^3} \sim \big\| \sum_{k-\ell < -N_0} \Delta_k^h \Delta_\ell^v u_0 \big \|_{L^3}\, .
$$
It follows that
$$
  \begin{aligned}
\|\Delta_\ell w_0\|_{L^3}&  \leq \sum_{k-\ell < -N_0} \|\Delta_k^h \Delta_\ell^v u_0\|_{L^3} \\
&  \leq C  \sum_{k-\ell < -N_0} 2^{\frac {k}3} 2^{\frac {\ell}6}  \|\Delta_k^h \Delta_\ell^v u_0\|_{L^2}
     \end{aligned}
  $$
by Bernstein's inequalities (see Proposition~\ref{propaniso}, applying successively the inequalities for the horizontal and   the vertical truncations). So using Proposition~\ref{propanisoiso} again which states in particular that~$\dot B^\frac12_{2,1} \subset \dot B^{0, \frac12}_{2,1}$, we get
$$
\|\Delta_\ell w_0\|_{L^3} \leq  C  \sum_{k-\ell < -N_0} 2^{\frac {k-\ell}3}
2^{\frac \ell2}    \|\Delta_k^h \Delta_\ell^v u_0\|_{L^2 } \leq C 2^{-\frac {N_0}3}\rho c_\ell \, ,
$$
where~$c_\ell $ is a sequence in the unit ball of~$\ell^1(\ZZ)$.
   So again if~$N_0$ is large enough (depending only on~$\rho$) then   we find that~$w_0$ is small in~$\dot B^0_{3,1}$ hence in~$L^3$.

\medskip
\noindent
To conclude we can start by solving (NS) associated with the data~$w_0$ which yields a global, unique solution~$w$ that by Proposition~\ref{propanisoiso} belongs to~${\mathcal Y}_{3,1}$, with norm smaller than~$2 \|w_0\|_{L^3}$ (by small data theory, as soon as~$N_0$ is large enough). Then since~$\dot B^ {0, \frac12}_{2,1}$ embeds in~$\dot B^{-\frac13, \frac13}_{3,1}$ we can apply Theorem~\ref{globalanisoperturbed} with~$F = 0$ and~$U=w$ which   solves the perturbed equation satisfied by~$u-w$ globally in time, as soon as~$N_0$
again is large enough.
The solution  belongs to~${\mathcal C}(\R^+;L^3(\R^3))$  by classical propagation of regularity arguments, and that proves the theorem.
\end{proof}
\begin{rmk}
{\rm
 Contrary to Theorem~\ref{mainresult}, the proof of Theorem~\ref{anisominimal} does not require the special structure of the nonlinear term in (NS) as it reduces to checking that the initial data is small in an adequate scale-invariant space.
}
\end{rmk}
 
 \bigskip

 \noindent{\bf Acknowledgements. } $ $ We are very grateful to Marius Paicu for  interesting
 discussions around the questions dealt with in this paper, and for many comments on the manuscript. We extend our thanks to the anonymous referee for a careful reading of the manuscript
and fruitful remarks, in particular
for suggesting the example discussed before~(\ref{defphinisotropic}). Finally we thank very warmly Vladimir \c{S}ver\'ak for helpful suggestions to improve the presentation of the text.

\bigskip

\bigskip

\bigskip

%%%%%%%%%%%%%%%%%%%%%%%%%%%%%%%%%%%%%%%%%%%%%

\end{document}